\documentclass[10pt,reqno]{amsart} 

\usepackage{amsmath} 
\usepackage{amssymb} 
\usepackage{amsthm} 
\usepackage{mathtools}
\usepackage{comment}
\usepackage{tikz}
\usetikzlibrary{cd}
\usepackage[colorlinks]{hyperref}
\usepackage{float}
\usepackage{enumitem}

\usepackage{tablefootnote}
\usepackage{array}

\newcommand*{\horzbar}{\rule[.5ex]{2.5ex}{0.5pt}}

\usepackage[normalem]{ulem}

\newcommand{\Franky}[1]{{\color{blue}\sf  Franky: [#1]}}
\newcommand{\Thomas}[1]{{\color{red}\sf  Thomas: [#1]}}

\numberwithin{equation}{section}

\newcommand{\sub}{\subseteq}
\newcommand{\Z}{\mathbb{Z}}
\newcommand{\R}{\mathbb{R}}

\newcommand{\N}{\mathbb{N}}
\newcommand{\eps}{\varepsilon}
\newcommand{\M}{\mathcal M}
\newcommand{\rank}{\mathrm{rank}}

\newcommand{\dist}{\mathrm{dist}}

\numberwithin{chap}{section}
\newtheorem{thm}{Theorem}
\numberwithin{thm}{section}

\newtheorem{prop}[thm]{Proposition}
\newtheorem{defn}[thm]{Definition}
\newtheorem{lem}[thm]{Lemma}

\newtheorem{cor}[thm]{Corollary}

\newtheorem{condition}[thm]{Condition}

\DeclarePairedDelimiter{\norm}{\lVert}{\rVert}

\makeatletter
\let\oldnorm\norm
\def\norm{\@ifstar{\oldnorm}{\oldnorm*}}

\makeatother

\begin{document}

\pagestyle{myheadings} \thispagestyle{empty} \markright{}
\title{Two principles of decoupling}

\author{Jianhui Li and Tongou Yang}
\address[Jianhui Li]{Department of Mathematics, Northwestern University\\
Evanston, IL 60208, United States
}
\email{jianhui.li@northwestern.edu}

\address[Tongou Yang]{Department of Mathematics, University of California\\
Los Angeles, CA 90095, United States}
\email{tongouyang@math.ucla.edu\\ 
tomyangcuhk@gmail.com}

\date{}

\begin{abstract}
    We put forward a radial principle and a degeneracy locating principle of decoupling. The former generalises the Pramanik-Seeger argument used in the proof of decoupling for the light cone. The latter locates the degenerate part of a manifold and effectively reduces the decoupling problem to two extremes: the nondegenerate case and the totally degenerate case. Both principles aim to provide a new algebraic approach to reducing decoupling for new manifolds to decoupling for known manifolds.

    \end{abstract}

\maketitle

\section{Introduction}\label{sec:introduction}

Fourier decoupling was first introduced by Wolff in \cite{Wolff2000} as a tool to prove $L^p$ local smoothing estimates for large $p$. In \cite{Wolff2000}, the decoupling was formulated in the setting of a truncated light cone in $\R^3$. Later, observed by Pramanik and Seeger \cite{PS2007}, the decoupling inequalities for light cones in $\R^3$ can be reduced to that of circles in $\R^2$. The technique is now known as the Pramanik-Seeger iteration. In simple terms, a light cone can be approximated by cylinders at intermediate scales. By projection and induction on scales, decoupling for a light cone in $\R^n$ is then reduced to decoupling for the unit sphere in $\R^{n-1}$. Therefore, together with the seminal work of Bourgain and Demeter \cite{BD2015} on the sharp decoupling inequalities for the unit sphere (and a compact piece of an elliptic paraboloid), satisfactory decoupling results for light cones are obtained.

Based upon decoupling for spheres and light cones, decoupling inequalities for general manifolds have also been largely studied, and we consider two main types of such manifolds. The first type of manifolds carry certain nondegenerate conditions. See Table \ref{tab:nondegenerate} below for some important examples, and also the following list of decoupling for nondegenerate manifolds in $\R^n$: \cite{BD2surfaceR4}\cite{Oh3surfacesinR5}\cite{DGS2019}\cite{GOZZ}\cite{GZ2019Inventiones}. The nondegeneracy conditions in the latter list are slightly harder to state, and we encourage the reader to refer to their papers for details. In general, these results rely on the corresponding multilinear decoupling inequalities derived from the nondegeneracy of the manifolds. 

\begin{table}[h!]
    \centering
    \begin{tabular}{|l|l|l|l|l|}
    \hline Results & $n$ & $k$ & non-degeneracy & critial exponents \\
    \hline    Bourgain-Demeter  & $n$ & $n-1$ & $D^2 \phi(x) \succ 0$  & $\ell^2(L^{\frac{2(n+1)}{n-1}})$\\
    \cite{BD2015} & & & &\\
    \hline    Bourgain-Demeter  & $n$ & $n-1$ & $\det D^2 \phi(x) \neq 0$ & $\ell^{\frac{2(n+1)}{n-1}}(L^{\frac{2(n+1)}{n-1}})$ \\
    \cite{BD2017}  & & & &\\
    \hline    Bourgain-Demeter- & $n$ & $1$ & $\det \left(\phi'',\dots,\phi^{(n)}\right) $ & $\ell^2(L^{n(n+1)})$ \\
    Guth \cite{BDG2016} & & &  $\neq 0$ &\\
    \hline Guth-Maldague- & $3$ & $2$ & $\det D^2 \phi(x) \neq 0$ & $\ell^2(L^4)$\\
    Oh \cite{GMO24} & & & &\\
    \hline
    \end{tabular}
    \caption{Decoupling inequalities for nondegenerate manifolds in $\R^n$ as graphs of $\phi : [-1,1]^{k} \to \R^{n-k}$. }
  \end{table} \label{tab:nondegenerate}

The second type of manifolds shown in Table \ref{tab:degenerate} below are more general as the nondegenerate conditions may not be satisfied. The major approach to these results is to deal with the decoupling for pieces of the manifold on which the nondegenerate factor is small. For instance, for the case of hypersurfaces, the nondegenerate factor is the Gaussian curvature. When compared to Table \ref{tab:nondegenerate}, very limited results are known in higher dimensions $n >3$ or higher co-dimensions $n-k \geq 2$.

\begin{table}[h!]
    \centering
    \begin{tabular}{|l|l|l|}
    \hline Results & $n$  & descriptions \\
    \hline Biswas-Gilula-Li- &  & \\
    Schwend-Xi \cite{BGLSX}; & 2 &  analytic/smooth planar curves\\
    Demeter \cite{Demeter2020}; & &\\
    Y. \cite{Yang2} & &\\
    \hline Bourgain-Demeter & $3$ & some surfaces of revolution: $\phi(x) = \gamma(|x|)$ \\
    -Kemp \cite{BDK2019} & &  \\
    \hline Kemp \cite{Kemp2} & 3 & surfaces lacking planar points\\ 
    \hline Kemp \cite{Kemp1,Kemp2024} & n & tangent developable surfaces  \\ 
    \hline L.-Y. \cite{LiYang} & 3 & mixed-homogeneous polynomials $\phi$\\
    \hline L.-Y. \cite{LiYang2023}; Guth-& 3 & smooth surfaces\\
    Maldague-Oh \cite{GMO24} & &\\
    \hline Gao-Li-Zhao-Zheng   & $n$ & $\phi(x) = \sum_{i} \phi_i(x_i)$, $\phi_i'' \sim 1$ or $\phi_i (x_i) = x_i^m$. \\
    \cite{GLZZ2022} & &\\
    \hline  
    \end{tabular}
    \caption{Decoupling inequalities for curves/surfaces in $\R^n$ as graphs of $\phi : [-1,1]^{n-1} \to \R$.}
    \label{tab:degenerate}
\end{table}

In this paper, we first generalise the cone-to-sphere reduction to a \textit{radial principle} that applies to more general manifolds that exhibit some conical structure. Then, we form the \textit{degeneracy locating principle} by abstracting the idea of reducing the manifolds with possible degeneracy to the nondegenerate counterpart. These results provide an algebraic approach to reducing decoupling for more complicated manifolds to simpler manifolds. In particular, we provide immediate corollaries of these principles, namely, decoupling inequalities for new manifolds.

\subsection{Decoupling inequalities}
We first formulate decoupling inequalities in a more general setting.

\begin{defn}\label{defn:decoupling}
    Given a compact subset $S\sub \R^n$ and a finite collection $\mathcal R$ of boundedly overlapping\footnote{To deal with technicalities, we will impose a slightly stronger overlapping condition. See Definition \ref{defn:enlarged_overlap}.} parallelograms $R\sub \R^n$. For Lebesgue exponents $p,q\in [2,\infty]$ and $\alpha\in \R$, we define the $(\ell^q(L^p),\alpha)$ decoupling constant $\mathrm{Dec}(S,\mathcal R,p,q,\alpha)$ to be the smallest constant $\mathrm{Dec}$ such that
    \begin{equation}\label{eq:defn_decoupling}
        \norm{\sum_{R}f_R}_{L^p(\R^n)} \leq \mathrm{Dec}\,\, (\#\mathcal R)^{\frac{1}{2}-\frac{1}{q}+\alpha}\norm{\norm{f_R}_{L^p(\R^n)}}_{\ell^q(R\in \mathcal R)}
    \end{equation}
    for all smooth test function $f_R$ Fourier supported on $R\cap S$.
    
    Given a subset $S\sub \R^n$, we say that $S$ can be $(\ell^q(L^p),\alpha)$ decoupled into the parallelograms $R\in \mathcal R$ at the cost of $K$, if $S\sub \cup \mathcal R$ and $\mathrm{Dec}(S,\mathcal R,p,q,\alpha)\le K$.
\end{defn}

When $p,q$ are given, the exponent $\alpha$ is most commonly taken to be $0$. Hence, we also say $S$ can be $\ell^q(L^p)$ decoupled if $S$ can be $(\ell^q(L^p),0)$ decoupled. To further reduce notation, if $S$ and the exponents $p,q,\alpha$ are clear from the context, we 
simply say that {\it $S$ can be decoupled into parallelograms $R\in \mathcal{R}$}.

By H\"older's inequality, $\ell^{q}(L^p)$ decoupling implies $\ell^{q'}(L^p)$ decoupling if $2\leq q'\leq q$. In the literature, $\ell^2(L^p)$ and $\ell^p(L^p)$ decoupling inequalities are the most common.

Thus, together with Plancherel's identity and ignoring the bounded overlap, we always have $\ell^{q}(L^2)$ decoupling for $q\ge 2$. Also, for fixed $q$ and $p'\ge p$,  $\ell^{q}(L^{p'})$ decoupling is generally stronger than $\ell^{q}(L^{p})$ decoupling. This follows from interpolation whenever applicable, or can usually be observed from the proof of decoupling. 

\subsubsection{Decoupling for manifolds in $\R^n$}

To formulate decoupling for manifolds, we first introduce the following notation: for $\delta>0$ and $1\le m\le n-1$, we say that a subset $R\sub \R^n$ is {\it $\delta$-flat in dimension $m$}, if it is contained in the $\delta$-neighbourhood of some $m$-dimensional affine subspace of $\R^n$. Here and throughout this article, for $M\sub \R^n$, we define the $\delta$-neighbourhood of $M$ by
\begin{equation}\label{eqn:delta_nbhd}
    N_\delta(M):=\{x+c:x\in M,|c|\le \delta\}.
\end{equation}
In the context of decoupling, we often consider the case where $S=N_\delta(M)$ where $0<\delta \ll 1$ and $M$ is a compact manifold.\footnote{In this article, by a compact manifold we always mean a compact piece of a connected manifold, with or without boundary; we also assume every manifold is at least $C^2$.} Since $M$ is compact, by choosing $O_n(1)$ many coordinate charts, we may assume that $M$ is locally given by the graph of a $C^2$ function $\phi:\Omega\to \R^{n-k}$, where $1\le k\le n-1$ and $\Omega$ is an open subset of $\R^k$. In this case, we often replace $N_\delta(M)$ by a more convenient notion of $\delta$-neighbourhood, called the $\delta$-{\it vertical neighbourhood}, defined by
\begin{equation}\label{eqn:vertical_nbhd}
N_\delta^\phi(\omega): =\{ (x , \phi(x) + c): x \in \omega, c \in [-\delta,\delta]^{n-k}\},
\end{equation}
for subsets $\omega\sub \R^k$, usually parallelograms.

We usually want to decouple $S=N_\delta(M)$ into smaller pieces $S'$, such that the \underline{convex hull} of $S'$ is {\it equivalent to} (see Section \ref{sec:notation}) a parallelogram $R\sub \R^n$ that is $\delta$-flat in dimension $m$. In particular, if $M$ is a hypersurface, then the convex hull of $S'$ will be essentially the same as $S'$, and we only need to consider the case $m=n-1$. On the other hand, if $M$ is a curve, then the convex hull of $S'$ will be nearly equivalent to the corresponding piece of an anisotropic neighbourhood\footnote{See, for instance, \cite{shortproofmomentcurve} or \cite[Chapter 11]{Demeter2020} for definitions of anisotropic neighbourhoods of a curve.} of $M$. We might need to consider all $m\in [1,n-1]$.

\begin{defn}[$\delta$-flat subsets for graphs]\label{defn:graphically_flat}
    Let $\Omega\sub \R^k$, $\phi:\Omega\to \R^{n-k}$ and $1\le m\le n-1$. We say that a parallelogram $\omega\sub \Omega$ is $(\phi,\delta)$-flat in dimension $m$, if the graph of $\phi$ above $\omega$ is $\delta$-flat in dimension $m$. Equivalently, this means that 
    \begin{equation}\label{eqn:Oct_22}
    \sup_{x\in \omega}  |\lambda(x,\phi(x))| \leq \delta,
\end{equation}
    for some affine function $\lambda : \R^{n} \to \R^{n-m}$ given by $y \mapsto A y + b$, where $A\in \R^{(n-m) \times n}$ has orthonormal rows, and $b \in \R^{n-m}$.
    \end{defn}

We have alternative definitions of flatness for graphs that are more convenient to use in practice; see Definition \ref{defn:new_flatness} in the appendix. It is shown in Proposition \ref{prop:equivalence_flatness_general} that Definition \ref{defn:graphically_flat} is equivalent to \eqref{eqn:F1} of Definition \ref{defn:new_flatness} when the parallelogram $\omega$ has width $\gg \delta$. Here and throughout this article, the {\it width} of a parallelogram refers to its smallest dimension (see \eqref{eqn:width} for the precise definition).

Since we often deal with decoupling for graphs of functions, we introduce the following notation for simplicity. See also Lemma \ref{lem:decoupling_entire_strip}.
\begin{defn}\label{defn:graphical_decoupling_convex_hull}
    Let $\Omega\sub \R^k$, $\phi:\Omega\to \R^{n-k}$ be $C^2$. We say that a parallelogram $\Omega_0\sub \Omega$ can be $\phi$-$(\ell^q(L^p),\alpha)$ decoupled into parallelograms $\omega$ at scale $\delta$ (at cost $C$), if $N_\delta^\phi(\Omega_0)$ can be $(\ell^q(L^p),\alpha)$ decoupled into parallelograms equivalent to the convex hulls of $N_\delta^\phi(\omega)$ (at cost $C$). When the exponents $p,q,\alpha$ are fixed, we simply say that $\Omega_0$ can be $\phi$-decoupled into parallelograms $\omega$ at scale $\delta$ (at cost $C$).
\end{defn}

With the above formulation, we are ready to heuristically describe the main results of this article.
\subsection{Radial principle of decoupling}\label{sec:intro_radial}

To push the ideas of Pramanik-Seeger iteration \cite{PS2007} further, we formulate and prove a radial principle of decoupling. This is a generalisation of the arguments used in \cite[Section 8]{BD2015} (see also \cite[Appendix B]{GGGHMW2022}). Roughly speaking, under certain regularity assumptions for the (possibly multivariate) functions $r(s)$ and $\psi(t)$, the decoupling for the manifold parameterised by 
$$
(s,r(s)t,r(s)\psi(t)), \quad r(s) \sim 1
$$
can be reduced to the decoupling for
$$
(s , t,r(s)+\psi(t)).
$$

The precise statement of the radial principle of decoupling is given in Section \ref{sec:radial}. We give two examples.

\begin{enumerate}    

\item Consider the surface $M$ given by the truncated cone $\{(t,|t|) , 1\leq |t|\leq 2\} \sub \R^n$. After a finite partition and rotation, we may reparametrise each part of the cone by 
$$\{(r,rt',r\psi(t')): 1\leq r\leq 2 , t'\in \R^{n-1},|t'| \leq 1/10\},$$ where $\psi(t') = \sqrt{1-|t'|^2}$. By the radial principle, it suffices to study decoupling for the following cylinder
$$
(r,t',r+\psi(t')).
$$
This is precisely the reduction given in \cite[Section 8]{BD2015}.

\item \label{item:radial} Consider the radial surface $M$ generated by a smooth curve $\gamma:\R \to \R$, that is, $M=\{(t,\gamma(|t|)): 1 \leq |t| \leq 2\} \subseteq \R^n$ under an additional assumption that $\gamma '(r) \sim 1$. $M$ can be thought of as a (possibly curved) perturbation of the light cone. The case $n=3$ has been studied in \cite{BDK2019}. Similarly, each part of the surface can be reparametrised by $\{(\gamma(r) , rt',r\psi(t')) , 1\leq r\leq  2 , |t'|\leq 1/10\}$, where $\psi(t')=\sqrt{1-|t'|^2}$. Since $\gamma$ is invertible, we further reparametrise the surface by 
$$
(s, \gamma^{-1}(s) t' , \gamma^{-1}(s) \psi(t')), \quad \gamma^{-1}(s) \sim 1.
$$
Now, the radial principle of decoupling reduces the original decoupling problem to the decoupling for the following hypersurface:
$$
(s,  t' , \gamma^{-1}(s) + \psi(t')).
$$

\noindent As the graph of a ``separable" function $\gamma^{-1}(s)+\psi(t')$, we can use elementary arguments to study its decoupling. However, we will instead put it under a more general framework introduced right below.
\end{enumerate}

\subsection{Degeneracy locating principle of decoupling}

We make abstract the idea of reducing the manifolds with degeneracy to two extremes: the totally degenerate and the nondegenerate cases. To avoid technicalities, we present a simplified heuristic version here. See Section \ref{sec:degeneracy_locating} for the detailed formulation.

\subsubsection{Degeneracy determinant}\label{sec:intro:DD}
Let $\mathcal M$ be a family of manifolds $M_\phi$ given by the graphs of $C^2$ functions $\phi : [-1,1]^k \to \R^{n-k}$. We say that $H$ is a {\it degeneracy determinant} if it quantitatively detects the degeneracy of the manifold in the following sense: $H$ maps $M_\phi \in \mathcal{M}$ to a smooth function on $[-1,1]^k$ such that the following holds:
\begin{itemize}
    \item if $\sup_R |HM_\phi| \leq \sigma\ll 1$ for some parallelogram $R \subseteq [-1,1]^k$ , then $M_{\phi}$ restricted to $R \times \R^{n-k}$ can be rescaled to a piece of manifold that stays within the $C\sigma^\beta$ neighbourhood of some totally degenerate manifold $M_\psi \in \mathcal{M}$, i.e. $HM_\psi \equiv 0$, for some $\beta>0$ and $C\ge 1$ depending on $\mathcal M,H$.
\end{itemize}

Technically, $H$ has to satisfy other regularity conditions. See Definitions \ref{defn:degeneracy_determinant} for a precise version. 

An important example is when $\mathcal M$ consists of graphs of polynomial vectors $\phi=(\phi_i)_{i=1}^{n-k}$ of degree at most $d$ with bounded coefficients. In this case, we give a sufficient condition for $H$ to be a degeneracy determinant in Proposition \ref{prop:degeneracy_determinant_Lojasiewicz}. In particular, we prove that when $HM_\phi$ is given by an algebraic combination of (possibly higher order) derivatives of $\phi$, then $H$ satisfies the condition above. The proof involves an innovative use of the \L ojasiewicz inequality \cite{Lojasiewicz}, applied to the coefficient space of the polynomial mapping $\phi$. See Section \ref{sec:degen_approximation_theorem}.

We now give some examples of degeneracy determinants $H$. Given $k=1$ and a polynomial space curve $(t,\phi_1(t),\dots,\phi_{n-1}(t)) , t\sub I \sub [-1,1]$. The functions 
\begin{equation*}
    H_1 \phi(t):=\max_{j=1,\dots,n-1}|\phi_j''(t)|,
\end{equation*}
and the Wronskian of $\phi''$, namely, 
\begin{equation*}
    H_2\phi(t):=\det \left(\phi'',\phi^{(3)},\dots,\phi^{(n)}\right)(t),
\end{equation*}
are examples of degeneracy determinants. For $k=2$ and $n-k=1$,
$\det D^2\phi$ is a degeneracy determinant, by Corollary \ref{cor:hessian}.

\subsubsection{Statement of degeneracy locating principle}
Our degeneracy locating principle says that to study decoupling for $M_\phi \in \mathcal{M}$, it suffices to understand the decoupling of these cases at a cost uniform in $M_\phi\in\mathcal{M}$:
\begin{itemize}
    \item decoupling of the sublevel set 
    $$
    \{x \in [-1,1]^k : |HM_{\phi}(x)|\leq \sigma\},
    $$
    where $M_\phi \in \mathcal{M}$, $\sigma\leq 1$, into parallelograms.
    \item decoupling of the totally degenerate manifolds $M_\phi \in \mathcal{M}$, i.e. $HM_\phi \equiv 0$.
    \item decoupling of the nondegenerate manifolds $M_\phi \in \mathcal{M}$, i.e. $|HM_\phi| \sim 1$.
\end{itemize}

In the examples given in Table \ref{tab:degenerate}, the totally degenerate cases are the same problems of at least one dimensional lower, and the nondegenerate cases follow from the famous theorems of Bourgain, Demeter, Guth \cite{BD2015,BDG2016,BD2017} and other results mentioned in Table \ref{tab:nondegenerate}. When $k=1$ and $\mathcal{M}$ consists of univariate polynomials of degree at most $d$, the sublevel set decoupling for $\mathcal{M}$ is a direct corollary of the fundamental theorem of algebra. For special cases where $k=2$ and $n=3$ in \cite{BDK2019} and \cite{Kemp2}, the sublevel sets are neighbourhoods of circles and parabolas. In \cite{LiYang2023}, exactly the same framework is used, where the sublevel set decoupling was called ``generalised 2D uniform decoupling". In fact, the degeneracy locating principle is an abstraction and generalisation of the main idea of proof used in \cite{LiYang2023}.

\subsection{Corollaries} As an application of the two principles, we present decoupling inequalities for new types of manifolds. The sketches of proof below assume the simplified versions above to avoid technicalities.
The reader is encouraged to use the detailed and rigorous versions of the two principles in Sections \ref{sec:radial} and \ref{sec:degeneracy_locating} to complete the proof. See also the remarks right before Section \ref{sec:outline}.

\begin{cor}[Decoupling for radial surfaces]\label{cor:radial}
Let $2\le p \le\frac{2(n+1)}{n-1}$. Let $\gamma:[1,2]\to \R$ be a smooth function with $\gamma' \sim 1$\footnote{The condition $\gamma'\sim 1$ can be dropped; See \cite[Corollary 1.4]{LiYang2025}.}. Let $M$ be the hypersurface in $\R^{n}$ given by the graph of the function $\phi(t) = \gamma(|t|)$ over 
$$
\Omega = \{1\leq |t| \leq 2 , |t_1| \geq |t|/2\},\quad \quad t:=(t_1,t'),\quad t'\in \R^{n-2}.
$$
For $0<\delta\ll_{\gamma} 1$, let $\mathcal{I}_\delta$ be a partition of $[1,2]$ into $(\gamma,\delta)$-flat intervals $I$ such that $2I\cap [1,2]$ is not $(\gamma,\delta)$-flat, and that each $I$ has length at least $\delta^{1/2}$. (When $\gamma'\not\equiv 0$ and $\delta\ll_\gamma 1$, the existence of such a partition is guaranteed by, for instance, \cite[Proposition 3.1]{Yang2}.) Let $\mathcal{R}_\delta$ be a boundedly overlapping covering of $\Omega$ by parallelograms $R$, each of which is equivalent to the region
$$
\{(r\sqrt{1-|t'|^2}, rt') : r \in I, t'\in T \}
$$
for $I \in \mathcal{I}_\delta$ and some squares $T$ of side length $\delta^{1/2}$.

Then $\Omega$ can be $\phi$-$\ell^p(L^{p})$ decoupled into $(\phi,\delta)$-flat parallelograms $R \in \mathcal{R}_\delta$ at scale $\delta$ as in Definition \ref{defn:graphical_decoupling_convex_hull}, at a cost of $O_{\varepsilon,\gamma,p}(\delta^{-\varepsilon})$ for any $\varepsilon>0$. Moreover, if $ D^2\phi$ is positive-semidefinite, this $\ell^p(L^{p})$ decoupling can be improved to $\ell^2(L^{p})$ decoupling.

\end{cor}

\begin{proof}[Sketch of Proof of Corollary \ref{cor:radial} using simplified versions of two principles] 
$ $\newline
Following the arguments in Example \ref{item:radial} in Section \ref{sec:intro_radial}, the manifold $M$ can be broken into $O(1)$ many pieces, each of which can be rotated and reparameterised by
\begin{equation}\label{eq:gamma-1*psi}
    \{(s, \gamma^{-1}(s) t' , \gamma^{-1}(s) \psi(t')): |t'|\leq 1/10 , s\in \gamma([1,2]) \},
\end{equation}
where $\psi(t')=\sqrt{1-|t'|^2}$. By the radial principle of decoupling, it reduces to the decoupling of the manifold given by
\begin{equation}\label{eq:gamma-1+psi}
    \{(s,  t' , \gamma^{-1}(s) + \psi(t')): |t'|\leq 1/10 , s\in \gamma([1,2]) \}.
\end{equation}

Since $\gamma$ is smooth with $\gamma' \sim 1$ on $[1,2]$, $\gamma^{-1}$ is smooth on $\gamma([1,2])$. By the smooth approximation argument in \cite[Section 2.3]{LiYang2023}, it suffices to prove the uniform decoupling result for the family $\mathcal M$ consisting of manifolds $M_\phi$ with $\phi(s,t')  = P(s) + \psi(t')$ where $P$ is a polynomial of degree at most $d$ and with bounded coefficients. We choose $HM_\phi(s,t') := P''(s)$. By the degeneracy locating principle, it suffices to check the following. 
\begin{itemize}

    \item $H$ is a degeneracy determinant: This follows from the mean value theorem.
    \item Decoupling of sublevel sets: this follows from the fundamental theorem of algebra.
    
    \item Decoupling of totally degenerate manifolds $HM_\phi\equiv 0$: These are cylinders because $P$ must be an affine function in this case. By cylindrical decoupling, it suffices to decouple
    $$
    \{(t',\psi(t')): |t'|\leq 1/10\}
    $$
    which has positive principal curvatures everywhere. The decoupling then follows from \cite{BD2015}.
    \item Decoupling of nondegenerate manifolds $|HM_\phi|\sim 1$: In this case, $M_\phi$ has non-vanishing Gaussian curvature everywhere. The decoupling then follows from \cite{BD2017}. Furthermore, if $D^2 \phi$ is positive-semidefinite, then $P''$ has the same signs as the eigenvalues of $D^2\psi$. Then a corresponding $\ell^2(L^p)$ decoupling follows from \cite{BD2015}.
\end{itemize}
We have checked all assumptions for the degeneracy locating principle. Therefore, we have obtained the decoupling of the manifold given in \eqref{eq:gamma-1+psi}, and hence \eqref{eq:gamma-1*psi}, into $\delta$-flat parallelograms as desired.

\end{proof}

\begin{cor}[Decoupling for graphs of additively separable functions]\label{cor:dec_add} Let $2\leq p \leq \frac{2(n+1)}{n-1}$. For $j=1,\dots,J$, let $n_j\in \N$ be such that $\sum_{j=1}^Jn_j = n-1$, and let $\phi_j: [-1,1]^{n_j} \to \R$ be $C^2$ functions. Assume that for each $j$, we either have
\begin{itemize}
    \item $\inf_{[-1,1]^{n_j}}|\det D^2 \phi_j|> 0$, or
    \item $n_j\in \{1,2\}$ and $\phi_j$ is smooth.
\end{itemize}
Let $\phi:[-1,1]^{n-1} \to \R$ be such that
\begin{equation}\label{eq:cor:dec_add}
    \phi(x_1,\dots,x_J) = \sum_{j=1}^J \phi_j(x_j), \quad x_j \in [-1,1]^{n_j}.
\end{equation}
Then for any $\varepsilon>0$, $0<\delta\ll_\phi 1$, $[-1,1]^{n-1}$ can be $\phi$-$\ell^p(L^p)$ decoupled into boundedly overlapping\footnote{See Definition \ref{defn:enlarged_overlap}; here $B$ depends on $\phi,\varepsilon$ but not $\delta$.} family $\mathcal R=\mathcal R_{\varepsilon,\delta}$ of $(\phi,\delta)$-flat parallelograms at scale $\delta$ as in Definition \ref{defn:graphical_decoupling_convex_hull}, at a cost of $O_{\varepsilon,\phi,p}(\delta^{-\varepsilon})$. Moreover, each $R\in \mathcal R$ is of the form $\prod_{j=1}^J R_j$ where $R_j$ is a parallelogram contained in $[-1,1]^{n_j}$ with width $\gg_\phi \delta$. Furthermore, if $ D^2\phi$ is positive-semidefinite, this $\ell^p(L^{p})$ decoupling can be improved to $\ell^2(L^{p})$ decoupling. 
\end{cor}
This generalises the main theorem of \cite{GLZZ2022}.

\begin{proof}[Sketch of proof of Corollary \ref{cor:dec_add} using simplified degeneracy locating principle] 
$ $\newline
We only give the outline of proof assuming the case where for all $j$ we have $n_j\in \{1,2\}$ and $\phi_j$ is smooth. The proof of the general case is similar.

By the smooth approximation argument in \cite[Section 2.3]{LiYang2023}, it suffices to fix $d\in \N$ and prove such decoupling assuming each $\phi_j\in \mathcal P_{n_j,d}$ (see Section \ref{sec:notation}), with the cost of decoupling independent of the coefficients of $\phi_j$.

For each $0\le l\le J$, let $\mathcal M_l^n$ be the collection of manifolds $M_\phi$ given by graphs of polynomials of the form \eqref{eq:cor:dec_add}, such that each $\phi_j\in \mathcal P_{n_j,d}$, and $|\det D^2 \phi_j| \sim_d 1$ on $[-1,1]^{n_j}$ for $j = J-l+1,\dots,J$. 
It suffices to prove the decoupling for $M_\phi \in \mathcal{M}_{0}^n$ at cost $O_{\varepsilon,d,p}(\delta^{-\varepsilon})$.

We now induct on $n$. The base cases $\mathcal{M}_0^1, \mathcal{M}_0^2$ are proved in \cite{Yang2} and \cite{LiYang2023} respectively. Now we let $n\ge 3$ and impose the induction hypothesis that decoupling holds for $\mathcal{M}_0^{n-1}$ at cost $O_{\varepsilon,d,p}(\delta^{-\varepsilon})$.

With $n$ fixed, we then apply another induction on $l$. The base case $\mathcal{M}_{J}^n$ is given by \cite{BD2015, BD2017} at a cost of $O_{\varepsilon,d,p}(\delta^{-\varepsilon})$. Now we let $0\le l\le J-1$ and impose the induction hypothesis that decoupling holds for $\mathcal{M}_{l+1}^{n}$ at cost $O_{\varepsilon,d,p}(\delta^{-\varepsilon})$.

We apply the degeneracy locating principle on $\mathcal{M}_l^n$ with $HM_\phi: = \phi_{1}''$ when $n_1=1$ or $HM_\phi:=\det D^2 \phi_{1}$ when $n_1=2$. It suffices to prove the harder case $n_1=2$.

\begin{itemize}
    \item $H$ is a degeneracy determinant: this follows from Proposition \ref{prop:degeneracy_determinant_Lojasiewicz}. (See also Section \ref{sec:intro:DD}.)
    \item Decoupling of sublevel sets: this follows from Theorem \ref{thm:2D_general_uniform_IJ}. 
    \item Decoupling of totally degenerate manifolds $HM_\phi\equiv 0$: $\det D^2 \phi_{1}\equiv 0$ means that by a rotation on $\R^2$ (and assuming $\phi$ has no affine terms), we can assume $\phi_{1}\in \mathcal P_{1,d}$. Then its decoupling follows from the induction hypothesis on $\mathcal M^{n-1}_0$.  
    \item Decoupling of nondegenerate manifolds $|HM_\phi|\sim 1$: by a rearrangement on the coordinates, this follows from the induction hypothesis on $\mathcal M^{n}_{l+1}$.
\end{itemize}

This closes the induction on $l$ and hence the induction on $n$.

Lastly, if $ D^2 \phi$ is positive-semidefinite, then $ D^2\phi_j$ is positive-semidefinite for every $j$. The corresponding decoupling results in \cite{BD2015}\cite{BD2017}\cite{Yang2}\cite{LiYang2023} can all be strengthened to $\ell^2(L^p)$ decoupling. Thus, we eventually obtain a $\ell^2(L^p)$ decoupling.
\end{proof}

\begin{cor}[Decoupling for analytic curves in $\R^n$]\label{cor:smooth_curve} Let $m \in [1,n-1]\cap \Z$ be an intermediate dimension and $2\leq p \leq (m+1)(m+2)$. Let $\phi=(\phi_i)_{i=1}^{n-1}: [-1,1] \to \R^{n-1}$ be such that each $\phi_i$ is real-analytic, and let $M_\phi$ be the curve in $\R^n$ given by 
\begin{equation*}
\{(s,\phi_1(s),\dots,\phi_{n-1}(s)): s \in [-1,1]\}.
\end{equation*}
For each $0<\delta\ll_{\phi} 1$, let  $\mathcal{I}_\delta$ be a collection of disjoint intervals $I$ that are $(\phi,\delta)$-flat but such that $2I\cap [1,2]$ is not $(\phi,\delta)$-flat in dimension $m$. Then $[-1,1]$ can be $\phi$-$\ell^2(L^p)$ decoupled into $\mathcal I_\delta$ at scale $\delta$ as in Definition \ref{defn:graphical_decoupling_convex_hull} at a cost $O_{\varepsilon,\phi,p}(\delta^{-\varepsilon})$ for any $\varepsilon>0$.
\end{cor}
{\it Remarks.}
\begin{itemize}
    \item Unless the whole curve lies in an affine space of dimension $m$, such a partition $\mathcal I_\delta$ always exists for $\delta\ll_\phi 1$, and each $I\in \mathcal I_\delta$ has length at least $\delta^{1/2}$. The interested reader may refer to \cite[Proposition 3.1]{Yang2} for a proof in the case $n=2$; the higher dimensional cases are similar. 
    \item See also \cite[Theorem 1.2.5]{Jaume_thesis} for an analogue of this corollary.
\end{itemize}

\begin{proof}[Sketch of proof of Corollary \ref{cor:smooth_curve} using simplified degeneracy locating principle.]
$ $\newline
For $1\le m\le n-1$, we define the generalised $m\times (n-1)$ Wronskian matrix
\begin{equation}\label{eq:Wons}
   W_{m}\phi(s) : = \left[
  \begin{array}{ccc}
    \horzbar & \phi''(s) & \horzbar \\
    \horzbar & \phi^{(3)}(s) & \horzbar \\
             & \vdots    &          \\
    \horzbar & \phi^{(m+1)}(s) & \horzbar
  \end{array}
\right]. 
\end{equation}
Define $HM_\phi:=\det (W_{m} \phi (W_{m} \phi)^T)$. Since $\phi$ is real analytic, it follows from linear algebra (see, for instance, \cite{Wronskian_linear_dependence}) that $H\phi\equiv 0$ if and only if the curve $\phi$ lies in an $m$-plane.

By the smooth approximation argument in \cite[Section 2.3]{LiYang2023}, it then suffices to fix $d\in \N$ and prove such decoupling assuming each $\phi_i\in \mathcal P_{1,d}$, with the cost of decoupling independent of the coefficients of $\phi_i$.

Let $\mathcal{M}$ be the collection of manifolds $M_\phi$ given by graphs of polynomials $\phi=(\phi_i)_{i=1}^{n-1}: [-1,1] \to \R^{n-1}$ where each $\phi_i\in \mathcal P_{1,d}$.
    

To apply the degeneracy locating principle, we need to check the following:
    \begin{itemize}
        \item $H$ is a degeneracy determinant: this follows from Proposition \ref{prop:degeneracy_determinant_Lojasiewicz}. (See also Section \ref{sec:intro:DD}.)
        \item Decoupling of sublevel sets: this follows from the fundamental theorem of algebra.
        \item Decoupling of totally degenerate manifolds $HM_\phi\equiv 0$: as mentioned above, these manifolds are contained in an $m$-plane. Since we are decoupling in dimension $m$, no decoupling is needed.
        \item Decoupling of nondegenerate manifolds $|HM_\phi|\sim 1$: by a partition of unity into $O_{n,d}(1)$ many pieces and relabelling the coordinates $\phi_i$, we may assume that the Wronskian matrix $W(\phi''_1,\dots,\phi''_{m})$ has determinant $\sim 1$. Its decoupling follows from direct applications of cylindrical decoupling and decoupling for the moment curve in $\R^m$ by Bourgain-Demeter-Guth \cite{BDG2016}.
    \end{itemize}
    
\end{proof}

{\it Remarks.}
\begin{enumerate}
    \item To rigorously prove Corollaries \ref{cor:radial} and \ref{cor:smooth_curve}, applying the full version of Theorem \ref{thm:degeneracy_locating_principle} only gives the existence of partitions $\mathcal I_\delta$. But by \cite[Section 3]{Yang2}, such $\mathcal I_\delta$ is ``essentially optimal" up to trivial partitions of length $\delta^\eps$. See Section \ref{sec:corallary_1.4} in the appendix for a more rigorous treatment for Corollary \ref{cor:radial}. The case of Corollary \ref{cor:smooth_curve} is easier, since there is no radial principle involved.

    \item To keep things simple, we did not mention how we can obtain the width lower bound $\delta$ of the decoupled parallelograms in Corollary \ref{cor:dec_add}. See Part \ref{item:width_lower_bound} of Theorem \ref{thm:degeneracy_locating_principle} and Theorem \ref{thm:refine_IJ}.
    \item (Added July 2025) In subsequent work \cite{LiYang2025}, \cite{LiYang202502}, we are able to apply the degeneracy locating principle to prove many more interesting decoupling results for other degenerate hypersurfaces.
\end{enumerate}

\subsection{Outline of the paper}\label{sec:outline}
In Section \ref{sec:radial} we state and prove the radial principle. In Section \ref{sec:degeneracy_locating} we state and prove the degeneracy locating principle. In Section \ref{sec:degen_approximation_theorem}, we state and prove the degenerate approximation theorem. Section \ref{sec_appendix} is the appendix, which discusses some fundamental concepts about parallelograms, polynomials and flatness in decoupling, and generalisations of previously established theorems in literature.

\subsection{Notation}\label{sec:notation}

We introduce some notation that will be used in this paper.
\begin{enumerate}
    \item \label{item:big_O} We use the standard notation $a=O_M(b)$, or $|a|\lesssim_M b$ to mean that there is a constant $C\ge 1$ depending on some parameter $M$, such that $|a|\leq Cb$. When the dependence on the parameter $M$ is unimportant, we may simply write $a=O(b)$ or $|a|\lesssim b$. Similarly, we define $\gtrsim$ and $\sim$.

    \noindent We use the notation $a\ll_M b$ or $b\gg_M a$ to mean that there is a constant $C\ge 1$ depending on some parameter $M$, such that $|a|\le C^{-1}b$. 

\item Given a subset $M\sub \R^n$, we define the $\delta$-neighbourhood $N_\delta(M)$ to be
    \begin{equation*}
N_\delta(M) := \{ x + c: x \in M, |c|<\delta\}.
\end{equation*}

    \item Given a continuous function $\phi:\Omega\sub \R^k\to \R^{n-k}$, we define the $\delta$-vertical neighbourhood $N^\phi_\delta(\Omega)$ by    
    \begin{equation*}
N_\delta^\phi(\Omega): =\{ (x , \phi(x) + c): x \in \Omega, c \in [-\delta,\delta]^{n-k}\}.
\end{equation*}

\item Given a collection $\mathcal R$ of subsets of $\R^n$, we abbreviate $\cup \mathcal R$ to be the subset of $\R^n$ defined by $\cup_{R\in \mathcal R}R$.

\item We say that an affine transformation $L:\R^n\to \R^m$ given by $Lx=Ax+b$ is bounded by $C>0$, if every entry of $A,b$ is bounded by $C$. When the constant $C$ is clear from the context, we simply say that $L$ is bounded.

\item We denote by $O(n)$ the orthogonal group in $\R^n$, that is, the collection of all real $n\times n$ orthonormal matrices.


\item \label{item:Pnd} We denote by $\mathcal P^0_{n,d}$ the collection of all $n$-variate real polynomials $\phi$ of degree at most $d$. In addition, we denote by $\mathcal P_{n,d}$ the collection of all $n$-variate real polynomials $\phi$ of degree at most $d$, such that $\sup_{[-1,1]^n}|\phi(x)|\le 1$.

\item For a $C^2$ function $\phi:[-1,1]^{k}\to \R^{n-k}$, we denote by $\nabla \phi$ the Jacobian matrix of $\phi$, that is, $\nabla \phi=(\frac{\partial \phi_i}{\partial x_j})_{1\le i\le n-k,1\le j\le k}$. For each $x\in [-1,1]^k$, we denote by $D^2 \phi(x)$ the bilinear map $\R^{k}\times \R^k\to \R^{n-k}$ defined by 
\begin{equation}
    (v,w)\mapsto 
        (v^T D^2 \phi_1(x) w,\dots, v^T D^2 \phi_{n-k}(x) w)^T.
\end{equation}
We denote by $\norm{D^2 \phi}_\infty:=\sup_i \norm{D^2 \phi_i}_\infty$, which is equivalent to the operator norm of $D^2 \phi$, up to a constant depending only on $n$. We also denote $\norm{\phi}_{C^2}=\norm{\phi}_\infty+\norm{\nabla \phi}_\infty+\norm{D^2 \phi}_\infty$.

\item For $\phi=(\phi_i):\R^k\to \R^{n-k}$ and $\vec \sigma=(\sigma_i)\in \R^{n-k}$, we denote by $\vec \sigma \phi$ the function $\R^k\to \R^{n-k}$ defined by the entrywise product
\begin{equation}\label{eqn:vector:entrywise_product}
    (\sigma_1 \phi_1,\dots,\sigma_{n-k}\phi_{n-k}).
\end{equation}
Furthermore, if each $\sigma_i\ne 0$, we denote
\begin{equation*}
    \vec \sigma^{-1}:=(\sigma_1^{-1},\dots,\sigma_{n-k}^{-1}).
\end{equation*}

\end{enumerate}

\subsubsection{Notation about parallelograms}\label{sec:equivalence_objects}
Decoupling inequalities are formulated using parallelograms in $\R^n$. To deal with technicalities, it is crucial that we make clear some fundamental geometric concepts.

\begin{enumerate}
    \item A parallelogram $R\sub \R^n$ is defined to be a set of the form
    \begin{equation*}
        \{x+b\in \R^n:|x\cdot u_i|\le l_i\},
    \end{equation*}
    where $\{u_i:1\le i\le n\}$ is a basis of $\R^n$ consisting of unit vectors, $b\in \R^n$ and $l_i\ge 0$. If $\{u_i:1\le i\le n\}$ is in addition orthogonal, we say that $R$ is a rectangle. Unless otherwise specified, we always assume a parallelogram in $\R^n$ has positive $n$-volume, namely, $l_i>0$ for all $1\le i\le n$. We say that two parallelograms are disjoint if their interiors are disjoint. 
    
    \noindent We define the {\it width} of a parallelogram $P$ to be 
    \begin{equation}\label{eqn:width}
        \inf\{\text{smallest dimension of $T$: $T$ is a rectangle containing $P$} \}.
    \end{equation}
    By the John ellipsoid theorem \cite{John}, every convex body in $\R^n$ is $C_n$-equivalent to a rectangle in $\R^n$, for some dimensional constant $C_n$. For this reason and in view of the enlarged overlap introduced in Definition \ref{defn:enlarged_overlap}, we often do not distinguish rectangles from parallelograms.

    \item Let $R\sub \R^n$ be a parallelogram. We denote by $\lambda_R$ an affine bijection from $[-1,1]^n$ to $R$. (There is more than one choice of $\lambda_R$, but in our paper, any such a choice will work identically, since in decoupling inequalities, the orientations of such affine bijections do not make any difference.)

    \item Unless otherwise specified, for a parallelogram $S\sub \R^n$ and $C>0$, we denote by $CS$ the concentric dilation of $S$ by a factor of $C$. Given $t\in \R^n$, we denote $S+t=\{s+t:s\in S\}$. Note that in this notation, we have $C(S+t)=CS+t$.

    \item Given a subset $S\sub \R^n$, a parallelogram $R\sub \R^n$ and $C\ge 1$, we say $S,R$ are $C$-equivalent if $C^{-1}R\sub S\sub CR$. If the constant $C$ is unimportant, we simply say that $S$ is equivalent to $R$, or that $S$ is an {\it almost parallelogram}.

    \item For $0<\delta<1$, by a {\it tiling} of a parallelogram $R\sub \R^n$ by cubes of side length $\delta$, we mean a covering $\mathcal T$ of $R$ by translated copies $T$ of a cube of side length $\delta$, such that different $T$ and $T'$ from $\mathcal T$ have disjoint interiors, and that $T\sub 2R$.

    \item For a subset $S\sub \R^n$, we denote by $\mathrm{Co}(S)$ the convex hull of $S$ in $\R^n$, namely, $\mathrm{Co}(S)$ is the smallest convex subset of $\R^n$ that contains $S$. Note that if $S\sub R$ where $R$ is a parallelogram, then $\mathrm{Co}(S)\sub R$.
    
\end{enumerate}

\subsubsection{Enlarged overlap}
Let $\mathcal R$ be a family of parallelograms. In all decoupling inequalities in the literature, we hope that the parallelograms in $\mathcal R$ do not overlap too much. Since we often deal with constant enlargement of parallelograms, we will need a slightly stronger version of the overlap condition.

\begin{defn}\label{defn:enlarged_overlap}
By an {\it overlap function} we mean an increasing function $B:[1,\infty)\to [1,\infty)$. Given a finite family $\mathcal R$ of parallelograms in $\R^n$ and an overlap function $B$, we say $\mathcal R$ is $B$-overlapping, or $B$ is an overlap function of $\mathcal R$, if for every $\mu\ge 1$ we have
    \begin{equation*}
        \sum_{R\in \mathcal R}1_{\mu R}\le B(\mu).
    \end{equation*}
\end{defn}

In the context of decoupling, the function $B$ may depend on many parameters such as $\eps$ and the family of the manifolds we are decoupling, but we hope that it is independent of $\delta$. In particular, if the function $B$ depends on unimportant parameters that are clear from the context, we simply say that $\mathcal R$ is boundedly overlapping. (In the main theorem of \cite{LiYang2023}, the overlap function depends on $\delta$, but we prove an analogous Theorem \ref{thm:refine_IJ} in the appendix, where the overlap is independent of $\delta$.)

\subsection{Acknowledgement}
The first author thanks Betsy Stovall for preliminary discussions at an early stage of the work. The first author was supported by AMS-Simons Travel Grants. The second author was supported by the Croucher Fellowships for Postdoctoral Research. The authors thank Terence Tao for numerous suggestions on both the content and the presentation of this article.

\section{The radial principle}\label{sec:radial}
In this section, we will introduce and prove the radial principle of decoupling.

\subsection{Setup}\label{subsub:01}
Let $k,l\ge 1$. Given a $C^2$ function $r(s)$ defined on $[-3,3]^k$ and mapping into $[1,2]$. Given a $C^2$ function $\psi(t)$ defined on $[-3,3]^l$ and mapping into $\R$. Suppose also that we are in either of the following scenarios:
\begin{itemize}
    \item $r(s)$ is affine, or
    \item $L\psi(t):=\psi(t)-t\cdot \nabla \psi(t)$ is away from $0$ on $[-3,3]^l$.    
\end{itemize}
We further assume that $\psi$ satisfies a regularity condition.
\begin{defn}\label{condition:2T_flat}
    We say $\psi$ is polynomial-like with parameter $C_{\mathrm{poly}}\ge 1$ if for every parallelogram $\omega$ such that $2\omega\sub [-3,3]^l$, we have $\norm{\psi}_{C^2(2\omega)}\le C_{\mathrm{poly}} \norm{\psi}_{C^2(\omega)}$.
\end{defn}
Thus, using Lemma \ref{lem:C2flat_entire} with $k=m=n-1=l$, we have the following corollary.
\begin{cor}\label{cor:2T_flat}
    Let $\psi:[-3,3]^l\to \R$ be polynomial-like with constant $C_{\mathrm{poly}}$. Then there is a constant $C'_{\mathrm{poly}}$ depending on only $C_{\mathrm{poly}}$ and $l$, such that the following holds: let $\omega\sub [-3,3]^l$ be a parallelogram with $2\omega\sub [-3,3]^l$ that is $(\psi,\delta)$-flat in the third sense of Definition \ref{defn:new_flatness}, namely,
    \begin{equation*}
        \sup_{\substack{t \in \omega, v\in \mathbb S^{l-1}  \\ a\in \R: t+a v\in \omega}} \left | v^T D^2 \psi(t)v\right| |a|^2 \le \delta.
    \end{equation*}
    Then $2\omega$ is $(\psi,C'_{\mathrm{poly}}\delta)$-flat in the third sense of Definition \ref{defn:new_flatness}, namely,
    \begin{equation*}
        \sup_{\substack{t \in 2\omega, v\in \mathbb S^{l-1}  \\ a\in \R: t+a v\in 2\omega}} \left | v^T D^2 \psi(t)v\right| |a|^2 \le C'_{\mathrm{poly}} \delta.
    \end{equation*}   
\end{cor}

{\it Remarks.}
\begin{itemize}
    \item By Corollary \ref{cor:polycoeff}, every polynomial $\psi\in \mathcal P_{l,d}$ is polynomial-like, with $C_{\mathrm{poly}}$ depending only on $d,l$.
    \item If $\det D^2\psi$ is away from $0$ on $[-3,3]^l$, then $\psi$ is polynomial-like with $C_{\mathrm{poly}}$ depending on the maximum ratio between the largest and the smallest eigenvalues (in absolute value) of $D^2\psi(t)$, $t\in [-3,3]^l$.
    \item A canonical example of $\psi(t)$ is given by $\psi(t)=\sqrt{1-|t|^2/(10l)}$, arising from hypersurfaces of revolution (after a constant rescaling). This $\psi$ satisfies both $L\psi(t)\sim_l 1$ and Definition \ref{condition:2T_flat} with $C_{\mathrm{poly}}\sim_l 1$. A slightly different model case satisfying both conditions is $\psi(t)=1-|t|^2$. Both motivate the terminology ``radial principle". We actually prove a slightly more general statement, which will be used in \cite[Theorem 1.6]{LiYang2025}.
\end{itemize}

\fbox{Notation} We use the following notation for the rest of this section. Let $r(s)$ and $\psi(t)$ be given as in Section \ref{subsub:01}. Any kind of $\delta$-flatness refers to Definition \ref{defn:new_flatness} with $m=n-1$. The exponents $p,q,\alpha$ in the $(\ell^q(L^p),\alpha)$ decoupling will be fixed and omitted. Every implicit constant below in this section is allowed to depend on $p,q,\alpha,k,l,\norm{r}_{C^1},\norm{\psi}_{C^2}$, the constant $C_{\mathrm{poly}}$ in Definition \ref{condition:2T_flat}, and $\inf_{t\in [-3,3]^l} |L\psi(t)|$ (dependence on this last parameter can be dropped if $r(s)$ is affine).

\subsubsection{Decoupling assumption}\label{sec:decoupling_assumptions}
We assume that for every $0<2\delta<\sigma\le 1$ and every $\eps>0$, the following conditions are satisfied:

\begin{condition}[Decoupling for sum surface]\label{condition:decoupling}
If a parallelogram $S_\sigma\sub [-2,2]^k$ is $(r,O(\sigma))$-flat in the first sense of Definition \ref{defn:new_flatness}, namely, there exists an affine function $A(s)$ such that 
    \begin{equation}\label{eqn:F1_conical}
        \sup_{s\in S_\sigma}|r(s)-A(s)|\lesssim \sigma,
    \end{equation}
and if a parallelogram $T_\sigma\sub [-2,2]^l$ is $(\psi,O(\sigma))$-flat in the third sense of Definition \ref{defn:new_flatness}, namely, 
\begin{equation}\label{eqn:F3_conical}
    \sup_{\substack{t \in T_\sigma, v\in \mathbb S^{l-1}  \\ a\in \R: t+a v\in T_\sigma}} \left | v^T D^2 \psi(t)v\right| |a|^2 \lesssim \sigma,
\end{equation}
and $|b|\sim 1$, then $S_\sigma\times T_\sigma$ can be $(br(s)+\psi(t))$-decoupled into
a collection $\mathcal K_\delta(S_\sigma,T_\sigma)$ consisting of $(br(s)+\psi(t),O(\delta))$-flat parallelograms, at the cost of $\le C_\eps (\delta\sigma^{-1})^{-\eps}$, where $C_\eps$ is independent of $\delta,\sigma$. 

Furthermore, we have the following assumptions on the covering rectangles: 
\begin{enumerate}[label=(\alph*)]
\item \label{item_Feb_5} $S_\sigma$ and $T_\sigma$ are covered by a collection of parallelograms $\mathcal S_\delta(S_\sigma)$ and $\mathcal T_\delta(T_\sigma)$, respectively, such that
\begin{equation*}
        \mathcal K_\delta(S_\sigma,T_\sigma)=\{S_\delta\times T_\delta:S_\delta\in \mathcal S_\delta(S_\sigma),T_\delta\in \mathcal T_\delta(T_\sigma)\}.
    \end{equation*}    
Also, $\cup \mathcal S_\delta(S_\sigma)\sub (2S_\sigma)\cap [-3,3]^k$ and $\cup \mathcal T_\delta(T_\sigma)\sub (2T_\sigma)\cap [-3,3]^l$.

\item \label{item_02_Nov_20} The collection $\mathcal K_\delta(S_\sigma,T_\sigma)$ has cardinality $\le \delta^{-C_{\mathrm{card}}}$. The constant $C_{\mathrm{card}}$ and the overlap function of $\mathcal K_\delta(S_\sigma,T_\sigma)$ are independent of $\sigma,\delta$.

    
    \item \label{item_04_Nov_20} We have the following size condition for each $T_\delta,T_\sigma$:
    \begin{equation}\label{eqn:size_condition_conical}
        \sigma^{-1}(T_\sigma-c(T_\sigma))\sub \delta^{-1} (T_\delta-c(T_\delta)),
    \end{equation}
    where $c(T_\sigma)$ and $c(T_\delta)$ are the centres of $T_\sigma$ and $T_\delta$, respectively.     
\end{enumerate}
\end{condition}

Before we move on to talk about the main theorem, we first leave a few important remarks, and explain why these assumptions are feasible in practice.

\begin{enumerate}
\item Logically, the decoupling assumption for the sum function $br(s)+\psi(t)$ ($b\sim 1$) implies the corresponding decoupling for $r(s)$ and $\psi(t)$ separately. However, in practice, the decoupling for $br(s)+\psi(t)$ usually follows from the degeneracy locating principle (Theorem \ref{thm:degeneracy_locating_principle}), which in turn requires the decoupling for both $r(s)$ and $\psi(t)$ as input. Thus, this assumption is essentially the decoupling for all three functions $r(s)$, $\psi(t)$ and $br(s)+\psi(t)$.


\item \label{item:Nov_22_remark_2} The cardinality assumption in Part \ref{item_02_Nov_20} facilitates the dyadic pigeonholing argument needed to combine decoupling at different levels; see Proposition \ref{prop:combine_decoupling}.

\end{enumerate}
Now we briefly explain why Condition \ref{condition:decoupling} is usually feasible. The existence of $\mathcal S_\delta(S_\sigma)$ and $\mathcal T_\delta(T_\sigma)$ is usually guaranteed; we only provide the reason for $\mathcal T_\delta(T_\sigma)$ since it is similar to $\mathcal S_\delta(S_\sigma)$. For instance, we have:
\begin{itemize}
    \item If $\psi$ is affine, we take $\mathcal T_\delta(T_\sigma)$ to be the singleton $\{T_\sigma\}$. The size Condition \eqref{eqn:size_condition_conical} is trivial, since $\sigma> \delta$.
    \item If $D^2\psi\succ 0$, by \cite{BD2015}, we can take $ T_\sigma$ to be a cube of side length $\sigma^{1/2}$, and 
    $\mathcal T_\delta(T_\sigma)$ to be the tiling of $T_\sigma$ into cubes of side length $\delta^{1/2}$. Condition \eqref{eqn:size_condition_conical} is also easy to verify.
    \item If $\psi$ has at most two variables, then one can show that the existence of a family $\mathcal T_\delta(T_\sigma)$ obeying \eqref{eqn:size_condition_conical} follows from Theorem \ref{thm:refine_IJ}.
\end{itemize} 

\subsection{Statement of Theorem}
We are now ready to state the radial principle of decoupling. For each $\delta>0$, we define the $\delta$-(vertical) neighbourhood of the ``product" surface:
\begin{align}
    \Pi_\delta(S,T)&:=\{(s,r(s)t,u):s\in S,t\in T,|u-r(s)\psi(t)|\le \delta\}.\label{eqn:defn_Pi(S,T)}
\end{align}

\begin{thm}[Radial principle of decoupling]\label{thm:radial_principle}
    Assume Condition \ref{condition:decoupling}. Then for $\eps\in (0,1)$ and $0<\delta\ll 1$, there exist families $\mathcal{S}_\delta$ and $\mathcal{T}_\delta$ of parallelograms contained in $[-2,2]^k$, $[-2,2]^l$, respectively, such that the set $\Pi_\delta([-1,1]^k,[-1,1]^l)$ can be decoupled into the following family
    \begin{equation}\label{eqn:defn_mathcal_Pdelta}
        \mathcal P_\delta:=\{\Pi_{\delta}(S,T):S\in \mathcal S_\delta,\,\,T\in \mathcal T_\delta\}
    \end{equation}
    of almost parallelograms with width at least $\delta$ at cost $\le C'_\eps\delta^{-\eps}$. Here, the constant $C'_\eps$ and the overlap functions of $\mathcal S_\delta$, $\mathcal T_\delta$ and $\mathcal P_\delta$ are independent of $\delta$ but may depend on $\eps$ and the constants $C_\eps,C_{\mathrm{card}}$ mentioned in Condition \ref{condition:decoupling}.

\smallskip

Moreover, we have the following special cases: 
\begin{enumerate}
    \item \label{item:Mar_25_radial_T} Suppose that in Condition \ref{condition:decoupling}, $T_\sigma$ is a cube of side length $ \sigma^{1/2}$, and $\mathcal T_\delta(T_\sigma)$ can be taken to be a tiling of $T_\sigma$ by cubes of side length $\delta^{1/2}$. Then the final $\mathcal T_\delta$ can be taken to be a tiling of $[-1,1]^l$ by cubes of side length $ \delta^{1/2}$.
    \item \label{item:Mar_25_radial_S} Suppose that there exist given families $\{\tilde{\mathcal S_\delta}:0<\delta\le 1\}$ each consisting of $(r,\delta)$-flat parallelograms (in the first sense of Definition \ref{defn:new_flatness}) whose overlap function is independent of $\delta$, such that the following condition (*) holds:
\begin{itemize}
\item [*] If $0<2\delta<\sigma\le 1$ and $S_\sigma \in \tilde{\mathcal{S}}_\sigma$, then $\mathcal{S}_\delta(S_\sigma)$ in Condition \ref{condition:decoupling} can always be taken to be
\begin{equation}\label{eqn:Nov_21}
\mathcal{S}_\delta(S_\sigma) = \{ S_\delta \in \tilde{\mathcal{S}}_\delta : S_\delta \cap S_\sigma \neq \emptyset,\,\, S_\delta\sub C_r S_\sigma\},
\end{equation}
for some $C_r$ depending on the function $r(s)$.
\end{itemize}
Then in \eqref{eqn:defn_mathcal_Pdelta}, $\mathcal S_\delta$ can be taken to be $\tilde{\mathcal S}_\delta$. In this case, the overlap function of $\mathcal P_\delta$ (and in turn the decoupling constant $C'_\eps$) may also depend on the constant $C_r$.
\end{enumerate} 
\end{thm}
{\it Remarks.} 
\begin{itemize}
    \item Strictly speaking, $\Pi_\delta([-1,1]^k,[-1,1]^l)$ should be decoupled into exact parallelograms as in Definition \ref{defn:decoupling}, not subsets equivalent to parallelograms. However, by slightly enlarging the related overlap functions, we may abuse notation and pretend that they were exact parallelograms.
    \item The special case in Part \eqref{item:Mar_25_radial_S} roughly says the following: suppose that we already know the $r$-decoupling of $[-1,1]^k$ into some $\tilde{\mathcal S}_\delta$. Then we can guarantee that $\Pi_\delta([-1,1]^k,[-1,1]^l)$ can be decoupled into parallelograms of the form $\Pi_\delta(S_\delta,T_\delta)$ where $S_\delta$ belongs to the given $\tilde{\mathcal S}_\delta$. This is particularly useful when we study decoupling for surfaces of revolutions generated by a radius function $r(s)$ with an ``essentially unique maximal partition", such as Corollary \ref{cor:radial}. See also Section \ref{sec:corallary_1.4} for a more detailed discussion of this technicality.
\end{itemize}


For example, in the case of the standard truncated light cone in $\R^3$ studied in \cite[Section 8]{BD2015}, we may take the following choices:
\begin{align*}
    &\cdot k=l=1,\quad r(s)=s+2, \quad \psi(t)=\sqrt{1-t^2/2},\\
    &\cdot \mathcal S_\delta \text{ is the trivial partition $\{[-1,1]\}$},\\
    &\cdot \mathcal T_\delta \text{ is the partition of $[-1,1]$ into intervals of length $\delta^{1/2}$},\\
    &\cdot \alpha=0, \quad p=6, \quad q=2.
\end{align*}
Thus, we have the following chain of reductions of decoupling:
\begin{align*}
    &(s,st,s\sqrt{1-t^2/2})\\
    \implies & (s,t,s+\sqrt{1-t^2/2}) \quad \text{by radial principle},\\
    \implies & (s,t,\sqrt{1-t^2/2}) \quad \text{by affine equivalence},\\
    \implies & (t,\sqrt{1-t^2/2})\quad \text{by cylindrical decoupling},
\end{align*}
which can be dealt with using \cite[Theorem 1.1]{BD2015}.

\subsection{Proof of Theorem \ref{thm:radial_principle}}
The rest of this section is devoted to the proof of Theorem \ref{thm:radial_principle}, based an induction on scales argument.
\subsubsection{Preliminary reductions}\label{sec:trivial_decoupling}

We may let $\eps\in \N^{-1}\cap (0,10^{-10})$ and $\delta\ll_\eps 1$. Define $\delta_1=\delta^{10\eps}$, and define
\begin{equation}\label{eqn:defn_delta_i}
    \delta_i:=\delta^{(i+9)\eps}, \quad 2\le i\le N,
\end{equation}
where $N= \eps^{-1}-9$.

Our first step is to trivially decouple $[-1,1]^k$ and $[-1,1]^l$ by grids of side length $\sim\delta_1^{1/2}$, so that each of them is $(r,\delta_1)$-flat (respectively $(\psi,\delta_1)$-flat). The cost of this trivial decoupling is $\delta^{-O(\eps)}$ which is acceptable. 

Let $2\le i\le N$ and suppose we have already decoupled into $\Pi_{\delta_{i-1}}(S,T)$ (recall \eqref{eqn:defn_Pi(S,T)}, where $S\in \mathcal S_{\delta_{i-1}}$ and $T\in \mathcal T_{\delta_{i-1}}$.

The main goal now is to further decouple into $\Pi_{\delta_{i}}(S',T')$, where $S'\in \mathcal S_{\delta_i}$ and $T'\in \mathcal T_{\delta_i}$.

We also perform a reduction to $r(s)$. Denote by $\bar c$ the centre of $S$. By a trivial rescaling in $r(s)$ by $r(\bar c)^{-1}$ and a translation in $s$, we may assume for simplicity that $\bar c=0$ and $r(0)=1$. We record that
\begin{equation}\label{eqn:r(s)_bound}
    |r(s)-1|\lesssim \delta_1^{1/2},\quad \forall s\in S.
\end{equation}
We also record an estimate on the width of $T\in \mathcal T_{\delta_i}$ for each $2\le i\le N$, which follows from Condition \ref{condition:decoupling}. More precisely, if we take 
$\delta,\sigma$ in \eqref{eqn:size_condition_conical} to be $\delta_i,\delta_{i-1}$, respectively, this implies that for all $2\le i\le N$, the width of $T'$ is bounded below by $\delta_{i}\delta_{i-1}^{-1}w(T)$. Thus, for $i=2$, using $w(T)=\delta_1^{1/2}$, we have $w(T')\ge \delta^{6\eps}$. Inductively,
\begin{equation}\label{eqn:length_Ti}
T\in \mathcal T_{\delta_i} \text{ has width } \ge \delta^{(i+4)\eps}\gg \delta_{i-1},\quad \forall 2 \le  i\le N.
\end{equation}

\subsubsection{Approximation}
This subsection will be the key to this proof. We first describe the main heuristic argument. First, we argue by affine invariance that we can let $T$ be centred at $0$, and that $\nabla \psi(0)=0$. Then by Taylor expansion, we can write
\begin{equation*}
    r(s) \psi(t)=(1+E(s))(\psi(0)+Q(t)),
\end{equation*}
where $E(s)$ and $Q(t)$ are $O(\delta_{i-1})$ over $S$ and $T$, respectively, and their gradients are $0$ at $0$. Thus, by affine invariance, the $r(s)\psi(t)$ can be approximated by $E(s)\psi(0)+Q(t)$ at scale $\delta_i$, since $E(s)Q(t)=O(\delta^2_{i-1})\le \delta_{i}$. But then reversing the affine invariance, the decoupling of $E(s)\psi(0)+Q(t)$ is equivalent to $r(s)L\psi(0)+\psi(t)$. More precisely, we now need to decouple the manifold parametrised by
\begin{equation*}
    (s,t,r(s)L\psi(0)+\psi(t)).
\end{equation*}
If $r(s)$ is affine, then the decoupling problem reduces to that for
\begin{equation*}
    (s,t,\psi(t)).
\end{equation*}
If $L\psi(0)\ne 0$, the decoupling problem essentially reduces to that for
\begin{equation*}
    (s,t,r(s)+\psi(t)).
\end{equation*}

We now provide the technical details.

\fbox{Translation in $t$}\label{subsub:Nov_10}

Denote by $t_0$ the centre of $T$, and define $\tau=t-t_0$ and $T_0:=T-t_0$, so $\tau\in T_0$ and $T_0$ is centred at $0$. Then we need to decouple
\begin{equation}\label{eqn:Nov_17_01}
    \{(s,r(s)(\tau+t_0),r(s)\psi(\tau+t_0)+u):s\in S,\tau\in T_0,|u|\le \delta_{i}\}.
\end{equation}
Define
\begin{equation*}
    Q(\tau):=\psi(\tau+t_0)-\psi(t_0)-\nabla\psi(t_0)\cdot \tau,
\end{equation*}
so that by affine transformations, \eqref{eqn:Nov_17_01} becomes
\begin{equation}\label{eqn:Nov_17_02}
    \{(s,r(s)(\tau+t_0),r(s)L\psi(t_0)+r(s)Q(\tau)+u):s\in S,\tau\in T_0,|u|\le \delta_{i}\}.
\end{equation}
We also record that
\begin{equation}\label{eqn:Nov_17_Q}
    Q(0)=0,\quad \nabla Q(0)=0,\quad |Q(\tau)|\le \delta_{i-1}.
\end{equation}

\fbox{Affine transformation in $s$}

We now recall the assumption that $|r(s)-A(s)|\le \delta_{i-1}$ for $s\in S$. Define an affine transformation $\Lambda:\R^{k+l}\to  \R^{k+l}$ by 
\begin{equation}\label{eqn:Lambda_conical}
    \Lambda(s,t):=(s,t-A(s)t_0).
\end{equation}
Then applying $\Lambda$ to \eqref{eqn:Nov_17_02}, we first decouple 
\begin{equation}\label{eqn:Nov_17_03}
    \{(s,r(s)(\tau+t_0)-A(s)t_0,r(s)L\psi(t_0)+r(s)Q(\tau)):s\in S,\tau\in T_0,|u|\le \delta_{i}\}.
\end{equation}

\fbox{Change of variables}

We now change
\begin{equation*}
    \theta:=r(s)(\tau+t_0)-A(s)t_0,
\end{equation*}
so that \eqref{eqn:Nov_17_03} becomes
\begin{equation}\label{eqn:Nov_17_03'}
\{(s,\theta,r(s)L\psi(t_0)+r(s)Q(\tau)+u):(s,\theta)\in R(S,T_0),|u|\le \delta_{i}\},
\end{equation}
where
\begin{equation*}
R(S,T_0):=\{(s,\theta):s\in S,\theta\in r(s)T_0+(r(s)-A(s))t_0\}.
\end{equation*}
To analyse $R(S,T_0)$, we have the following lemma.
\begin{lem}\label{lem_Nov_17_01}
    The set $R(S,T_0)$ is equivalent to the parallelogram $S\times T_0$.
\end{lem}
\begin{proof}
    We prove the $\sub$ direction first. Fix $s\in S$. If $\tau\in T_0$, then 
    \begin{align*}
        r(s)\tau-(r(s)-A(s))t_0=r(s)\tau+O(\delta_{i-1}). 
    \end{align*}
    But $|r(s)-1|\lesssim \delta_1$, $\delta\ll_\eps 1$, and $T_0$ has width $\gg \delta_{i-1}$ by \eqref{eqn:length_Ti}, we have $r(s)\tau+O(\delta_{i-1})\in 2T_0$.

    On the other hand, we show that if $\tau\in \tfrac 1 2 T_0$, then $r(s)^{-1}(\tau-(r(s)-A(s))t_0)\in T_0$. For, if $\tau\in \tfrac 1 2 T_0$, then
    \begin{align*}
        r(s)^{-1}(\tau-(r(s)-A(s))t_0)
        &=r(s)^{-1}\tau+O(\delta_{i-1}).
    \end{align*}
    Since $r(s)\ge 3/4$, $\tau\in \tfrac 1 2 T_0$ which has width $\gg \delta_{i-1}$ by \eqref{eqn:length_Ti}, the result follows.
\end{proof}
With this lemma, to decouple \eqref{eqn:Nov_17_03'}, we equivalently decouple
\begin{equation}\label{eqn:Nov_17_04}
    \{(s,\theta,r(s)L\psi(t_0)+r(s)Q(\tau)+u):s\in S,\theta\in 2T_0,|u|\le \delta_{i}\}.
\end{equation}

\fbox{Approximation}

We then approximate \eqref{eqn:Nov_17_04} at scale $\delta_i$ by 
\begin{equation}\label{eqn:Nov_17_05}
    \{(s,\theta,r(s)L\psi(t_0)+Q(\theta)+u):s\in S,\theta\in 2T_0,|u|\le  2\delta_{i}\}.
\end{equation}
To see this, we prove that
\begin{equation}
    |Q(\theta)-r(s)Q(\tau)|<\delta_i.
\end{equation}
To this end, we first use \eqref{eqn:Nov_17_Q} and \eqref{eqn:r(s)_bound} to get
\begin{equation*}
    |(r(s)-1)Q(\tau)|=O(\delta_1\delta_{i-1})<\delta_i/2.
\end{equation*}
Thus, it suffices to prove
\begin{equation}
    |Q(\theta)-Q(\tau)|<\delta_i/2.
\end{equation}
To this end, we use the following elementary calculus lemma.
\begin{lem}\label{lem:calculus}
Let $T_0\sub [-1,1]^l$ be a parallelogram centred at $0$. Let $Q:T\to \R$ be a $C^2$ function with $\nabla Q(0)=0$. For all $\theta\ne\tau\in T_0$, denote by $v$ the unit vector in the same direction as $\theta-\tau$. Then we have
\begin{equation*}
    |Q(\theta)-Q(\tau)|\lesssim \sup_{t'\in T_0}|v^T D^2 Q(t') v||\theta-\tau|(|\theta|+|\tau|),
\end{equation*}
where the implicit constant depends on $l$ only.
\end{lem}
\begin{proof}[Proof of lemma]
    Given $\theta\ne\tau\in T$. Then
    \begin{align*}
        |Q(\theta)-Q(\tau)|
        &=\left|\int_0^1 \nabla Q((1-u)\tau+u\theta)\cdot (\theta-\tau) du\right|\\
        &=\left|\int_0^1 \nabla Q((1-u)\tau+u\theta)\cdot (\theta-\tau) du-\int_0^1 \nabla Q(0)\cdot (\theta-\tau) du\right|\\
        &\le |\theta-\tau|\int_0^1 |\nabla Q((1-u)\tau+u\theta)-\nabla Q(0)|du\\
        &\lesssim \sup_{t'\in T}|v^T D^2 Q(t') v| |\theta-\tau| (|\theta|+|\tau|). 
    \end{align*}
\end{proof}
Now by \eqref{eqn:F3_conical}, Corollary \ref{cor:2T_flat} and Lemma \ref{lem:calculus} with $2T_0$ in place of $T_0$, we can estimate
\begin{align*}
    |Q(\theta)-Q(\tau)|
    &\lesssim \sup_{t'\in 2T_0}|v^T D^2 Q(t') v||\theta-\tau|(|\theta|+|\tau|)\\
    &\lesssim \delta_{i-1}(|\theta|+|\tau|)\\
    &\lesssim \delta_{i-1}\delta_1^{1/2}\\
    &<\delta_i,
\end{align*}
where the second to last line follows from \eqref{eqn:defn_delta_i}. In conclusion, we have reduced the decoupling \eqref{eqn:Nov_17_04} to \eqref{eqn:Nov_17_05}, which, in the language of Definition \ref{defn:graphical_decoupling_convex_hull}, is to $\phi$-decouple $S\times 2T_0$ where 
\begin{equation*}
    \phi(s,\theta):=r(s)L\psi(t_0)+Q(\theta).
\end{equation*}

\subsubsection{Applying decoupling for sum surface}\label{sec:decoupling_section_radial}
We now apply the decoupling assumption for the sum surface given in Section \ref{sec:decoupling_assumptions}. We first deal with the case where $|L\psi(c)|\sim 1$, and then point out the differences in the easier case where $r(s)$ is affine.

We apply Condition \ref{condition:decoupling} with $S_\sigma$ and $T_\sigma$ taken to be $S$ and $2T_0$, respectively, and $\sigma$ and $\delta$ taken to be $\delta_{i-1}$ and $\delta_i$, respectively, to $\phi$-decouple $S\times 2T_0$ into
\begin{equation*}
    \mathcal K_{\delta_i}(S,2T_0)=\{S'\times T':S'\in \mathcal S_{\delta_i}(S),T\in \mathcal T_{\delta_i}(2T_0)\}
\end{equation*}
at the cost of $C_\eps (\delta_i\delta_{i-1}^{-1})^{-\eps}$. In particular, by \eqref{eqn:size_condition_conical}, we record that
\begin{equation}\label{eqn:size_T0_T'}
    T'-c(T')\supseteq 2\delta_i \delta_{i-1}^{-1}T_0\supseteq \delta^{2\eps} T_0.
\end{equation}

The easier case where $r(s)$ is affine is similar, except that we do not partition $S$ further.


\subsubsection{Reversing affine transformation}

Now we go back and consider \eqref{eqn:Nov_17_03}. The above decoupling gives rise to parallelograms of the form
\begin{equation*}
    \{(s,r(s)t-A(s)t_0):s\in S',r(s)t-A(s)t_0\in T'- t_0\}=S' \times (T'-t_0).
\end{equation*}
To close the induction, we prove the following lemma.
\begin{lem}\label{lem:Nov_17_02}
    The parallelogram $ S' \times (T'-t_0)$ is equivalent to the set 
    \begin{equation}
        \{(s,r(s)t'-A(s)t_0):s\in S',t'\in T'\}.
    \end{equation}
\end{lem}
\begin{proof}
    We prove the $\sub$ direction first. Fix $s\in S'$. If $t'\in T'$, then we need to show that $r(s)^{-1}(t'-t_0+A(s)t_0)\in T'$, or $r(s)^{-1}(t'-t_0+A(s)t_0)-t'_0\in T'-t'_0$ where $t'_0$ is the centre of $T'$. To this end, we compute
    \begin{align*}
        r(s)^{-1}(t'-t_0+A(s)t_0)-t'_0
        &=r(s)^{-1}(t'-t_0)+t_0-t'_0+O(\delta_{i-1}).
    \end{align*}
    But using \eqref{eqn:length_Ti}, $T'$ has width $\gg \delta_{i-1}$. Thus, it suffices to show that $r(s)^{-1}(t'-t_0)+t_0-t_0'\in T'$.
    To see this, we further compute
    \begin{align*}
        &r(s)^{-1}(t'-t_0)+t_0-t_0'\\
        &=r(s)^{-1}(t'-t_0+t_0 r(s)-t'_0 r(s))\\
        &=r(s)^{-1}(t'-t_0+t_0 r(s)-t'_0-t'_0r(s)+t'_0)\\
        &=r(s)^{-1}\Big((t'-t'_0)+(r(s)-1)(t_0-t'_0)\Big).
    \end{align*}
    Since $t'_0\in T'\sub CT_0$, we use \eqref{eqn:r(s)_bound} and \eqref{eqn:size_T0_T'} to conclude that
    \begin{equation*}
        (r(s)-1)(t_0-t'_0)\in C\delta_1^{1/2} T-t_0\sub \delta^{2\eps}T_0\sub T'-t'_0.
    \end{equation*}
    Thus, $(t'-t'_0)+(r(s)-1)(t_0-t'_0)\in 2T-t_0'$. Since $r(s)\ge 1$, the $\sub$ direction follows.

    We then prove the $\supseteq $ direction. Fix $s\in S'$ and $t'\in T'$. We want to show that $r(s)t'-A(s)t_0+t_0\in T'$. To this end, using \eqref{eqn:length_Ti}, it suffices to show that $r(s)(t'-t_0)+t_0\in T'$, or $r(s)(t'-t_0)+t_0-t_0'\in T'-t_0'$. But then we compute
    \begin{align*}
        &r(s)(t'-t_0)+t_0-t'_0\\
        &=r(s)(t'-t_0'+t_0'-t_0)-t'_0+t_0\\
        &=r(s)(t'-t_0')+(r(s)-1)(t_0'-t_0).
    \end{align*}
    Using the same argument as in the $\sub$ direction, we are done.
\end{proof}
By Lemma \ref{lem:Nov_17_02}, we see that the decoupled coordinate parallelograms at scale $\delta_i$ can be regarded as of the form $\{(s,r(s)t-A(s)t_0):s\in S',t\in T'\}$. Reversing the affine transformation $\Lambda$ in \eqref{eqn:Lambda_conical}, in the original coordinate space of $(s,r(s)t)$, the above parallelogram becomes
\begin{equation*}
    \{(s,r(s)t):s\in S',t\in T'\}.
\end{equation*}
In this way, we have decoupled each $\Pi_{\delta_{i}}(S\times 2T)$ into almost parallelograms $\Pi_{\delta_{i}}(S'\times T')$, whose collection is denoted by $\mathcal P_{\delta_i}(S\times 2T)$, and whose overlap function is independent of $\delta$. The last step is to define
\begin{equation*}
    \mathcal P_{\delta_i}:=\bigcup_{S\in \mathcal S_{\delta_{i-1}},T\in \mathcal T_{\delta_{i-1}}}\mathcal P_{\delta_i}(S\times 2T),
\end{equation*}
which covers $\Pi_\delta([-1,1]^k,[-1,1]^l)$, and whose overlap function is independent of $\delta$ by Proposition \ref{prop:iterative_overlap}, using the fact that $N$ is independent of $\delta$. Also, using Proposition \ref{prop:tiling_containedin_2R0}, we see that the projection of $\cup\mathcal P_{\delta_i}$ onto $\R^{k+l}$ stays within $[-2,2]^{k+l}$ for every $i\le N$. Hence the induction closes.

\subsubsection{Decoupling cost}\label{sec:radial_principle_pigeon}
Recall from Section \ref{sec:decoupling_section_radial} that the cost of decoupling in Step $i$ is $\le C_\eps (\delta_i\delta_{i-1}^{-1})^{-\eps}$. Combining the decoupling costs using Proposition \ref{prop:combine_decoupling}, we see that the total cost of decoupling is 
\begin{equation*}
    \lesssim C_\eps^N \delta^{-\eps} \delta^{-O(\eps)}\sim C'_\eps \delta^{-O(\eps)},
\end{equation*}
using $N=N(\eps)$. This finishes the proof of the general case.

\subsubsection{Special cases}
We now prove the special cases. Part \eqref{item:Mar_25_radial_T} is trivial, so it suffices to prove Part \eqref{item:Mar_25_radial_S}. The main goal is to $r$-decouple $[-1,1]^k$ into the given family $\tilde {\mathcal S}_\delta$. Similarly to the argument in Section \ref{sec:trivial_decoupling}, by an elementary Fourier truncation argument and losing $\delta^{-O(\eps)}$, we may assume without loss of generality that each $S\in \tilde {\mathcal S}_{\delta_1}$ has diameters $\lesssim \delta_1^{1/2}$. Then at each step $2\le i\le N$ of the iteration, by assumption, we can always use members of $\tilde {\mathcal S}_{\delta_i}$ to $r$-decouple $[-1,1]^k$. In particular, taking $i=N$, we can use members of $\tilde {\mathcal S}_{\delta}$ to $r$-decouple $[-1,1]^k$ at scale $\delta_N=\delta$. Using Proposition \ref{prop:iterative_overlap}, the final overlap function and the decoupling cost will additionally depend on the constant $C_r$.

\section{The degeneracy locating principle}\label{sec:degeneracy_locating}
In this section, we introduce and prove the degeneracy locating principle of decoupling. We need a little notation to make this principle more general.
\subsection{The setup}
Let $1\le k\le n-1$. Let $V,W\sub \R^k$ be compact sets whose interiors are nonempty and connected. Let $\Phi:V\to \R^n$, $\Psi:W\to \R^n$ be parametrisations of $C^2$ manifolds. (Abusing notation, we identify $\Phi,\Psi$ with the manifolds they are parametrising.)

By the formulation of decoupling in Definition \ref{defn:decoupling}, it enjoys a fundamental property called affine invariance. Namely, if $\Lambda:\R^n\to \R^n$ is an affine bijection, and $\Lambda(\mathcal R):=\{\Lambda(R):R\in \mathcal R\}$, then one has 
\begin{equation*}
    \mathrm{Dec}(\Lambda(S),\Lambda(\mathcal R),p,q,\alpha)=\mathrm{Dec}(S,\mathcal R,p,q,\alpha).
\end{equation*}
To utilise this property, below we systematically develop the theory of affine equivalence for manifolds in $\R^n$.

\begin{defn}[Affine equivalence of manifolds in parametric form]\label{defn:affine_equivalence_parametric}
    We say $\Phi$ and $\Psi$ are affine equivalent in parametric form if there exist an affine bijection $\Lambda:\R^n\to \R^n$ and a diffeomorphism $\Xi:W\to V$, such that 
    \begin{equation}\label{eqn:affine_equivalence_parametric}
        \Lambda\circ\Psi=\Phi\circ \Xi.
    \end{equation}
\end{defn}
We give some examples below. 

\fbox{Notation} Throughout this section, we use row vectors by default, but switch to column vectors whenever we need to demonstrate matrix multiplications. 
\begin{itemize}
    \item Consider $\Psi_1(x)=(x,x^2)$ on $W_1=[1,1+\delta^{1/2}]$, and $\Phi_1(x)=(x,\delta x^2)$ on $V_1=[0,1]$. We can define $\Xi_1 x=\delta^{-1/2}(x-1)$ and $\Lambda_1 (x,y)=(\Xi x,y-2x+1)$, such that \eqref{eqn:affine_equivalence_parametric} holds.
    \item Consider $\Psi_2(x)=(x,x^2)$ on $W_2=[1,2]$, and $\Phi_2(x)=(x,\sqrt x)$ on $V_2=[1,4]$, $\Xi_2 x=x^2$ and $\Lambda_2(x,y)=(y,x)$, such that \eqref{eqn:affine_equivalence_parametric} holds. Note that the diffeomorphism $\Xi_2$ is a reparametrisation of $\Phi$, and is not always affine.
    \item Consider $\Psi_3(x)=(x,x^2,x^3)$ on $W_3=[1,1+\delta^{1/3}]$, and $\overline\Phi_3(x)=(x,\delta^{2/3}x^2,\delta x^3)$ on $V_3=[0,1]$. We can define $\Xi_3 x=\delta^{-1/3}(x-1)$ and 
    \begin{equation*}
        \overline\Lambda_3 (x,y,z)=(\Xi_3 x,-2x+y+1,3x-3y+z-1),
    \end{equation*}
    such that \eqref{eqn:affine_equivalence_parametric} holds with $\overline \Phi_3,\overline \Lambda_3$. We can also normalise by defining
    \begin{equation*}
    \begin{aligned}
        &\Phi_3(x)=\left(x,\frac{10}{\sqrt {10}}\delta^{2/3}x^2+\frac 3 {\sqrt {10}}\delta x^3,\frac 1 {\sqrt {10}}\delta x^3\right)^T,\\
        &\Lambda_3\begin{bmatrix}
            x \\ y \\z
        \end{bmatrix}=
        \begin{bmatrix}
            \delta ^{-1/3} & 0 & 0\\
            -\frac {11}{\sqrt {10}} & \frac 1 {\sqrt {10}} & \frac 3 {\sqrt {10}}\\[4pt]
            \frac 3 {\sqrt {10}} & -\frac 3 {\sqrt {10}} & \frac 1 {\sqrt {10}}
        \end{bmatrix}\begin{bmatrix}
            x \\ y \\z
        \end{bmatrix}
        +\begin{bmatrix}
            -\delta^{-1/3}\\ \frac 7 {\sqrt {10}}\\[4pt] -\frac 1 {\sqrt{10}}
        \end{bmatrix},
    \end{aligned}
    \end{equation*}
    such that \eqref{eqn:affine_equivalence_parametric} holds with $ \Phi_3, \Lambda_3$. Note that the lower right $2\times 2$ minor $U_3:=\begin{bmatrix}
        \frac 1 {\sqrt {10}} & \frac 3 {\sqrt {10}}\\[4pt]
        -\frac 3 {\sqrt {10}} & \frac 1 {\sqrt {10}}
    \end{bmatrix}$
    is an orthonormal matrix.
\end{itemize}

\subsubsection{Affine equivalence of manifolds in graph form}
In many of our applications of decoupling, we would like that our manifolds $\Phi$ be given in graph form. 

\fbox{Notation} For convenience, we also restrict ourselves to a subcollection $\mathcal{M}_0=\mathcal M_0(n,k,C_0)$ of manifolds $M_\phi$ in $\R^n$ given by graphs of vector valued $C^2$ functions $\phi: [-1,1]^k \to \R^{n-k}$, such that $\norm{\phi}_{C^2}\le C_0$. Denote $\Phi(x)=(x,\phi(x))$. 

We also denote by $\mathcal R_0$ the collection of all parallelograms contained in $[-1,1]^k$ with positive $k$-volume.

\begin{defn}[Affine equivalence of manifolds in graph form]\label{defn:affine_equivalence_graphic}
    Let $M_\phi,M_\psi\in \mathcal M_0$, and let $R_\phi,R_\psi\in \mathcal R_0$. Let $\Xi:R_\psi\to R_\phi$ be an affine bijection. We say that $(M_\psi,R_\psi)$ is $\Xi$-affine equivalent to $(M_\phi,R_\phi)$, or in symbols,
    $$
    (M_\psi, R_\psi) \overset{\Xi}{\longmapsto} (M_\phi,R_\phi),
    $$
    if there exist an orthonormal matrix $U:\R^{n-k} \to \R^{n-k}$, a matrix $A:\R^k\to \R^{n-k}$ and a vector $b\in \R^{n-k}$, such that
    \begin{equation}\label{eqn:affine_equivalence_graphic}
        \psi(x) = U ( \phi(\Xi x) + Ax + b), \quad \forall x \in R_\psi.
    \end{equation} 
    Typically, $\Xi$ is clear when the pairs $(M_\phi,R_\phi)$ and $(M_\psi,R_\psi)$ are given. We therefore simply say that $(M_\phi,R_\phi)$ and $(M_\psi,R_\psi)$ are affine equivalent (in graph form).
\end{defn}
For the examples after Definition \ref{defn:affine_equivalence_parametric}, we have
\begin{itemize}
    \item $(M_{\psi_1},R_{\psi_1})$ is $\Xi_1$-affine equivalent to $(M_{\phi_1},R_{\phi_1})$ with $U=1$, $A=-2$ and $b=1$; 
    \item $(M_{\psi_2},R_{\psi_2})$ is not affine equivalent to $(M_{\phi_2},R_{\phi_2})$ in graph form. 
    \item $(M_{\psi_3},R_{\psi_3})$ is $\Xi_3$-affine equivalent to $(M_{\phi_3},R_{\phi_3})$ with $U=U_3^{-1}$, $A=(\frac {11}{\sqrt{10}}, -\frac 3 {\sqrt{10}})^T$, and $b=(-\frac {7}{\sqrt{10}}, \frac 1 {\sqrt{10}})^T$.
\end{itemize}


We summarise the following useful properties of affine equivalence for future reference.

\begin{lem}\label{lem:AE_property}
Let $M_\phi,M_\psi\in \mathcal M_0$, and let $R_\phi,R_\psi\in \mathcal R_0$. Then we have the following statements.
\begin{enumerate}
        \item \label{item:lemma3.3_part1} If $(M_\psi,R_\psi)$ is $\Xi$-affine equivalent to $(M_\phi,R_\phi)$, then for any parallelogram $R \subseteq R_\psi$, $(M_\psi,R)$ is $\Xi$-affine equivalent to $(M_\phi,\Xi(R))$.
        \item \label{item:lemma3.3_part2} If $(M_{\psi_0},R_{\psi_0})$ is $\Xi_1$-affine equivalent to $(M_{\psi_1},R_{\psi_1})$, and $(M_{\psi_1},R_{\psi_1})$ is $\Xi_2$-affine equivalent to $(M_{\psi_2},R_{\psi_2})$, then $(M_{\psi_0},R_{\psi_0})$ is $(\Xi_2\circ \Xi_1)$-affine equivalent to $(M_{\psi_2},R_{\psi_2})$.
        \item \label{item:lemma3.3_part3} If $(M_\psi,R_\psi)$ is $\Xi$-affine equivalent to $(M_\phi,R_\phi)$, then $R_\psi$ is $(\psi,\sigma)$-flat if and only if $R_\phi$ is $(\phi,\sigma)$-flat, in any of the three senses in Definition \ref{defn:new_flatness}. 
        \item \label{item:lemma3.3_part4} If $(M_\psi,R_\psi)$ is $\Xi$-affine equivalent to $(M_\phi,R_\phi)$, then $(M_{U\psi},R_\psi)$ is $\Xi$-affine equivalent to $(M_\phi,R_\phi)$ for any $U \in O(n-k)$.
    \end{enumerate}
\end{lem}

\begin{proof}
    Part \eqref{item:lemma3.3_part3} follows immediately from Proposition \ref{prop:new_flatness_affine_invariance}. The other parts follow directly from Definition \ref{defn:affine_equivalence_graphic}.
\end{proof}

It will be useful to study when these two notions of affine equivalence in Definitions \ref{defn:affine_equivalence_parametric} and \ref{defn:affine_equivalence_graphic} are equivalent when $\Phi,\Psi$ are given by graphs. One direction is obvious: if $(M_\psi,R_\psi)$ and $(M_\phi,R_\phi)$ are affine invariant in graph form via \eqref{eqn:affine_equivalence_graphic}, then $\Psi:=(x,\psi(x))$ and $\Phi:=(x,\phi(x))$ are affine invariant in parametric form via \eqref{eqn:affine_equivalence_parametric}, where $\Lambda(x,y):=(\Xi x,U^{-1}y-Ax-b)$. What we will often need in the future is the converse statement below.

\begin{prop}
Let $U_\Phi, U_\Psi \subseteq \R^k$ be compact subsets whose interiors are open and connected. Let $\Phi \stackrel{\mathrm{def}}{=} (\bar \Phi,\tilde \Phi) :U_\Phi\to \R^n$, $\Psi \stackrel{\mathrm{def}}{=} (\bar \Psi,\tilde \Psi):U_\Psi\to \R^n$ be $C^2$ parametrisations such that $D\bar\Phi$ and $D\bar \Psi$ are $k\times k$ matrices whose eigenvalues have their absolute value $\sim 1$. Let $\phi = \tilde \Phi \circ (\bar\Phi)^{-1}$ and $\psi = \tilde \Psi \circ (\bar\Psi)^{-1}$ so that $\Phi=M_\phi$, $\Psi=M_\psi$ are both elements of $\mathcal M_0$.

Assume that $\Phi$ and $\Psi$ are affine equivalent in parametric form via \eqref{eqn:affine_equivalence_parametric}, that is, $\Lambda\circ \Psi=\Phi\circ \Xi$. Assume also that there exists an orthonormal matrix $Q:\R^{n-k}\to \R^{n-k}$ such that    \begin{equation}\label{eqn:condition_two_affine_equivalent}
        \Lambda(x,y)-\Lambda (x,0)=(0,Qy),\quad \forall x\in \R^k, \quad y\in \R^{n-k}.
    \end{equation}
Denote $\Lambda=(\bar \Lambda,\tilde \Lambda)$, where $\bar \Lambda$ denotes the first $k$ coordinates of $\Lambda$.     
Then we have
\begin{enumerate}
    \item \label{item:Mar_25_01} $\bar\Lambda (x,y)=\bar \Lambda(x,0)$. Thus, we may abuse notation and assume that $\bar \Lambda$ is defined on $\R^k$.
    \item \label{item:Mar_25_02} For any parallelogram $R_\psi\sub \bar\Psi(U_\Psi)$, $(M_\psi,R_\psi)$ is $\bar\Lambda$-affine equivalent to $(M_\phi,\bar \Lambda(R_\phi))$ are affine invariant in graph form. (See Figure \ref{fig:affine_equivalence} below.)
\end{enumerate}
\end{prop}
\begin{figure}[h]
    \centering
        \[
        \begin{tikzcd}[column sep = small, row sep = large]
        \bar \Psi^{-1}(R_\Psi)\arrow[ddr, out=-150, in=-180, "\tilde \Psi", near end, swap] \arrow{rr}{\Xi} \arrow[rightarrow]{d}{\bar\Psi} & & \bar \Phi^{-1}(R_\Phi) \arrow[rightarrow]{d}{\bar\Phi} \arrow[ddl, out=-30, in=0, "\tilde \Phi", near end] \\
        R_\Psi \arrow{rr}{\bar\Lambda} \arrow{dr}{\psi} & & R_\Phi \arrow{dl}[swap]{\phi}\\
        & \R^{n-k}  &
        \end{tikzcd}
        \]
        \caption{Diagram of affine equivalence.}
        \label{fig:affine_equivalence}
\end{figure}

\begin{proof}
    By \eqref{eqn:condition_two_affine_equivalent}, we have
    \begin{align}
        &\bar\Lambda (x,y)=\bar\Lambda (x,0),\label{eqn:Mar_25_01}\\
        &\tilde \Lambda (x,y)=\tilde \Lambda(x,0)+Qy,\label{eqn:Mar_25_02}
    \end{align}
    and so Part \eqref{item:Mar_25_01} follows. To prove Part \eqref{item:Mar_25_02}, by Part \ref{item:lemma3.3_part1} of Lemma \ref{lem:AE_property}, it suffices to prove the case $R_\psi=\bar \Psi(U_\Psi)$. But by \eqref{eqn:affine_equivalence_parametric}, we have
    \begin{equation}\label{eqn:Oct_04}
        \Lambda(\bar \Psi,\psi\circ \bar \Psi)=\Lambda (\bar \Psi,\tilde \Psi)=(\bar \Phi\circ \Xi,\tilde \Phi\circ \Xi )=(\bar \Phi\circ \Xi,\phi\circ \bar \Phi\circ \Xi).
    \end{equation}
    By \eqref{eqn:Mar_25_01} and \eqref{eqn:Oct_04}, we have
    \begin{equation}\label{eqn:Mar_25_03}
        \bar \Lambda\circ \bar \Psi=\bar \Phi\circ \Xi.
    \end{equation}
    By \eqref{eqn:Mar_25_02} and \eqref{eqn:Oct_04}, we have
    \begin{equation}\label{eqn:Mar_25_04}
        Q(\psi\circ \bar \Psi)=\phi\circ \bar \Phi \circ \Xi-\tilde \Lambda(\bar \Psi,0).
    \end{equation}
    Denote $U=Q^{-1}$ and $\tilde \Lambda (x,0)=-Ax-b$. By \eqref{eqn:Mar_25_03} and \eqref{eqn:Mar_25_04}, we have    \begin{equation*}
        \psi\circ \bar \Psi=U(\phi\circ \bar \Phi \circ \Xi+A\bar \Psi+b)=U(\phi\circ  \bar\Lambda\circ \bar \Psi+A\bar \Psi+b),
    \end{equation*}
    that is,
    \begin{equation*}
        \psi(x)=U(\phi  (\bar\Lambda x)+Ax+b).
    \end{equation*}
    Thus, \eqref{eqn:affine_equivalence_graphic} follows with $\Xi$ replaced by $\bar \Lambda$.
\end{proof}

\smallskip

\fbox{Notation} Throughout the rest of this section, we restrict ourselves to the graphic case. We fix $1\le k\le m\le n-1$, and all flatness conditions are stated in dimension $m$ in the \underline{first} sense as in Definition \ref{defn:new_flatness}. 

\subsubsection{Affine invariance of decoupling}

Now we study the interaction between decoupling and affine equivalence. We recall the terminology introduced in Definitions \ref{defn:decoupling} and \ref{defn:graphical_decoupling_convex_hull} as well as the vertical neighbourhood notation introduced in \eqref{eqn:vertical_nbhd}.

\fbox{Notation} The exponents $p,q,\alpha$ in the $(\ell^q(L^p),\alpha)$ decoupling will be fixed and omitted throughout the rest of this section.

\begin{lem}\label{lem:equivalent_decoupling}
    Suppose that $M_\phi, M_\psi \in \mathcal{M}_0$, and assume that $R_\phi,R_\psi\in \mathcal R_0$ are parallelograms such that the pair $(M_\phi,R_\phi)$ is $\Xi$-affine equivalent to $(M_\psi,R_\psi)$. For $\vec{\sigma} \in \R^{n-k}$ with $\max_{i}|\sigma_i| \leq 1$ and $\sigma := \min_{i} |\sigma_i|>0$, denote by $\vec{\sigma}\psi$ the $\R^{n-k}$-valued function 
    \begin{equation*}
        (\sigma_1\psi_1,\dots,\sigma_{n-k} \psi_{n-k}).
    \end{equation*}
    For $\delta>0$, the following two statements are equivalent:
    \begin{enumerate}
        \item \label{item:Mar_3_01} $R_\phi$ can be $\phi$-decoupled into parallelograms $\omega$ at scale $\delta$ at cost $C$.
        \item \label{item:Mar_3_02} $R_\psi$ can be $\psi$-decoupled into parallelograms $\Xi(\omega)$ at scale $\delta$ at cost $C$.
     \end{enumerate}
     Either statement above implies:
    \begin{enumerate}
    \setcounter{enumi}{2}
        \item \label{item:Mar_3_03} $R_\psi$ can be $\vec \sigma\psi$-decoupled into parallelograms $\Xi(\omega)$ at scale $\sigma\delta$ at cost $C$.
    \end{enumerate}

     
    \end{lem}



Now we give a proof of Lemma \ref{lem:equivalent_decoupling}. It uses Lemma \ref{lem:decoupling_entire_strip}, which is purely technical and is thus put in the appendix.

\begin{proof}[Proof of Lemma \ref{lem:equivalent_decoupling}]
Define $\Lambda:\R^n\to \R^n$ by
    \begin{equation*}
        \Lambda (x,y):=(\Xi x,y),\quad x\in \R^{k},\,y\in \R^{n-k}.
    \end{equation*} 
Then $\Lambda$ maps $N^\psi_\delta(\omega\cap R_\psi)$ bijectively onto $N^\phi_\delta(\Xi(\omega)\cap R_\phi)$. Thus, the equivalence of parts \eqref{item:Mar_3_01}\eqref{item:Mar_3_02} is clear from Lemma \ref{lem:decoupling_entire_strip}. 

By Lemma \ref{lem:decoupling_entire_strip} again, proving Part \eqref{item:Mar_3_03} is equivalent to proving that $N^{\vec \sigma \psi}_{\sigma \delta}(R_\psi)$ can be decoupled into $\Xi(\omega)\times \R^{n-k}$. But by an anisotropic rescaling in the last $n-k$ variables, it is equivalent to proving that the subset
\begin{equation*}
    \{(x,y):x\in R_\psi:|y_i-\psi_i(x)|\le \sigma_i^{-1}\sigma \delta\quad \forall 1\le i\le n-k\}
\end{equation*}
can be decoupled into $\Xi(\omega)\times \R^{n-k}$. But by Part \eqref{item:Mar_3_02} and Lemma \ref{lem:decoupling_entire_strip} again, we see that $N^\psi_\delta(R_\psi)$ can be decoupled into $\Xi(\omega)\times \R^{n-k}$. Then the claim follows from the easy observation that
\begin{equation*}
    \{(x,y):x\in \Xi(\omega):|y_i-\psi_i(x)|\le \sigma_i^{-1}\sigma \delta\quad \forall 1\le i\le n-k\}\sub N^\psi_\delta(\Xi(\omega)),
\end{equation*}
since $\sigma_i\ge \sigma$ for every $i$.
\end{proof}

\begin{defn}[Rescaling invariant pair]\label{defn:rescaling_invariant_pair}
Let $\M\sub\M_0$ be a subfamily, and for each $M_\phi\in \mathcal M$, let $\mathcal R_\phi\sub \mathcal R_0$ be a subfamily\footnote{In all examples in this paper, we actually have $\mathfrak R=\{\mathcal R\}$ for some $\mathcal R\sub \mathcal R_0$. However, for future applications, we develop a more general theory.}. Let $\mathfrak R=\{\mathcal R_\phi:M_\phi\in \mathcal M\}$ be the collection of all such families $\mathcal R_\phi$.

We say that the pair of families $(\mathcal M,\mathfrak R)$ is rescaling invariant (in dimension $m$) if there exists a constant $C\ge 1$ such that the following holds. For every $\sigma\in (0,1]$, $M_\phi \in \M$ and every  parallelogram $R_\phi\in \mathcal R_\phi$ that is $(\phi,\sigma)$-flat, there exist some $ M_\psi \in \M$, some $\vec \sigma\in [-1,1]^{n-k}$, and some (bounded) affine bijection $\Xi:[-1,1]^k \to R_\phi$,
such that
\begin{equation}\label{eqn:rescaling invariant}
    (M_{\vec \sigma \psi},[-1,1]^k)\overset{\Xi}{\longmapsto} (M_\phi, R_\phi) ,
\end{equation}
where $\vec{\sigma}:=(\sigma_i)$ in addition obeys $\sigma_i\ge C^{-1}\sigma$ for all $1\le i\le n-k$, and $\sigma_i\le C\sigma$ for at least $n-m$ many $1\leq i \leq n-k$.

More generally, let $\mathcal A$ be a family of affine bijections on $\R^k$. We say that the pair $(\mathcal{M},\mathfrak R)$ is $\mathcal{A}$-rescaling invariant if we further assume that the affine bijection $\Xi$ in \eqref{eqn:rescaling invariant} can be taken from $\mathcal{A}$.

\end{defn}

In simple words, every $\sigma$-flat (in dimension $m$) piece of a manifold $M_\phi\in \mathcal M$ over $R\in \mathcal R_\phi$ can be rescaled to become another manifold $M_{\vec \sigma \psi}\in \mathcal M$ over $[-1,1]^k$, where at least $n-m$ components of $\vec \sigma$ are essentially $\sigma$, the smallest component of $\vec \sigma$. Also, by a permutation on the $n-k$ variables in the ``graphical" space, we may assume $\sigma_i\sim \sigma$ for $1\le i\le n-m$.

We illustrate a special case of Definition \ref{defn:rescaling_invariant_pair} and Lemma \ref{lem:equivalent_decoupling} using the following example.

\begin{itemize}
\item 
Consider $m=2$ and the family $\mathcal M$ of curves $M_{\phi_a}$ given by
\begin{equation*}
    \{(s,\phi_{a}(s)):=(s,a^{2/3}s^2,as^3):s\in [-1,1]\},\,\, a\in (0,1],
\end{equation*}
obtained from the standard moment curve $(s,s^2,s^3)$ by a suitable rescaling. Let $\sigma\in (0, 1)$. One can check that an interval $R\sub [-1,1]$ is $(\phi_a,\sigma)$-flat if and only if its translation centred at $0$ is $(\phi_a,\sigma)$-flat. It can also be checked that $R:=[-u^{1/3},u^{1/3}]$ is $(\phi_a,O(\sigma))$-flat in dimension $2$ if and only if $u\lesssim \min\{1,a^{-1}\sigma\}$. 

\noindent Rescaling $R$ to $[-1,1]$ gives the curve
    \begin{equation*}
        \{(s,(au)^{2/3}s^2,au s^3):s\in [-1,1]\},
    \end{equation*}    
    which corresponds to $\vec \sigma \phi_{a'}$, where $a'=au\sigma^{-1}$ and $\vec \sigma:=(\sigma^{2/3},\sigma)$ so that there is exactly $n-m=3-2=1$ entry of $\vec \sigma$ equal to $\sigma=\min_i \sigma_i$. Therefore, we can define the singleton family $\mathfrak R=\{\mathcal R_0\}$, such that the pair $(\mathcal M,\mathfrak R)$ is rescaling invariant. In this case, Lemma \ref{lem:equivalent_decoupling} is exactly the key rescaling property of the moment curve used in \cite{BDG2016}.

\item Consider $m=1$ and the family $\mathcal M$ of curves $M_{\phi_{a,b}}$ given by
\begin{equation*}
    \{(s,\phi_{a,b}(s)):=(s,as^2,bs^3):s\in [-1,1]\},\,\, a\in (0,1], \,\,b\in [0,a^{3/2}].
\end{equation*}
Let $\sigma\in (0, 1)$. One can check that an interval $R\sub [-1,1]$ is $(\phi_{a,b},\sigma)$-flat if and only if its translation centred at $0$ is $(\phi_{a.b},\sigma)$-flat. It can also be checked that $R:=[-u^{1/2},u^{1/2}]$ is $(\phi_{a,b},O(\sigma))$-flat in dimension $1$ if and only if $u\lesssim \min\{1,a^{-1}\sigma\}$.

\noindent Rescaling $R$ to $[-1,1]$ gives the curve
    \begin{equation*}
        \{(s,aus^2,bu^{3/2} s^3):s\in [-1,1]\},
    \end{equation*}    
    which corresponds to $\vec \sigma \phi_{a',b'}$, where $a':=au\sigma^{-1}$, $b':=bu^{3/2}\sigma^{-1}$, and $\vec \sigma:=(\sigma,\sigma)$, so that there are exactly $n-m=3-1=2$ entries of $\vec \sigma$ equal to $\sigma=\min_i \sigma_i$. We also have $b'\le (a')^{3/2}$. Therefore, we can define the singleton family $\mathfrak R=\{\mathcal R_0\}$, such that the pair $(\mathcal M,\mathfrak R)$ is rescaling invariant. 

\noindent The above discussion is relevant if we want to study the decoupling for the moment curve $(s,s^2,s^3)$, but unlike the canonical case in \cite{BDG2016}, we ignore the torsion given by the third coordinate $s^3$. By the projection principle and decoupling for the parabola in $\R^2$, at scale $\delta$, it gives a $\ell^2(L^6)$ decoupling of $[-1,1]$ into intervals of length $\delta^{1/2}$, which are $\delta$-flat in dimension $m=1$. 
\end{itemize}

\subsubsection{Invariance of family of manifolds $\mathcal{M}$ under particular affine transformations}\label{subsub:invariance_affine_manifolds}

For a given family of manifolds $\mathcal{M}$, it is often helpful to restrict our attention to a special class of affine transformations $\mathcal A\sub \mathcal A_0$ that preserve $\mathcal{M}$. (Here, $\mathcal A_0$ denotes the collection of all affine bijections on $\R^k$ that are bounded by a fixed constant $O(1)$.)

For example, in Corollary \ref{cor:dec_add}, we consider $\mathcal{M}$ to consist of graphs of functions of the form $\sum_{j=1}^J{\phi_j}(x_j)$, where $x_j \in [-1,1]^{n_j}$. It is natural to only consider affine transformations that preserve the structure of $\mathcal{M}$, i.e. those affine transformations of the form $x \mapsto Dx +v$ where $D$ is a block diagonal matrix with only $2\times 2$ or $1\times 1$ (depending on whether $n_j=1$ or $2$) nonsingular blocks on its diagonal.

By restricting ourselves to a special collection of affine transformations $\mathcal{A}\sub \mathcal A_0$ that preserve $\mathcal{M}$, we gain additional information on what the decoupled sets look like. In Corollary \ref{cor:dec_add}, we have a natural choice of $\mathfrak R = \{\mathcal{R}\}$ where $\mathcal R$ consists of parallelograms of the form $R_1\times \dots \times R_J$ where $R_j$ are $n_j$-dimensional parallelograms.

For another example, namely Corollary \ref{cor:radial}, we consider $\mathcal M$ to consist of all functions $\phi(s,t') = P(s) + c\psi(t')$, where $P\in \mathcal P_{1,d}$, $|\det D^2 \psi|\sim 1$, and $|c|\le 1$. The collection $\mathfrak R=\{\mathcal R\}$ where $\mathcal R$ consists of $I\times Q$ where $I\sub [-1,1]$ is an interval and $Q\sub [-1,1]^{n-2}$ is a square. In this case, $\mathcal A$ consists of affine bijections of the form $(s,t')\mapsto (\delta_1 s+c_1,\delta_2 t'+c_2)$, where $\delta_1,\delta_2\in (0,1]$, $c_1\in [-1,1]$, $c_2\in [-1,1]^{n-2}$.

In the general case where $\mathfrak R$ consists of different families $\mathcal R_\phi$, the analysis is more subtle. In principle, we hope that the decoupling for $(M_\phi,R_\phi)$ gives parallelograms over $\mathcal{R}_\phi$.

We now describe the necessary conditions in our method to ensure that the decoupling for $M_\phi$ results in coordinate parallelograms $R_\phi\in \mathcal R_\phi$. (Strictly speaking, to decouple a parallelogram contained in $[-1,1]^k$, in general we need parallelograms contained in a cube slightly larger than $[-1,1]^k$, say $[-2,2]^k$, when there are rotations or shear transformations. We omit this slight technicality for simplicity.) We start with some terminology.

\begin{defn}[Compatibility with $\mathcal A$]\label{defn:compatible}
Let $\mathcal{M}\subseteq \mathcal{M}_0$, $\mathfrak{R}$ be a collection of families $\mathcal{R}_\phi \subseteq \mathcal{R}_0$ and $\mathcal{A}$ be a family of affine bijections on $\R^k$ that is closed under composition, namely, $\Xi_1\circ \Xi_2\in \mathcal A$ whenever $\Xi_1,\Xi_2\in \mathcal A$. We say that $(\mathcal M,\mathfrak R)$ is $\mathcal A$-compatible if the following condition is satisfied:

    For any $M_\phi,M_\psi \in \mathcal{M}$, $\Xi \in \mathcal{A}$, $\vec\sigma\in [-1,1]^{n-k}$ satisfying  \eqref{eqn:rescaling invariant} with $R_\phi=\Xi([-1,1]^k)$, i.e.
    $$
    (M_{\vec \sigma \psi},[-1,1]^k) \overset{\Xi}{\longmapsto} (M_\phi, \Xi([-1,1]^k)),
    $$
    and any $R_\psi \in \mathcal{R}_\psi$, we have $\Xi(R_\psi) \in \mathcal{R}_\phi$.

\end{defn}
{\it Remark.} Note that if $(\mathcal M,\mathfrak R)$ is $\mathcal A$-compatible, and $\Xi\in \mathcal A$, then for every $M_\phi\in \mathcal M$ and every $R\in \mathcal R_{\phi\circ \Xi}$, we have $\Xi(R)\in \mathcal R_\phi$. In fact, this follows if we take $\vec \sigma=(1,\dots,1)$ and $\psi=\phi\circ\Xi$.

The following version of Lemma \ref{lem:equivalent_decoupling} involving $\mathcal A$ will be crucial.

\begin{lem}\label{lem:equivalent_decoupling_A}
Let $(\mathcal M,\mathfrak R)$ be $\mathcal A$-compatible as in Definition \ref{defn:compatible}. Let $M_\phi,M_\psi \in \mathcal{M}$, $\vec\sigma=(\sigma_i)\in \mathbb{R}^{n-k}$ satisfy \eqref{eqn:rescaling invariant} with $R_\phi=\Xi([-1,1]^k)$, i.e.
    $$
    (M_{\vec \sigma \psi},[-1,1]^k)\overset{\Xi}{\longmapsto} (M_\phi, \Xi([-1,1]^k)).
    $$
For $\delta>0$, the following two statements are equivalent:
   \begin{enumerate}
        \item \label{item:Mar_5_01} $R_\phi$ can be $\phi$-decoupled into parallelograms $\omega\in \mathcal R_\phi$ at scale $\delta$ at cost $C$.
        \item \label{item:Mar_5_02} $R_{\vec \sigma\psi}$ can be $\vec \sigma\psi$-decoupled into parallelograms $\Xi^{-1}(\omega)\in \mathcal R_{\psi}$ at scale $\delta$ at cost $C$.
     \end{enumerate}
     If, in addition, $\sigma:=\min_i |\sigma_i|\ge \delta$, then either statement above is implied by:
    \begin{enumerate}
    \setcounter{enumi}{2}
        \item \label{item:Mar_5_03} $R_\psi$ can be $\psi$-decoupled into parallelograms $\Xi^{-1}(\omega)\in \mathcal R_\psi$ at scale $\sigma^{-1}\delta$ at cost $C$.
    \end{enumerate}
    
    \end{lem}
Its proof follows from that of Lemma \ref{lem:equivalent_decoupling}. The rephrased Part \eqref{item:Mar_5_03} is what we will use in the proof of our main Theorem \ref{thm:degeneracy_locating_principle}.

\vspace{1cm}

For the remainder of this subsection before Section \ref{sec:degeneracy_determinant}, we state and prove some results concerned with the construction of families $\mathcal M,\mathfrak R,\mathcal A$ such that $(\mathcal M,\mathfrak R)$ is $\mathcal{A}$-rescaling invariant and $\mathcal A$-compatible. They are not logically necessary for the proof of Theorem \ref{thm:degeneracy_locating_principle} below, but we expect them to be useful in future applications of the principle.

\begin{condition}\label{cond_AM}
    For any $M_\phi \in \mathcal{M}$ and any affine bijection $\lambda:\R^{n-k}\to  \R^{n-k}$ (not necessarily bounded), if $M_{\lambda\phi} \in \mathcal{M}_0$, then $M_{\lambda\phi} \in \mathcal{M}$.
\end{condition}
Condition \ref{cond_AM} means that $\mathcal M$ is invariant under affine bijections on the ``graphical space" $\R^{n-k}$ as long as the transformed function $\lambda\phi$ still has bounded $C^2$ norm. In particular, in the case of hypersurfaces, i.e. $n-k=1$, this condition says that if $M_\phi\in \mathcal M$, then $M_{c\phi}\in \mathcal M$ for any scalar $c\in \R$ provided that $c\phi$ has bounded $C^2$ norm.

{\it Remark.} Definition \ref{defn:rescaling_invariant_pair} implies a weaker version of Condition \ref{cond_AM}. Indeed, choose some constant $C'$ much larger than the constant $C$ in Definition \ref{defn:rescaling_invariant_pair}. If $R$ is $(\phi,\sigma)$-flat, then $R$ is also $(\phi,\sigma^{(j)})$-flat for all $\sigma^{(j)} := (C')^{j}\sigma$, $j\geq 0$. Applying Definition \ref{defn:rescaling_invariant_pair} for each $j\gtrsim \log (\sigma^{-1})$, there exist $\vec {\sigma}^{(j)}$ and $\psi_j \in \mathcal{M}$ so that $(\vec \sigma^{(j)} \psi_j , [-1,1]^{k})$ are affine equivalent to each other. Since the first $n-m$ entries of $\vec{\sigma}^{(j)}$ are comparable to $\sigma^{(j)}$ (with constant $C$), all these $\psi_j$ differ by at a least constant factor. This means that under Definition \ref{defn:rescaling_invariant_pair}, once $M_\phi\in \mathcal M$, one can roughly say that $M_{\lambda\phi}\in \mathcal M$ for every $\lambda\ge 1$. This is implied by the slightly stronger Condition \ref{cond_AM}. Indeed, for simpler arguments later, we impose Condition \ref{cond_AM} to help us handle ``non-optimally" $\sigma$-flat rectangles $R$. However, Condition \ref{cond_AM} might be unnecessarily strong when $n-k \geq 2$. For example, the families given immediately after Definition \ref{defn:rescaling_invariant_pair} do not satisfy Condition \ref{cond_AM}.

\begin{condition}\label{cond:F1F3_equivalent}
    For any $M_\phi\in \mathcal M$ and any $\omega\sub [-1,1]^k$, $\omega$ is $\delta$-flat in the sense of \eqref{eqn:F1} if and only if it is $\sim\delta$-flat in the sense of \eqref{eqn:F3}, where the implicit constant depends on $\mathcal M$ only.
\end{condition}
This condition is guaranteed in many cases; see Propositions \ref{prop:F1F2F3_equivalent} and \ref{prop:squares_flat}.

\begin{lem}\label{lem:cond3.9_implies_rescaling_invariant}
    Let $\mathcal{M}\subseteq \mathcal{M}_0$, $\mathfrak{R}$ be a collection of families $\mathcal{R}_\phi \subseteq \mathcal{R}_0$ and $\mathcal{A}$ be a family of affine bijections on $\R^k$. If $\mathcal{M}$ satisfies Conditions \ref{cond_AM} and \ref{cond:F1F3_equivalent}, and for every $M_\phi\in \mathcal M$ and every rectangle $R_\phi \in \mathcal{R}_\phi$, there exist $M_\psi \in \mathcal{M}$ and $\Xi\in \mathcal{A}$ such that 
    $$
    (M_\psi, [-1,1]^k)\overset{\Xi}\longmapsto(M_\phi,R_\phi),$$
    Then $(\mathcal{M},\mathfrak{R})$ is $\mathcal{A}$-rescaling invariant.
\end{lem}

\begin{proof}
    Let $R_\phi\in \mathcal{R}_\phi$ be $(\phi,\sigma)$-flat in dimension $m$ for some $\sigma\in (0,1]$. By the assumption, there exist $M_{\tilde \psi} \in \mathcal M$ and $\Xi\in \mathcal A$ such that  
    \begin{equation}\label{eqn:Mar_7_lem3.10_01}
    (M_{\tilde \psi}, [-1,1]^k)\overset{\Xi}\longmapsto(M_\phi,R_\phi).
    \end{equation}
    By Part \eqref{item:lemma3.3_part3} of Lemma \ref{lem:AE_property}, $[-1,1]^k$ is $(\tilde \psi,\sigma)$-flat. By Condition \ref{cond:F1F3_equivalent} and Lemma \ref{lem:C2flat_entire}, there exist $U \in O(n-k)$, $\vec \sigma\in \R^{n-k}$ and a bounded affine transformation $L:\R^k\to \R^{n-k}$ such that $U\tilde \psi = \vec \sigma \psi+L$ for some $M_\psi \in \mathcal{M}_0$, and the first $(n-m)$ entries of $\vec \sigma$ are $\sim \sigma\sim \min_i \sigma_i$. Thus, by Part \eqref{item:lemma3.3_part4} of Lemma \ref{lem:AE_property}, we have
    \begin{equation}\label{eqn:Mar_7_lem3.10_02}
        (M_{\vec \sigma\psi},[-1,1]^k)\overset{I_k}\longmapsto (M_{\tilde \psi}, [-1,1]^k),
    \end{equation}
    where $I_k$ denotes the identity map on $\R^k$. Combining \eqref{eqn:Mar_7_lem3.10_01} and \eqref{eqn:Mar_7_lem3.10_02} and using Part \eqref{item:lemma3.3_part2} of Lemma \ref{lem:AE_property}, we thus have \eqref{eqn:rescaling invariant}, i.e.
    \begin{equation*}
        (M_{\vec \sigma\psi},[-1,1]^k)\overset{\Xi}\longmapsto (M_\phi,R_\phi).
    \end{equation*}
    Thus, our proof is complete if we can show that $M_\psi\in \mathcal M$. But this follows from Condition \ref{cond_AM}, since $M_\psi\in \mathcal M_0$, $\psi=\vec \sigma^{-1}(U\tilde \psi-L)$ and $M_{\tilde \psi}\in \mathcal M$.
\end{proof}

Now we return to Definition \ref{defn:compatible}. Although it is a key condition that we need, it tells us nothing about how we should construct the families $\mathcal{A}$ and $\mathcal{R}_\phi$ given a family $\mathcal M$ of manifolds. 

\begin{prop}
    Let $\mathcal{M} \subseteq \mathcal{M}_0$ satisfy Conditions \ref{cond_AM} and \ref{cond:F1F3_equivalent}. Let $\mathcal{A} = \mathcal{A}(\mathcal{M})$ be the collection of all (bounded) affine bijections $\Xi:\R^k\to \R^k$ such that for any $M_\psi\in \mathcal{M}$, there exist $M_\phi\in \mathcal{M}$ and $\vec \sigma\in [-1,1]^{n-k}$, such that  \eqref{eqn:rescaling invariant} holds with $R_\phi=\Xi([-1,1]^k)\in \mathcal R_0$, i.e.
    $$
    (M_{\vec \sigma \psi},[-1,1]^k) \overset{\Xi}{\longmapsto} (M_\phi, \Xi([-1,1]^k).
    $$
    
    For $M_\phi \in \mathcal{M}$, let $\mathcal{R}_\phi$ be the collection of all parallelograms $\Xi([-1,1]^{k})$ in $\mathcal{R}_0$ that satisfy \eqref{eqn:rescaling invariant} for some $\Xi \in \mathcal{A}$, $M_\psi \in \mathcal{M}$ and $\vec \sigma\in [-1,1]^{n-k}$. Write $\mathfrak R=\{\mathcal R_\phi:M_\phi\in \mathcal M\}$. Then $(\mathcal{M},\mathfrak{R})$ is $\mathcal A$-rescaling invariant as in Definition \ref{defn:rescaling_invariant_pair} and $\mathcal{A}$-compatible as in Definition \ref{defn:compatible}.

\end{prop}

\begin{proof}
    We first use Lemma \ref{lem:cond3.9_implies_rescaling_invariant} to show that $(\mathcal{M},\mathfrak{R})$ is $\mathcal{A}$-rescaling invariant. Let $M_\phi\in \mathcal M$ and $R_\phi\in \mathcal R_\phi$. By the definition of $\mathcal{R_\phi}$, there exist $\Xi \in \mathcal{A}$, $M_\psi \in \mathcal{M}$ and $\vec \sigma\in [-1,1]^{n-k}$, such that \eqref{eqn:rescaling invariant} holds. Since $M_\psi\in \mathcal M$ and $\vec\sigma\in [-1,1]^{n-k}$, Condition \ref{cond_AM} implies that $M_{\vec \sigma \psi}\in \mathcal M$.

To prove $(\mathcal{M},\mathfrak{R})$ is $\mathcal{A}$-compatible, we first claim that
\begin{equation}\label{eqn:Feb_19_A_transitive}
    \text{If $\Xi_1,\Xi_2\in\mathcal{A}$, then $\Xi_2\circ \Xi_1 \in \mathcal{A}$.}
\end{equation}
To see the claim, let $M_{\psi} \in \mathcal{M}$ be arbitrary, and we want to prove that there exists $M_\phi\in \mathcal M$ and $\vec \sigma\in [-1,1]^{n-k}$ such that
\begin{equation}\label{eqn:Mar_7_00}
    (M_{\vec \sigma \psi},[-1,1]^k) \overset{\Xi_2\circ \Xi_1}{\longmapsto} (M_\phi, \Xi_2\circ \Xi_1 ([-1,1]^k)).
\end{equation}
But $\Xi_1\in \mathcal A$, so there exist $M_{\phi_1}\in \mathcal M$ and $\vec \sigma_1\in [-1,1]^{n-k}$ such that 
\begin{equation}\label{eqn:Mar_7_01}
    (M_{\vec \sigma_1 \psi},[-1,1]^k) \overset{\Xi_1}{\longmapsto} (M_{\phi_1}, \Xi_1([-1,1]^k)).
\end{equation}
Since $\Xi_2\in \mathcal A$, with $\psi$ in the definition of $\mathcal A$ taken to be $\phi_1$, there exist $M_{\phi_2}\in \mathcal M$ and $\vec \sigma_2\in [-1,1]^{n-k}$ such that 
$$
    (M_{\vec \sigma_2 \phi_1},[-1,1]^k) \overset{\Xi_2}{\longmapsto} (M_{\phi_2}, \Xi_2([-1,1]^k)).
$$
In particular, by Part \eqref{item:lemma3.3_part1} of Lemma \ref{lem:AE_property}, we have
\begin{equation}\label{eqn:Mar_7_02}
    (M_{\vec \sigma_2 \phi_1},\Xi_1([-1,1]^k)) \overset{\Xi_2}{\longmapsto} (M_{\phi_2}, \Xi_2\circ \Xi_1([-1,1]^k)).
\end{equation}
Also, \eqref{eqn:Mar_7_01} implies that
\begin{equation}\label{eqn:Mar_7_03}
    (M_{\vec \sigma_2\vec \sigma_1 \psi},[-1,1]^k) \overset{\Xi_1}{\longmapsto} (M_{\vec \sigma_2\phi_1}, \Xi_1([-1,1]^k)).
\end{equation}
Combining \eqref{eqn:Mar_7_02} and \eqref{eqn:Mar_7_03} and using Part \eqref{item:lemma3.3_part2} of Lemma \ref{lem:AE_property} gives
\begin{equation*}
    (M_{\vec \sigma_2\vec \sigma_1 \psi},[-1,1]^k) \overset{\Xi_2\circ\Xi_1}{\longmapsto}(M_{\phi_2}, \Xi_2\circ \Xi_1([-1,1]^k)).
\end{equation*}
Therefore, setting $\vec \sigma=\vec \sigma_2\vec \sigma_1$ (entrywise product) and $\phi=\phi_2$, we have finished the proof of \eqref{eqn:Mar_7_00} and hence \eqref{eqn:Feb_19_A_transitive}.

Now we come to the proof of compatibility. By definition, we need to prove that for any $M_\phi,M_\psi \in \mathcal{M}$, $\Xi \in \mathcal{A}$, $\vec\sigma\in \R^{n-k}$ satisfying
    \begin{equation}\label{eqn:Mar_7_prop3.11_01}
    (M_{\vec \sigma \psi},[-1,1]^k) \overset{\Xi}{\longmapsto} (M_\phi, \Xi([-1,1]^k)),
    \end{equation}
    and any $R_\psi \in \mathcal{R}_\psi$, we have $\Xi(R_\psi) \in \mathcal{R}_\phi$.

But $\Xi(R_\psi) \in \mathcal{R}_\phi$ means that there exist $\Xi'\in \mathcal A$, $M_{\psi'}\in \mathcal M$, $\vec \sigma'\in [-1,1]^{n-k}$ such that $\Xi(R_\psi)=\Xi'([-1,1]^k)$, and 
\begin{equation*}
    (M_{\vec \sigma' \psi'},[-1,1]^k)\overset{\Xi'}{\longmapsto}  (M_\phi,\Xi'([-1,1]^k)).
\end{equation*}
But since $R_\psi\in \mathcal R_\psi$, there exist $\Xi_1\in \mathcal A$, $M_{\psi_1}\in \mathcal M$, $\vec \sigma_1\in [-1,1]^{n-k}$, such that $R_\psi=\Xi_1([-1,1]^k)$, and
\begin{equation*}
    (M_{\vec \sigma_1 \psi_1},[-1,1]^k)\overset{\Xi_1}{\longmapsto} (M_\psi,\Xi_1([-1,1]^k)).
\end{equation*}
This implies that
\begin{equation}\label{eqn:Mar_7_prop3.11_02}
    (M_{\vec \sigma\vec \sigma_1 \psi_1},[-1,1]^k)\overset{\Xi_1}{\longmapsto} (M_{\vec \sigma\psi},\Xi_1([-1,1]^k)).
\end{equation}
Also, by Part \eqref{item:lemma3.3_part1} of Lemma \ref{lem:AE_property}, \eqref{eqn:Mar_7_prop3.11_01} implies
\begin{equation}\label{eqn:Mar_7_prop3.11_03}
    (M_{\vec \sigma \psi},\Xi_1[-1,1]^k) \overset{\Xi}{\longmapsto} (M_\phi, \Xi\circ \Xi_1([-1,1]^k)).
\end{equation}
Combining \eqref{eqn:Mar_7_prop3.11_02} and \eqref{eqn:Mar_7_prop3.11_03} using Part \eqref{item:lemma3.3_part2} of Lemma \ref{lem:AE_property}, we have
\begin{equation*}
    (M_{\vec \sigma\vec \sigma_1 \psi_1},[-1,1]^k)\overset{\Xi\circ\Xi_1}{\longmapsto} (M_\phi, \Xi\circ \Xi_1([-1,1]^k)).
\end{equation*}
By \eqref{eqn:Feb_19_A_transitive}, $\Xi\circ \Xi_1\in \mathcal A$. Thus, \eqref{eqn:Mar_7_prop3.11_01} follows if we define $\vec \sigma'=\vec \sigma\vec \sigma_1$ (entrywise) and $\Xi'=\Xi\circ \Xi_1$. This finishes the proof.

\end{proof}

We would like to comment that the proposition gives an explicit and convenient construction of the families $\mathcal{A}$ and $\mathfrak{R}$. It can be verified directly that the families in Corollaries \ref{cor:radial} to \ref{cor:smooth_curve} can be  constructed likewise.

\subsubsection{Degeneracy determinant}\label{sec:degeneracy_determinant}
Given $\mathcal{M}$, $\mathfrak{R}$, we will define a function acting on $\mathcal M$, called the {\it degeneracy determinant}, that detects how far a manifold in $\mathcal M$ is from being ``degenerate".

\begin{defn}[Degeneracy determinant]\label{defn:degeneracy_determinant}
    Let $\mathcal{M}\subseteq \mathcal{M}_0$. We say that an operator $H$ on $\mathcal M$ that maps $M_\phi \in \M$ to a smooth function on $[-1,1]^k$ is a {\rm degeneracy determinant}, if there exist constants $C,C'\ge 1$ and $\beta\in (0,1]$ depending only on $\mathcal M$, such that the following holds:

    \begin{enumerate}
        \item \label{item:Lipschitz} (Lipschitz) For every $M_\phi\in \mathcal M$, 
        \begin{equation}\label{eqn:Lipschitz}
            |HM_\phi(x)-HM_\phi(y)|\le C|x-y|,\quad \forall x,y\in [-1,1]^k.
        \end{equation}
        
        \item \label{item:totally_degenerate_approximation} (Totally degenerate approximation)
        Let $M_\phi \in \M$ be such that $|H M_\phi(x)|\le \sigma$ on $[-1,1]^k$ for some $\sigma\in (0,1]$. Then there exists a manifold $M_{\psi}\in \mathcal M$ satisfying $HM_\psi\equiv 0$ on $[-1,1]^k$, and that            \begin{equation}\label{eqn:small_hessian_error}
            \|\phi-C'\psi\|_{L^\infty([-1,1]^k)}\leq C\sigma^{\beta}.
        \end{equation}
        \end{enumerate}
\end{defn}
In many cases, we also require that the operator $H$ is compatible with the family $\mathfrak R$ in the following sense.
\begin{defn}[Regularity of degeneracy determinant]\label{defn:degeneracy_determinant_reg}
    Let $\mathcal{M}\subseteq \mathcal{M}_0$, and $\mathfrak{R}$ be a collection of families $\mathcal{R}_\phi \subseteq \mathcal{R}_0$. We say that a degeneracy determinant $H$ on $\mathcal M$ is $\mathfrak R$-regular, if there exists a constant $C''\ge 1$ depending only on $\mathcal M,\mathfrak R$, such that the following holds:

    \begin{enumerate}

        \item \label{item:containment} For $M_\phi,M_\psi\in \mathcal M$ mentioned in Part \eqref{item:totally_degenerate_approximation} of Definition \ref{defn:degeneracy_determinant}, we have $\mathcal R_\psi\sub \mathcal R_\phi$.

        \item \label{item:regularity_small_Hessian_rescaling} (Regularity under rescaling) Let $R\in \mathcal R_\phi$. Then we have
        \begin{equation}\label{eqn:regularity_rescaling_determinant}
            \mu^{C''}|HM_{\phi}(\lambda_{R}x)|\le |HM_{\phi\circ \lambda_{R}}(x)|\le |HM_{\phi}(\lambda_{R}x)|,\quad \forall x\in [-1,1]^k,
        \end{equation}
    where $\mu$ is the width of $R$.
        \end{enumerate} 

\end{defn}
\textit{Remarks:} 
\begin{itemize}
\item The definition of degeneracy determinants is independent of rescaling invariance or compatibility under affine transformations.

    \item Under the current setting, i.e. manifolds represented by graphs, we typically pick $\mathcal{R}_\phi$ to be a single $\mathcal{R}$ for all $\phi$. However, for manifolds in parametric form, rectangles in the coordinate space depend largely on the parameterisations of individual manifolds. When approximating $\phi$ by $\psi$, it is natural to require that their parameterisations are compatible, in the sense that the interesting collections of rectangles $\mathcal{R}_\phi, \mathcal{R}_\psi$ are the same. In fact, it turns out that a slightly weaker assumption $\mathcal R_\psi\sub \mathcal R_{\phi}$ suffices, as can be seen in the proof of the main theorem below.
\end{itemize}

In many applications, it will be important to check whether a mapping $H$ is a degeneracy determinant. For this purpose, we state the following proposition, whose proof will be given in Section \ref{sec:degen_approximation_theorem} below.
\begin{prop}\label{prop:degeneracy_determinant_Lojasiewicz}
Let $\mathcal M$ consist of graphs $\phi=(\phi_i)_{i=1}^{n-k}$ of polynomials $\phi_i\in \mathcal P_{k,d}$ (see Section \ref{sec:notation} for the notation), and let $\mathfrak R=\{\mathcal R\}$ for some $\mathcal R\sub \mathcal R_0$. Let $H:\mathcal M\to \mathcal P_{k,D}^0$ be a mapping obeying the following property: If $HM_\phi \equiv 0$ and $c\in \R$, then $HM_{c \phi} \equiv 0 $. Then $H$ is a degeneracy determinant on $\mathcal M$, with the constants $C,C',\beta$ depending only on $H$.
\end{prop}
Using Proposition \ref{prop:degeneracy_determinant_Lojasiewicz} and the fact that $\det D^2 \phi$ satisfies \eqref{eqn:regularity_rescaling_determinant}, we obtain the following corollaries:
\begin{cor}\label{cor:hessian}
    Let $\mathcal M$ consist of graphs of polynomials $\phi\in \mathcal P_{n-1,d}$, and let $\mathfrak R=\{\mathcal R_0\}$. Then $HM_\phi:=\det D^2\phi$ is a $\mathfrak R$-regular degeneracy determinant on $\mathcal M$.
\end{cor}
This is a higher dimensional analogue of the small Hessian theorem, that is, \cite[Theorem 3.2]{LiYang2023}, which works for polynomial hypersurfaces in $\R^n$. The next corollary works for polynomial curves in $\R^n$.

\begin{cor}\label{cor:analytic_curve}
    Let $\mathcal M$ consist of graphs $\phi=(\phi_i)_{i=1}^{n-1}$ of polynomials $\phi_i\in \mathcal P_{1,d}$, and let $\mathfrak R=\{\mathcal R_0\}$. Then $HM_\phi:=\det (W_m\phi (W_m\phi)^T)$, where $W_m$ is as defined in \eqref{eq:Wons}, is a $\mathfrak R$-regular degeneracy determinant on $\mathcal M$.
\end{cor}


\subsubsection{Trivial covering property}
Before we state the main theorem, we still need a mild technical assumption.

\begin{defn}[Trivial covering property]\label{defn:trivial_covering}
    Let $(\mathcal M,\mathfrak R)$ be $\mathcal A$-rescaling invariant and $\mathcal A$-compatible. 
    We say that $(\mathcal M,\mathfrak R)$ satisfies the trivial covering property if there exist constants $C_1,C_2,C\ge 1$, depending only on $\mathcal M,\mathfrak R,\mathcal A$, such that the following holds for every $M_\phi\in \mathcal M$: there exists a family $\mathcal R_{\mathrm{tri},\phi}\sub \mathcal R_\phi$ such that for every $K\ge 1$, we can cover $[-1,1]^k$ by $\le CK^{C_1}$ parallelograms $R_0\in \mathcal R_{\mathrm{tri},\phi}$, such that each $R_0$ has diameter $\le K^{-1}$, and has its width $\ge C^{-1}K^{-C_2}$. Moreover, the affine bijection $\lambda_{R_0}\in \mathcal A$.
        
\end{defn}
For example, in the case where $\mathcal R_\phi$ contains all the (axis-parallel) squares contained in $[-1,1]^k$, and $\mathcal A=\mathcal A_0$, we may simply cover $[-1,1]^k$ by squares of side length $\sim K^{-1}$. In this case, we may take $C_1=k$, $C_2=1$.

\subsection{Statement of theorem}

We are now ready to state the main theorem of this section, the degeneracy locating principle. 

Let $\mathcal A$ be a family of affine bijections on $\R^k$ that is closed under composition (see Definition \ref{defn:compatible}). Let $(\mathcal M,\mathfrak R)$ be a pair that is $\mathcal A$-rescaling invariant as in Definition \ref{defn:rescaling_invariant_pair} and $\mathcal{A}$-compatible as in Definition \ref{defn:compatible}, and assume it satisfies the trivial covering property as in Definition \ref{defn:trivial_covering}. Let $H$ be a $\mathfrak R$-regular degeneracy determinant defined as in Definition \ref{defn:degeneracy_determinant_reg}. (Also recall that we have fixed dimensions $1\le k\le m\le n-1$, and exponents $p,q,\alpha$ for decoupling, right before Lemma \ref{lem:equivalent_decoupling}.)

\begin{thm}[Degeneracy locating principle]\label{thm:degeneracy_locating_principle}
For $M_\phi\in \mathcal M$, there exist subfamilies $R_{\mathrm{sub,\phi}}$, $R_{\mathrm{deg,\phi}}$, $R_{\mathrm{non,\phi}}$ of $\mathcal R_\phi$, such that the following holds for every $\sigma>0$ and every $\eps>0$, where all constants $C$ (resp. $C_\eps$, overlap functions $B$)  below depend only on $\mathcal M,\mathfrak R,\mathcal A,H$ (resp. $\mathcal M,\mathfrak R,\mathcal A,H,\eps$):
    \begin{enumerate}[label=(\alph*)]
        \item \label{item:sublevel_set_decoupling} (Sublevel set decoupling) The subset
        \begin{equation*}
            \{x\in [-1,1]^k:|HM_\phi(x)|\le \sigma\}
        \end{equation*}
        can be decoupled (as in Definition \ref{defn:decoupling}) into $B$-overlapping parallelograms $R_{\mathrm{sub}}\in \mathcal R_{\mathrm{sub},\phi}$ at a cost of $C_\varepsilon\sigma^{-\varepsilon}$, on each of which $|HM_\phi|\le C\sigma$. Moreover, we require that $\lambda_{R_{\mathrm{sub}}}\in \mathcal A$.
    
    \item \label{item:degnerate_decoupling} (Degenerate decoupling) If $HM_\phi\equiv 0$ on $[-1,1]^k$, then $[-1,1]^k$ can be $\phi$-decoupled (as in Definition \ref{defn:graphical_decoupling_convex_hull}) at scale $\sigma$ into $(\phi,\sigma)$-flat $B$-overlapping parallelograms $R_{\mathrm{deg}}\in \mathcal R_{\mathrm{deg},\phi}$ at a cost of $C_\eps\sigma^{-\varepsilon}$. Moreover, we require that $\lambda_{R_{\mathrm{deg}}}\in \mathcal A$.

    \item \label{item:nondegnerate_decoupling} (Nondegnerate decoupling) If $|HM_\phi| \geq K^{-1}$ on $[-1,1]^k$ for some $K\ge 1$, then $[-1,1]^k$ can be $\phi$-decoupled (as in Definition \ref{defn:graphical_decoupling_convex_hull}) at scale $\sigma$ into $(\phi,\sigma)$-flat $B$-overlapping parallelograms $R_{\mathrm{non}}\in \mathcal R_{\mathrm{non},\phi}$ at a cost of $K^{C}\sigma^{-\varepsilon}$. (Importantly, the power $C$ on $K$ cannot depend on $\eps$.)

    \end{enumerate}
Moreover, we require that the collection of all such parallelograms (at scale $\sigma$) $R_{\mathrm{sub}}$, $R_{\mathrm{deg}}$, $R_{\mathrm{non}}$ has cardinality at most $K^{C_\eps}\sigma^{-C_\eps}$.
    
    Then we have the following conclusions:
    \begin{enumerate}[label=(\Roman*)]
        \item \label{item:Feb_22_01} For each $M_\phi\in \mathcal M$ and each $\delta>0$, $[-1,1]^k$ can be $\phi$-decoupled into $(\phi,\delta)$-flat $B'$-overlapping parallelograms $R_{\mathrm{final}}\in \mathcal R_\phi$ at the cost of $C'_\eps \delta^{-\eps}$, where $C'_\eps$ and $B'$ depend only on $\mathcal M,\mathfrak R,\mathcal A,H,\eps,p,q,\alpha$.

        \item \label{item:Feb_7_03} In addition to Part \ref{item:Feb_22_01}, assume the following is correct:
        \begin{itemize}
            \item each element of $\mathcal A$  is an affine bijection of the form $x\mapsto Dx+b$ where $D$ is diagonal,
            \item the parallelograms $R_0$ in Definition \ref{defn:trivial_covering} are disjoint axis-parallel rectangles,
            \item each decoupling in \ref{item:sublevel_set_decoupling}, \ref{item:degnerate_decoupling} and \ref{item:nondegnerate_decoupling} results in axis-parallel rectangles with disjoint interiors,
        \end{itemize}
        then the final parallelograms $R_{\mathrm{final}}$ are disjoint axis-parallel rectangles.
        \item \label{item:width_lower_bound} In addition to Part \ref{item:Feb_22_01}, suppose $k\ge 2$ and that each decoupling in \ref{item:sublevel_set_decoupling}, \ref{item:degnerate_decoupling} and \ref{item:nondegnerate_decoupling} results in parallelograms with width at least $\sigma$ when $\sigma\ll 1$. Suppose, in addition, that the following lower dimensional decoupling (*) holds:
        \begin{itemize}
            \item [(*)] Let $M_\phi\in \mathcal M$. Given a parallelogram $R$ that is either an element of $\mathcal R_{\mathrm{tri},\phi}$ from the trivial decoupling in Definition \ref{defn:trivial_covering}, or an element of $\mathcal R_{\mathrm{sub},\phi}$ in Assumption \ref{item:sublevel_set_decoupling}. Denote by $ w$ the width of $R$, and let $L\sub \R^{k-1}$ be the lower dimensional parallelogram obtained from $R$ by removing the shortest dimension of $R$. Then the following conditions hold:
            \begin{enumerate}[label=(\roman*)]
                \item $L$ can be $\phi|_L$-decoupled into $(\phi|_L, w)$-flat parallelograms $R'_{\mathrm{low}}$ at scale $ w$ at cost $C_\eps  w^{-\eps}$, with $R'_{\mathrm{low}}$ having width at least $ w$.
                \item Denote by $R_{\mathrm{low}}$ the parallelogram obtained from $R'_{\mathrm{low}}$ by adding back the dimension of length $  w$. Then we require that $\lambda_{R_{\mathrm{low}}}\in \mathcal A$, and that $R_{\mathrm{low}}$ is contained in $2R$. 
           
            \end{enumerate}  
        \end{itemize}  
        \noindent Then the final parallelograms $R_{\mathrm{final}}$ in Part \ref{item:Feb_22_01} will have width bounded below by $\delta$ (when $\delta\ll 1$). In particular, by the bounded overlap, the final collection consisting of $R_{\mathrm{final}}$ will have cardinality $\le C_\eps\delta^{-k}$.
    \end{enumerate}

\end{thm}

For example, in \cite{LiYang2023}, \ref{item:sublevel_set_decoupling} follows from Theorem \cite[Theorem 3.1]{LiYang2023}, \ref{item:degnerate_decoupling} follows from Theorem \cite[Theorem 1.4]{Yang2}, and \ref{item:nondegnerate_decoupling} follows from Proposition \ref{prop:Bourgain_Demeter_K-1_lowerbound}. 

For more examples with $\mathfrak R$ taken to be a singleton $\{\mathcal R\}$, it is straight forward to check that the following applications of the degeneracy locating principle in Corollaries \ref{cor:radial} to \ref{cor:smooth_curve} satisfy the full assumptions in Theorem \ref{thm:degeneracy_locating_principle}. 

\begin{table}[h!]
    \centering
    \begin{tabular}{|c|c|c|c|}
    \hline   Corollary  & Manifolds $M_\phi \in \mathcal{M}$ & Elements in $\mathcal{R}$ & $HM_\phi$ \\
    \hline   \ref{cor:radial}   & $\phi(s,t') = P(s) + c\psi(t')$\tablefootnote{In the introduction, $\psi$ is taken to be a single function $\sqrt{1-|t'|^2}$ for simplicity. In reality, we should consider a larger family for rescaling invariance.}, & $(s,t') \in I \times Q$, & $P''(s)$ \\
    & $P\in \mathcal P_{1,d}$, $ D^2\psi\succ 0$, $|c|\le 1$ & $Q$ square &\\
    \hline   \ref{cor:dec_add}   & $\phi (x_1,x_2) = \phi_1(x_1)+c\phi_2(x_2) $\tablefootnote{For simplicity, we only write down the case for the family $\mathcal{M}_1^n$, $J=2,n=3,4,5$. The other families are straightforward generalisations of this case.} & $(x_1,x_2) \in R \times Q$, & $\det D^2 \phi_1$ \\
    & $\in \mathcal P_{2,d}$, $|\det D^2 \phi_2| \sim 1$, $|c|\le 1$ & $Q$ square &\\
    \hline   \ref{cor:smooth_curve}   & $\phi=(\phi_i)_{1\le i\le n-1}$ ,  & all intervals & $\det ((W_m\phi)$ \\
    & $\phi_i\in \mathcal P_{1,d}$ & i.e. $\mathcal{R}_0$ & $(W_m \phi)^{T}) $ \\
    \hline  
    \end{tabular}
    \caption{Applications of the degeneracy locating principles in the corollaries.}
    \label{tab:my_label}
\end{table}

It is clear that Corollaries \ref{cor:radial} and \ref{cor:smooth_curve}, and the special case of Corollary \ref{cor:dec_add} in the table above fall into the scenario of Part \ref{item:Feb_7_03}. Therefore, these manifolds can be decoupled into a collection of rectangles having non-overlapping interiors. By a similar construction to \cite[Proposition 3.1]{Yang2}, the decoupled $(\phi,\delta)$-flat rectangles are essentially those in the statements of the corollaries.

\subsection{Proof of Theorem \ref{thm:degeneracy_locating_principle}}
The rest of this section is devoted to the proof of Theorem \ref{thm:degeneracy_locating_principle}. It features an induction on scales argument similar to that of \cite[Section 5]{LiYang2023}. The only essential difference here is that we choose a different scale $K$ (the scale $M$ in \cite{LiYang2023}), which allows us to improve the overlapping bound of the parallelograms from $C_\eps \delta^{-\eps}$ to $C_\eps$ (we summarise this in Section \ref{sec:refined_IJ} in the appendix). This improvement is nontrivial (c.f. \cite{GMO24}).

All implicit constants below in the proof are allowed to depend on the family $\mathcal M,\mathfrak R,\mathcal A,H$ and parameters $\eps,p,q,\alpha$.

\subsubsection{The base case}
We may assume $\delta\ll_{\eps} 1$, otherwise the decoupling inequality is trivial.
 
Let $K=\delta^{-\eps}$. Our induction hypothesis is that the conclusion holds at all coarser scales $\delta'\gtrsim K^\beta\delta$, where $\beta$ is as in Part \eqref{item:totally_degenerate_approximation} of Definition \ref{defn:degeneracy_determinant}.

\subsubsection{Degenerate and nondegenerate parts}\label{subsub_curved_flat}
Let $M_\phi\in \mathcal M$. Partition $[-1,1]^k$ into two parts $S_{\text{non}}$ and $S_{\text{degen}}$ as follows:
\begin{align*}
    &S_{\text{non}}:=\{x\in [-1,1]^k: |HM_\phi(x)|>K^{-1}\},\\
    &S_{\text{degen}}:=\{x\in [-1,1]^k: |HM_\phi(x)|\leq K^{-1}\}.
\end{align*}

By the trivial covering property, we cover the whole $[-1,1]^k$ by $K^{O(1)}$ parallelograms $R_0\in \mathcal R_\phi$ of diameter at most $cK^{-1}$ for some $c\ll 1$, but each dimension of $R_0$ is $\gtrsim K^{-O(1)}$. We then consider:
\begin{align*}
    & S'_{\text{non}}:=\cup\{R_0:R_0\cap S_{\text{non}}\ne \varnothing\}.
\end{align*}
It suffices to decouple $S'_{\text{non}}$ and $S_{\text{degen}}$ separately.

\subsubsection{Nondegenerate subset}\label{subsub_nondeg}
We first deal with $S'_{\text{non}}$. First, by a trivial decoupling at cost $K^{O(1)}=\delta^{-O(\eps)}$, we may locate to one $R_0$. By the Lipschitz Condition \eqref{eqn:Lipschitz}, if $c$ is small enough, then we have $|HM_\phi|>K^{-1}/2$ on each $R_0$. 

Now we rescale $R_0$ to $[-1,1]^k$. Using the assumption that each dimension of $R_0$ is $\gtrsim K^{-O(1)}$ and using Part \eqref{item:regularity_small_Hessian_rescaling} of Definition \ref{defn:degeneracy_determinant_reg}, we see that $|HM_{\phi\circ \lambda_{R_0}}|\gtrsim K^{-O(1)}$ on $[-1,1]^k$. Thus, we may apply the nondegenerate decoupling Assumption \ref{item:nondegnerate_decoupling} (with $K^{O(1)}$ in place of $K$) to obtain parallelograms $R_{\mathrm{non}}\in \mathcal R_{\phi\circ\lambda_{R_0}}$. It remains to rescale the decoupled parallelograms $R_{\mathrm{non}}$ back to $R_0$. But $\lambda_{R_0}\in \mathcal A$ by Definition \ref{defn:trivial_covering}, and by the remark after Definition \ref{defn:compatible}, we have $\lambda_{R_0}(R_{\mathrm{non}})\in \mathcal R_\phi$.

\subsubsection{Degenerate subset}\label{sec:degenerate_subset}
Now we deal with $S_{\text{degen}}$, with the hope of reducing to the induction hypothesis.

\fbox{Sublevel set decoupling}

We apply the sublevel set decoupling Assumption \ref{item:sublevel_set_decoupling} at scale $\sigma:=K^{-1}$ to decouple $\{x\in [-1,1]^k:|HM_\phi(x)|<K^{-1}\}$ into parallelograms $R_{\mathrm{sub}}\in \mathcal R_\phi$ on each of which $|HM_\phi|\lesssim K^{-1}$. It remains to further $\phi$-decouple each $R_{\mathrm{sub}}$ into smaller $(\phi,\delta)$-flat parallelograms.

\fbox{Rescaling and approximation}

We rescale $R_{\mathrm{sub}}$ to $[-1,1]^k$. For simplicity of notation we denote $\tilde \phi:=\phi\circ \lambda_{R_{\mathrm{sub}}}$.

By Part \eqref{item:regularity_small_Hessian_rescaling} of Definition \ref{defn:degeneracy_determinant_reg}, we have $|HM_{\tilde \phi}|\lesssim K^{-1}$. Thus, by Part \eqref{item:totally_degenerate_approximation} of Definition \ref{defn:degeneracy_determinant}, we may approximate $\tilde \phi$ by $(C')^{-1}\psi$ with error $O(K^{-\beta})$, where $HM_\psi\equiv 0$ on $[-1,1]^k$.

\fbox{Degenerate decoupling}

By the degenerate decoupling Assumption \ref{item:degnerate_decoupling} with $\sigma=K^{-\beta}$, we may $\psi$-decouple $[-1,1]^k$ into $(\psi,K^{-\beta})$-flat parallelograms $R_{\mathrm{deg}}\in \mathcal R_\psi$, which is also $(\tilde \phi,O(K^{-\beta}))$-flat by the approximation. By Part \eqref{item:totally_degenerate_approximation} of Definition \ref{defn:degeneracy_determinant_reg} with $\phi$ taken to be $\tilde \phi$, we also have $R_{\mathrm{deg}}\in \mathcal R_{\tilde\phi}$.

\fbox{Rescaling and applying induction hypothesis}

By the rescaling invariance assumption in Definition \ref{defn:rescaling_invariant_pair}, we have
\begin{equation*}
    (M_{\vec \sigma\varphi},[-1,1]^k)\overset{\Xi}{\longmapsto} (M_{\tilde\phi}, R_{\mathrm{deg}}),
\end{equation*}
for some $ M_\varphi \in \M$ and some $\vec{\sigma}\in \R^{n-k}$ with $\min_{i} \sigma_i \sim K^{-\beta}$. 

Hence, by Lemma \ref{lem:equivalent_decoupling_A}, the $\tilde \phi$-decoupling of $R_{\mathrm{deg}}$ into $(\tilde\phi,\delta)$-flat parallelograms belonging to $\mathcal R_{\tilde \phi}$ is implied by the $\varphi$-decoupling of $[-1,1]^k$ into $(\varphi,O( K^\beta \delta))$-flat parallelograms belonging to $\mathcal R_\varphi$. The latter is a decoupling at a larger scale, whence the induction hypothesis can be applied. We have thus obtained rectangles $ R_{\mathrm{ind}}\in \mathcal R_{\tilde \phi}$ which $\tilde \phi$-decouple $R_{\mathrm{deg}}$. 

We then apply the affine bijection $\Xi:=\lambda_{R_{\mathrm{sub}}}\circ \lambda_{R_{\mathrm{deg}}}$ to these $ R_{\mathrm{ind}}$, and so it remains to show that $\Xi( R_{\mathrm{ind}})\in \mathcal R_\phi$. But $\Xi\in \mathcal A$, by the assumption that $\mathcal A$ is closed under composition, and Parts \ref{item:sublevel_set_decoupling} and \ref{item:degnerate_decoupling} of the decoupling assumptions. Then this follows from the remark after Definition \ref{defn:compatible}, with $\psi=\tilde \phi=\phi\circ \Xi$.

\subsubsection{Decoupling inequality}\label{sec_decoupling_Sec6}

We briefly describe the proof of the required decoupling inequality, and point out the differences from \cite[Section 5.3]{LiYang2023}.

Denote by $D(\delta)$ the cost of $\phi$-decoupling of $[-1,1]^k$ into $(\phi,\delta)$-flat parallelograms. Our goal is to show that $D(\delta)\leq C'_{\eps}\delta^{-\eps}$.

The first difference is the choice of $K=\delta^{-\eps}$ instead of $K\sim_\eps 1$ in \cite{LiYang2023} (where we used $M$ there in place of $K$). The second difference is that we now need to combine $\ell^q(L^p)$ decoupling inequalities obtained at each step, which is done by a simple dyadic decomposition argument with a logarithmic loss $O(|\log \delta|)$ at each step (see Proposition \ref{prop:combine_decoupling}, and recall the cardinality assumption in Theorem \ref{thm:degeneracy_locating_principle}). This is acceptable since we already lose $K^\eps=\delta^{-\eps^2}\gg |\log \delta|$ at each step of the induction.  
 
Carrying out the same proof and combining the decoupling inequalities obtained from the nondegenerate and degenerate parts in Sections \ref{subsub_curved_flat} and \ref{sec:degenerate_subset}, respectively, we essentially obtain the following bootstrap inequality:
\begin{equation}\label{eqn:bootstrap}
    D(\delta)\leq C_\eps\delta^{-\eps}+C_\eps D(K^\beta \delta)K^{\eps}\le C_\eps\delta^{-\eps}+C_\eps D(K^\beta \delta)\delta^{-\eps^2}.
\end{equation}
(Here we have omitted the unimportant implicit constants.) Iterating \eqref{eqn:bootstrap} for $N$ times, we obtain
\begin{align*}
     D(\delta)&\leq C_\eps\delta^{-\eps}(1+\delta^{-\eps^2}+\delta^{-2\eps^2}+\cdots+\delta^{-N\eps^2})+\delta^{-N\eps^2}D(K^{N\beta}\delta).
\end{align*}
We will stop this iteration once we reach $K^{N\beta}\delta\sim 1$, that is, when 
\begin{equation}
    N\sim (\beta \eps)^{-1}.
\end{equation}
In this case we have $D(K^{\beta N} \delta)\sim 1$, and 
\begin{equation*}
    1+\delta^{-\eps^2}+\delta^{-2\eps^2}+\cdots+\delta^{-N\eps^2}+\delta^{-N\eps^2}D(K^{N\beta}\delta)\lesssim \delta^{-O(\eps)},
\end{equation*}
which finishes the proof of the decoupling inequality. This establishes Part \ref{item:Feb_22_01}.

\subsubsection{Analysis of overlap}\label{sec:overlap}

There are two main sources of overlaps between parallelograms, namely, overlaps between parallelograms created within one iteration, and overlaps between different iterations.

The former kind of overlap is bounded by the overlap condition in the decoupling assumptions of Theorem \ref{thm:degeneracy_locating_principle}. For the second kind of overlap, we can now see that the choice of $K=\delta^{-\eps}$ leads to $N=O(1/\eps)$, so we may just apply Proposition \ref{prop:iterative_overlap} to roughly bound the overlapping number by $O_\eps(1)$. 
In the special case of Part \ref{item:Feb_7_03}, we see that there should be no overlap.

\subsubsection{Lower bound on dimensions}
We still need to prove Part \ref{item:width_lower_bound}. 

\fbox{Nondegenerate part}

For the nondegenerate part, we go back to the argument in Subsection \ref{subsub_nondeg} after arriving at one $R_0$. Denote by $w$ the width of $R_0$. By Definition \ref{defn:trivial_covering}, we have $K^{-C_2}\lesssim w\lesssim K^{-1}$. Using the extra decoupling Assumption (*) at scale $w$, we may further $\phi$-decouple $R_0$ into some $(\phi,w)$-flat parallelogram $R_{\mathrm{low}}\sub 2R_0$ at cost $O_\eps(w^{-\eps})=O_\eps(K^{O(\eps)})$, such that $\lambda_{R_{\mathrm{low}}}\in \mathcal A$. Moreover, the dimensions of $R_{\mathrm{low}}$ are $\lesssim K^{-1}$ and $\gtrsim w$.

Since $R_{\mathrm{low}}$ is $(\phi,w)$-flat, by Definition \ref{defn:rescaling_invariant_pair}, we have
\begin{equation*}
    (M_{\vec \sigma\phi_1},[-1,1]^k)\overset{\lambda_{R_{\mathrm{low}}}}{\longmapsto} (M_{\phi}, R_{\mathrm{low}}),
\end{equation*}
for some $M_{\phi_1}\in \mathcal M$ and some $\vec \sigma$ with $\min \sigma_i\sim w$. We then apply the nondegenerate decoupling Assumption \ref{item:nondegnerate_decoupling} at scale $w^{-1}\delta$ to $\phi_1$-decouple $[-1,1]^k$ into parallelograms $R_{\mathrm{non}}$, which have dimensions at least $w^{-1}\delta$, by the first sentence of \ref{item:width_lower_bound} (Note that no sublevel set decoupling or degenerate decoupling is used at this step.) Rescaling back, the final parallelograms are given by $\lambda_{R_\mathrm{low}}(R_{\mathrm{non}})$, whose dimensions are bounded below by $ww^{-1}\delta=\delta$.

\fbox{Degenerate part}

The proof will be similar to that of Subsection \ref{sec:degenerate_subset}, but we will outline the main differences. The part about sublevel set decoupling still goes through, with the extra property that each $R_{\mathrm{sub}}$ has width at least $K^{-1}$.

However, before the ``rescaling and approximation" and ``degenerate decoupling" steps, we will add an extra step of a lower dimensional decoupling to the nondegenerate part right above. More precisely, let $w$ be the width of $R_{\mathrm{sub}}$, and by the first sentence of \ref{item:width_lower_bound}, we have $w\ge K^{-1}$. Apply the extra decoupling Assumption (*) at scale $w$ to $\phi$-decouple $R_{\mathrm{sub}}$ into $(\phi,w)$-flat parallelograms $R_{\mathrm{low}}$, each with dimensions at least $w$ at a cost of $O_\varepsilon(w^{-\varepsilon}) =O_\varepsilon(K^\varepsilon)$.

Since $R_{\mathrm{low}}$ is $(\phi,w)$-flat, by Definition \ref{defn:rescaling_invariant_pair}, we have
\begin{equation*}
    (M_{\vec \sigma\phi_1},[-1,1]^k)\overset{\lambda_{R_{\mathrm{low}}}}{\longmapsto} (M_{\phi}, R_{\mathrm{low}}),
\end{equation*}
for some $M_{\phi_1}\in \mathcal M$ and some $\vec \sigma$ with $\min \sigma_i\sim w$. By the compatibility assumption in Definition \ref{defn:compatible}, it then suffices to study decoupling for $\phi_1$ on $[-1,1]^k$ at scale $w^{-1}\delta$, and prove that it produces parallelograms of dimensions at least $w^{-1}\delta$.

To this end, we consider two cases. Denote 
\begin{equation*}
    \gamma=(2C'')^{-1},
\end{equation*}
where $C''$ is as in \eqref{eqn:regularity_rescaling_determinant}.

If $w\le K^{-\gamma}$, then we can use the induction hypothesis of decoupling at scale $w^{-1}\delta\ge K^\gamma \delta$.

If $w>K^{-\gamma}$, then on $[-1,1]^k$, using \eqref{eqn:regularity_rescaling_determinant}, we have 
\begin{equation*}
    |HM_{\phi_1}|\lesssim w^{-C''}K^{-1}\le K^{C''\gamma-1}= K^{-1/2}.
\end{equation*}
Then we can continue with the same argument as ``rescaling and approximation", ``degenerate decoupling" (at scale $K^{-\beta/2}$) and ``rescaling and applying induction hypothesis" in Subsection \ref{sec:degenerate_subset} (with $K^\beta \delta$ replaced by $K^{\beta/2} w^{-1}\delta$). By the first sentence of \ref{item:width_lower_bound}, the parallelograms of the form $\lambda_{R_{\mathrm{deg}}}(R_\mathrm{ind})$ will have dimensions at least $K^{\beta/2} w^{-1}\delta K^{-\beta/2}=w^{-1}\delta$, as required.

\fbox{Decoupling inequality}

The bootstrap inequality is slightly different from \eqref{eqn:bootstrap} in this case, due to the extra step of decoupling at scale $w$. Essentially, \eqref{eqn:bootstrap} becomes
\begin{equation*}
    D(\delta)\leq C_\eps\delta^{-\eps}+C_\eps K^\eps\left(\sup_{K^{-\gamma}< w\le 1}D(w^{-1}K^{\beta/2} \delta)+\sup_{K^{-1}\le w\le K^{-\gamma}}D(w^{-1}\delta)\right).
\end{equation*}
A similar iteration should lead to $D(\delta)\lesssim_\eps \delta^{-\eps}$, and we omit the details.

\pushQED{\qed} 
\qedhere
\popQED

\section{Degenerate approximation theorem}\label{sec:degen_approximation_theorem}

In this section, our main goal is to prove the degenerate approximation theorem. As a corollary, we are able to prove Proposition \ref{prop:degeneracy_determinant_Lojasiewicz}.

In this context, we may denote $l:=n-k$ and slightly abuse notation by letting $H$ act on $\phi=(\phi_i)_{i=1}^{l}$ directly. 


We need a little more notation. For convenience, we introduce the natural linear isomorphism $\Lambda_{k,d}: \mathcal{P}^0_{k,d}\to \R^{\binom{k+d}{k}}$ defined by
$\Lambda_{k,d}(\sum_{|\alpha|\leq d} c_\alpha x^\alpha) = (c_\alpha)_{|\alpha|\leq d}$. Define a norm on $\mathcal P^0_{k,d}$ by $\norm {\phi_i}=\sup_{x\in [-1,1]^k}|\phi_i(x)|$. By Corollary \ref{cor:poly_norm}, $\Lambda_{k,d}$ is also essentially a normed space isomorphism that maps $\mathcal P_{k,d}$ to the unit ball of $\R^{\binom{k+d}{k}}$, with its operator norm $\sim_{k,d} 1$. (For this reason, we encourage the reader to regard $\phi=(\phi_i)$ as vectors, not as polynomials.) We also define $\Lambda_{k,D}$ similarly. 

In addition, denote by $(\mathcal P^0_{k,d})^{l}$ the collection of polynomial mappings $\phi=(\phi_i)_{1\le i\le l}$, where each $\phi_i\in \mathcal P^0_{k,d}$. Then we may naturally define the ``tensor product" $\otimes\Lambda_{k,d}$ that maps from $(\mathcal P^0_{k,d})^l$ to $(\R^{\binom{k+d}{k}})^{l}\cong \R^{l\binom{k+d}{k}}$, which is still essentially a normed space isomorphism with the natural norm $\norm {\phi}:=\sup_{1\le i\le l}\sup_{x\in [-1,1]^k}|\phi_i(x)|$.

In this way, for $l\ge 1$, $H$ induces a polynomial map $Q=(Q_j)$ from $\R^{l\binom{k+d}{k}}$ to $\R^{\binom{k+D}{k}}$, where each $Q_j$ is a polynomial of degree bounded by some $D'=D'(H)$. (For example, when $l=1$, $H\phi:=\det D^2 \phi$, we have $D'=k$.) In this way, $H$ has a natural extension to all of $(\mathcal{P}_{k,d}^0)^{l} $. See Figure \ref{fig:cd} below.

\begin{figure}[h]
    \centering
        \[
        \begin{tikzcd}[column sep = large, row sep = large]
        (\mathcal{P}_{k,d}^0)^{l} \arrow{r}{H} \arrow[leftrightarrow]{d}{\bigotimes\Lambda_{k,d}} & \mathcal{P}_{k,D}^0 \arrow[leftrightarrow]{d}{\Lambda_{k,D}} \\
        \left (\R^{\binom{k+d}{k}}\right)^{l}\cong \R^{l\binom{k+d}{k}} \arrow{r}{Q} & \R^{\binom{k+D}{k}} 
        \end{tikzcd}
        \]
        \caption{Map $Q$ induced by $H$.}
        \label{fig:cd}
\end{figure}
We now state and prove the following main approximation theorem in this section.
\begin{thm}[Degenerate approximation theorem]\label{thm:lojaciewicz}
    Let $H: (\mathcal P_{k,d}^0)^{l} \mapsto \mathcal P_{k,D}^0$ be such that the map $Q$ as in Figure \ref{fig:cd} has all its components being polynomials. Then there exist constants $C=C(H)>0$ and $\beta=\beta(H)>0$ such that the following holds:

    For every $\sigma\in (0,1]$ and every $\phi\in (\mathcal P^0_{k,d})^{l}$ such that $\sup_{x\in [-1,1]^k} |H\phi(x)|\le \sigma$, there exists some $\bar \phi\in (\mathcal P_{k,d}^0)^{l}$ such that $H\bar \phi\equiv 0$, and that
    \begin{equation*}
        \sup_{x\in [-1,1]^k}|\phi(x)-\bar \phi(x)|\leq C \sigma^\beta.
    \end{equation*}
\end{thm}

The proof of Theorem \ref{thm:lojaciewicz} uses the following version of the \L ojasiewicz inequality. 
    \begin{prop}[\cite{Lojasiewicz}]\label{prop:Lojasiewicz}
        Let $S$ be an $N$-variate real polynomial of degree at most $D'\ge 2$, and denote by $Z(S)\sub \R^N$ the zero set of $S$, which we assume to be nonempty. Then for every $u\in \R^N$,
        \begin{equation*}
            \mathrm{dist}(u,Z(S))\le C|S(u)|^\gamma,
        \end{equation*}
        where $\gamma^{-1}=\max\{D'(3D'-4)^{N-1},2D'(3D'-3)^{N-2}\}$, and $C$ depends only on $N,D'$.
    \end{prop}
(The precise value is not important in this paper.) All implicit constants in the proof below depend only on $H$.
\begin{proof}[Proof of Theorem \ref{thm:lojaciewicz}]
    By the assumption of $H$, $Q:= \Lambda_{k,D} \circ H \circ (\otimes\Lambda^{-1}_{k,d}): \R^{l\binom{k+d}{k}} \mapsto \R^{\binom{k+D}{k}}$ is a vector of polynomials with their degrees bounded by some $D'=D'(H)$. Moreover, $S(u) := \|Q(u)\|_2^2$ is a $\binom{k+d}{k}$-variate real polynomial of degree at most $2D'$.

    Let $\phi \in (\mathcal{P}^0_{k,d})^l$ be such that $\sup_{x\in [-1,1]^k} |H\phi(x)|\le \sigma$, i.e. $\|Q(\Lambda_{k,d}\phi)\|_{\infty} \lesssim \sigma$, whence $\|S(\Lambda_{k,d}\phi)\|_{\infty} \lesssim \sigma^2$. Using Proposition \ref{prop:Lojasiewicz} with $N=\binom {k+d}k$ and $2D'$ in place of $D'$, there exists some $\gamma'>0$, depending in turn only on $H$, such that 
    $$
    \dist(\Lambda_{k,d}\phi , Z(S)) \lesssim |S(\Lambda_{k,d}\phi)|^{\gamma'} \lesssim \sigma^{2 \gamma'}.
    $$
    Thus, there exists $u_0 \in Z(S)$ such that $|\Lambda_{k,d}\phi - u_0|\lesssim \sigma^{2\gamma'}$. Let $\beta=2\gamma'$, and let $\bar \phi:=\Lambda_{k,d}^{-1}u_0$. Thus, $H\overline \phi\equiv 0$ since $S(u_0)=0$, and so $\bar \phi$ is as required.

\end{proof}

Proposition \ref{prop:degeneracy_determinant_Lojasiewicz}, which is rephrased as the following corollary using Figure \ref{fig:cd}, can then be deduced from Theorem \ref{thm:lojaciewicz}. (Recall $l=n-k$.)

\begin{cor}\label{cor:degeneracy_determinant_Lojasiewicz}
    Let $H: (\mathcal P^0_{k,d})^{l} \mapsto \mathcal P_{n,D}^0$ be a map such that the map $Q$ as in Figure \ref{fig:cd} obeys the following: if $u \in Z(Q)$ and $c\in \R$, then $cu \in Z(Q)$.   
    
    In the context of Definition \ref{defn:degeneracy_determinant_reg}, let $\mathcal{M}\subseteq \mathcal{M}_0$ be such that $\mathcal{M}$ consists of graphs of $\phi=(\phi_i)_{i=1}^{l}$ where each $\phi_i\in \mathcal{P}_{k,d}$, and $\mathfrak{R}=\{\mathcal R\}$ for some $\mathcal R\sub \mathcal R_0$. Then $H$ is a degeneracy determinant on $\mathcal M$, with constants $C,C',\beta$ depending only on $H$.
\end{cor}
\begin{proof}
By the definition of $H$, there exist constants $D'=D'(H)$ and $K=K(H)$ such that the map $Q$ as in Figure \ref{fig:cd} has all its components being polynomials of degree bounded by $D'$ and with coefficients bounded by $K$.


To prove Part \eqref{item:totally_degenerate_approximation}, given $\phi\in (\mathcal P_{k,d})^l$ with $\sup_{x\in [-1,1]^k}|H\phi(x)|\le \sigma$, by Theorem \ref{thm:lojaciewicz}, there exists some $\bar \phi\in \mathcal P_{k,d}^0$ such that $H\bar \phi\equiv 0$, and that $\sup_{x\in [-1,1]^k}|\phi(x)-\bar \phi(x)|\leq C \sigma^\beta$. Note that $\|\bar \phi\| \leq \|\phi\|+C\sigma^{\beta}\le C'$ for some $C'\ge 1$. Defining $\psi:=(C')^{-1}\bar \phi$, we have $\psi \in \mathcal P_{k,d}$. Also, by the assumption that $cu\in Z(Q)$ whenever $u\in Z(Q)$, we have $H\psi\equiv 0$. Thus, Part \eqref{item:totally_degenerate_approximation} holds.
\end{proof}

\section{Appendix}\label{sec_appendix}

\subsection{Facts about parallelograms}
Recall from Section \ref{sec:equivalence_objects} that $CP$ denotes the concentric dilation of a parallelogram $P\sub \R^n$ by a factor of $C$.
\begin{prop}\label{prop:dilation_interchangable_with_affine_for_parallelogram}
    Let $P\sub \R^n$ be a parallelogram, and $\Lambda:\R^n\to \R^m$ be an affine transformation. For any $C>0$, we have $\Lambda(CP)=C\Lambda P$.
\end{prop}
\begin{proof}
Denote by $\lambda x:=Ax+b$ an affine bijection from $Q_0:=[-1,1]^n$ to $P$. Then $P=\lambda Q_0=AQ_0+b$, and $CP=(CA)Q_0+b$.

Write $\Lambda x:=Ux+v$, so $\Lambda(cP)=U(CAQ_0)+Ub+v=CUAQ_0+Ub+v$. On the other hand, $\Lambda P=UAQ_0+Ub+v$, so $C\Lambda P=CUAQ_0+Ub+v=\Lambda (CP)$.
\end{proof}

\begin{prop}
    If $P_1\sub P_2\sub \R^n$ are parallelograms and $C\ge 1$, the $CP_1\sub CP_2$. (Note the dilations are with respect to possibly different centres.)
\end{prop}
\begin{proof}
By Proposition \ref{prop:dilation_interchangable_with_affine_for_parallelogram}, we may assume $P_2=[-1,1]^n$, so $CP_2=[-C,C]^n$. The case $n=1$ is easy. For a general $n$, it suffices to prove that $\pi_i(CP_1)\sub [-C,C]$ for each $1\le i\le n$, where $\pi_i$ denotes the projection onto the $i$-th coordinate. But Since $P_1\sub P_2$, we have $\pi_i(P_1)\sub [-1,1]$ for each $i$, and so $\pi_i(CP_1)\sub [-C,C]$ follows from the one-dimensional case.
\end{proof}
The previous proposition has the following corollary.
\begin{cor}\label{cor:transitivity_dilation}
    Let $P_1,P_2,P_3\sub \R^n$ be parallelograms, and $C_1,C_2\ge 1$ be constants. If $P_2\sub C_1 P_1$ and $P_3\sub C_2 P_2$, then $P_3\sub (C_1C_2)P_1$.
\end{cor}

\subsection{Facts about polynomials}

The following facts about polynomials will be used extensively, and are the key to many of our analysis of polynomials.
\begin{prop}\label{prop:polycoeff}
For any $k$-variate real polynomial $P=\sum_\alpha c_\alpha x^\alpha$ of degree at most $d$, we have the following relation:
$$
\sup_{ [-1,1]^k}|P|\sim_{k,d}\max_\alpha|c_\alpha|.
$$
As a result, we also have the following relations:
\begin{align*}
&\sup_{[-1,1]^k}|P|\sim_{k,d} \sum_\alpha|c_\alpha|\\
&\sup_{ [-C,C]^k}|P|\sim_{k,d} \sup_{ B^k(0,C)}|P|\lesssim_{k,d}C^d \sup_{[-1,1]^k}|P|,
\end{align*}
for any constant $C\ge 1$.
\end{prop}
For a proof of this proposition, the reader may consult \cite{Kellogg1928}.

\begin{cor}\label{cor:polycoeff}
For any parallelogram $T \subset \R^k$ and any $\mu\ge 1$,
$$
\sup_{T} |P| \le  \sup_{\mu T} |P|\lesssim_{k,d}\mu^d\sup_{T} |P|.
$$
\end{cor}
The proof follows from applying Proposition \ref{prop:polycoeff} to $P \circ \lambda_{T}$.

Another corollary of Proposition \ref{prop:polycoeff} is that a polynomial can be identified with its coefficient vector.
\begin{cor}\label{cor:poly_norm}
    Let $\mathcal{P}_{k,d}^0$ denote the normed space consisting of all $k$-variate real polynomials $P$ of degree at most $d$, equipped with the norm $\|P\| : = \sup_{[-1,1]^k}|P|$. Let $\Lambda_{k,d}: \mathcal{P}_{k,d}^0 \mapsto \R^{\binom{k+d}{k}}$ be a map defined by $\Lambda_{k,d}(\sum_{|\alpha|\leq d} c_\alpha x^\alpha) = (c_\alpha)_{|\alpha|\leq d}$. Then $\Lambda_{k,d}$ is a linear isomorphism between the normed spaces $(\mathcal{P}_{k,d}^0,\|\cdot \|)$ and $(\R^{\binom{k+d}{k}}, \|\cdot\|_{\infty})$, with the two norms comparable up to a factor depending only on $k,d$.
\end{cor}

\subsection{Flatness of graphs}\label{sec:flatness}

In this subsection, we discuss the relation between various formulations of flatness of graphs. We first introduce a few other notions of flatness; in the end, we will discuss their relationships with Definition \ref{defn:graphically_flat}.

\subsubsection{Alternative definitions of flatness}
\begin{defn}\label{defn:new_flatness}
Let $1\le k\le m\le n-1$, and $\phi:[-1,1]^k\to \R^{n-k}$ be a $C^2$ function. Let $\omega\sub [-1,1]^k$ be a parallelogram and let $\delta>0$. We say that $\omega$ is $(\phi,\delta)$-flat (or $\phi$ is $\delta$-flat over $\omega$) in dimension $m$
\begin{enumerate}
    \item in the first sense, or $F1(\delta)$ holds, if there exist an affine transformation $L: \R^{k}\to \R^{n-k}$, and a matrix $U': \R^{n-k} \to \R^{n-m}$ with orthonormal rows such that 
    \begin{equation}\label{eqn:F1}
        \sup_{x\in \omega}|U'(\phi(x) - Lx)|\leq \delta;
    \end{equation}
    \item in the second sense, or $F2(\delta)$ holds, if there exists a matrix $U': \R^{n-k} \to \R^{n-m}$ with orthonormal rows, such that (here $\nabla\phi\in \R^{(n-k)\times k}$)
    \begin{equation}\label{eqn:F2}
        \sup_{x,y\in \omega}|U'(\phi(y)-\phi(x)-\nabla \phi(x) (y-x))|\leq \delta;
    \end{equation}
    
    \item in the third sense, or $F3(\delta)$ holds, if there exist a matrix $U': \R^{n-k} \to \R^{n-m}$ with orthonormal rows, such that
    \begin{equation}\label{eqn:F3}
        \sup_{\substack{x \in \omega, v\in \mathbb S^{k-1}  \\ t\in \R: x+t v\in \omega}} \left |U'(v^T D^2 \phi(x) v) \right| |t|^2 \leq \delta.
    \end{equation}
    Here, $v^T D^2 \phi(x) v \in \R^{n-k}$ is defined by $(v^TD^2 \phi_1(x)v , \dots, v^T D^2 \phi_{n-k}(x)v)^T$. Intuitively, \eqref{eqn:F3} means that the quadratic term in the Taylor expansion of each $\phi_i$ along any line segment contained in $\omega$ does not exceed $\delta$.
\end{enumerate}
\end{defn}
We record a useful observation here, which follows directly from matrix algebra:
\begin{equation}\label{eqn:Mar_1}
\begin{aligned}
    U'(v^T D^2 \phi(x) v)
    &=v^T D^2 (U'\phi)(x) v\\
    &:=\big(v^T D^2 (U'\phi)_1 (x) v,\dots, v^T D^2 (U'\phi)_{n-m} (x) v\big)^T.
\end{aligned}
\end{equation}
{\it Remarks.} 
\begin{itemize}
 \item In the special case of $k=m=n-1$, that is, the case of hypersurfaces, we have $U'\in O(1)=\{1,-1\}$. Without loss of generality, we take $U'=1$. In this case, the formulation \eqref{eqn:F2} is exactly the one used in 
\cite{LiYang} and \cite{LiYang2023}. The entire Section \ref{sec:radial} is in this setting.

\noindent As a side note, when $k=m=1$, $n=2$, \eqref{eqn:F3} is just \cite[Equation (1.2)]{Yang2}.

\item For many interesting cases in consideration, these three senses of flatness are equivalent up to harmless constants. See Propositions \ref{prop:F1F2F3_equivalent} and \ref{prop:squares_flat}.

\end{itemize}



\begin{prop}\label{prop:new_flatness_affine_invariance}
    All three senses of flatness for a function $\phi$ are invariant under affine rescaling and adding an affine component to $\phi$. More precisely, let $\Xi:\R^k\to \R^k$ be an affine bijection and let $\Lambda:\R^k\to \R^{n-k}$ be an affine transformation. Let $\phi:\omega\to \R^{n-k}$ be $C^2$, and let $\psi:\Xi^{-1}(\omega)\to \R^{n-k}$ be defined by
    \begin{equation*}
        \psi(x):=\phi(\Xi x)+L(x).
    \end{equation*}
    Then $\omega$ is $(\phi,\delta)$-flat if and only if $\Xi(\omega)$ is $(\psi,\delta)$-flat, in any of the three senses above.
\end{prop}
The proof is elementary, and similar to \cite[Section 9.2]{LiYang}.

\begin{prop}\label{prop:F1F2F3_equivalent}
Fix $\phi,\omega,\delta$. We have $F3(\delta)\implies F2(C\delta)\implies F1(C\delta)$, for $C$ depending only on $n$. Conversely, if each entry of $\phi$ is a polynomial of degree at most $d$, then we also have $F1(\delta)\implies F3(C_{n,d}\delta)$, where $C_{n,d}$ is a constant depending only on $n,d$.  
\end{prop}

\begin{proof}
    $F3(\delta)\implies F2(C\delta)\implies F1(C\delta)$ is trivial. For the implication $F1(\delta)\implies F3(C_{n,d}\delta)$, by Proposition \ref{prop:new_flatness_affine_invariance}, we may assume $\omega=[-1,1]^k$. By the assumption of $F1(\delta)$, we have
    \begin{equation}
        \sup_{x\in [-1,1]^k}|U'(\phi(x)-Lx)|\le \delta,
    \end{equation}
    for some $U'$ with orthonormal rows. Since each entry of $\phi$ is a polynomial of degree at most $d$, so is each entry of $U'(\phi(x)-Lx)$. Using Proposition \ref{prop:polycoeff}, we see that all coefficients of each entry of $U'(\phi(x)-Lx)$ are $O_{n,d}(\delta)$, whence each entry of $v^T D^2(U'\phi(x))v$ is $O_{n,d}(\delta)$, for every $v\in \mathbb S^{k-1}$. Thus, by \eqref{eqn:Mar_1}, we have
    \begin{equation*}
        \sup_{x\in [-1,1]^k}|U'(v^T D^2 \phi(x)v)|=\sup_{x\in [-1,1]^k}|v^T D^2 (U' \phi)(x)v|\lesssim_{n,d} \delta.
    \end{equation*} 
\end{proof}

The following proposition is related to the flatness of parallelograms in the case of hypersurfaces.
\begin{prop}\label{prop:squares_flat}
Let $1\le k=m=n-1$ and $c\in (0,1]$. Then the following holds, where all implicit constants depend only on $c,n$.
\begin{enumerate}
    \item If $\norm{\phi}_{C^2}\le 1$ and all eigenvalues of  $D^2\phi$ are $\ge c>0$, then a subset $\omega\sub \R^{n-1}$ is $O(\delta)$-flat in any of the three senses in Definition \ref{defn:new_flatness} if and only if the diameter of $\omega$ is $\lesssim \delta^{1/2}$. Thus, a maximally $(\phi,\delta)$-flat subset must be a square of side length $\sim \delta^{1/2}$.
    \item If $\norm{\phi}_{C^2}\le 1$ and $|\det D^2\phi|\ge c>0$, and if $\omega\sub \R^{n-1}$ is a square, then $\omega$ is $O(\delta)$-flat in any of the three senses in Definition \ref{defn:new_flatness} if and only if the side length of $\omega$ is $\lesssim\delta^{1/2}$. Thus, a maximally $(\phi,\delta)$-flat square must have side length $\sim \delta^{1/2}$.
\end{enumerate}    
\end{prop}
We leave the elementary proof to the reader as an exercise.

The following lemma demonstrates a good property that is only enjoyed by flatness in the third sense.
\begin{lem}\label{lem:C2flat_entire}
    Let $1\le k\le m\le n-1$, and let $\phi:[-1,1]^k\to \R^{n-k}$ with $\|\phi\|_{C^2} \leq 1$. Then the following statements are equivalent for $\delta\in (0,1]$:
    \begin{enumerate}
        \item \label{item:Mar_6_01} The unit cube $[-1,1]^k$ is $(\phi,O(\delta))$-flat in dimension $m$ in the third sense.
        \item \label{item:Mar_6_02} There exist a function $\psi:[-1,1]^k\to \R^{n-k}$ with $\|\psi\|_{C^2} \leq 1$, an affine transformation $L:\R^k\to \R^{n-k}$ with norm bounded by $O(1)$, and an orthonormal matrix $U : \R^{n-k} \to \R^{n-k}$, such that (see \eqref{eqn:vector:entrywise_product} for the notation)
    \begin{equation}\label{eqn:Mar_1_01}
        U\phi = \vec \delta \psi+L,
    \end{equation}
    where $\vec \delta = (\delta_1,\dots,\delta_{n-k})$ and $\delta_i \sim \delta$ when $1\leq i\leq n-m$.
    \item \label{item:Mar_6_03} The same statement of Part \eqref{item:Mar_6_02} holds, and in addition, $\min_i\delta_i\sim \delta$, $\max_i \delta_i\lesssim 1$.
    \end{enumerate} 
    All implicit constants here and the proof below depend only on $n$.
\end{lem}

\begin{proof}
``\eqref{item:Mar_6_03} $\implies $\eqref{item:Mar_6_02}" This part is trivial.

``\eqref{item:Mar_6_02} $\implies $\eqref{item:Mar_6_01}" Given $U$, we take $U'=[I_{n-m}|O_{m-k}]U\in \R^{(n-m)\times (n-k)}$, so by \eqref{eqn:Mar_1},
\begin{align*}
    U'(v^T D^2 \phi(x) v)&=[I_{n-m}|O_{m-k}]U(v^T D^2 \phi(x)v)\\
    &=[I_{n-m}|O_{m-k}](v^T D^2 (U\phi)(x))\\
    &=[I_{n-m}|O_{m-k}](v^T D^2 (\vec\delta\psi)(x))\\
    &=(\delta_1 v^T D^2 \psi_1(x)v, \dots, \delta_{n-m}v^T D^2 \psi_{n-m}(x)v)^T.
\end{align*}
But since $\delta_i\sim\delta$ for $1\le i\le n-m$, and $\norm{\psi}_{C^2}\le 1$, the result follows.

``\eqref{item:Mar_6_01} $\implies $\eqref{item:Mar_6_03}" Given $U'$, we let $U \in O(n-k)$ such that $U'$ equals the first $n-m$ rows of $U$. Then \eqref{eqn:F3} and \eqref{eqn:Mar_1} imply that $\norm{D^2 (U'\phi)_i}_\infty\lesssim \delta$ for every $1\le i\le n-m$. But by definition of $U'$, this means that $\norm{D^2 (U\phi)_i}_\infty\lesssim \delta$ for every $1\le i\le n-m$. Denote by $L_i:\R^k\to \R$ the linear approximation of $(U\phi)_i$ at $0$, and define $\psi_i:=(C\delta)^{-1}((U\phi)_i-L_i)$ for a suitable constant $C\sim 1$, such that $\norm{\psi_i}_{C^2}\le 1$ for $1\le i\le n-m$. We finally define
\begin{equation*}
\begin{aligned}
    &\psi=(\psi_1,\dots,\psi_{n-m},(U\phi)_{n-m+1},\dots,(U\phi)_{n-k}),\\
    &L=(L_i,\dots,L_{n-m},0,\dots,0),\\
    & \vec \delta=(C\delta, \dots,C\delta,1,\dots, 1),
\end{aligned}
\end{equation*}
such that $\norm{D^2\psi}_\infty\le 1$, $\norm{L}\lesssim 1$, and that \eqref{eqn:Mar_1_01} holds.

\end{proof}

\subsubsection{Relation with Definition \ref{defn:graphically_flat}}
We now discuss the relationship between flatness in Definition \ref{defn:graphically_flat} with flatness in Definition \ref{defn:new_flatness}.

\begin{lem}\label{lem:flat_lemma1}
Let $C_0\ge 1$ and $C_1\ge 1$. Then there exists some $c_0$ depending on $n,C_0,C_1$, such that for every $\delta\in (0,c_0]$ the following holds:

Let $A_1\in \R^{(n-m)\times k}$ and $A_2\in \R^{(n-m)\times (n-k)}$ be arbitrary real matrices whose entries are bounded by $C_0$, and let $b\in \R^{n-m}$. Let $\omega\sub \R^{k}$ be a parallelogram centred at $0$, and assume that all its $k$ dimensions are at least $C_1\delta$. Let $\phi:\omega\to \R^{n-k}$ be a $C^2$ function that satisfies $\norm{\phi}_{C^2}\le C_0$, $\phi(0)=0$ and $\nabla\phi(0)=0$. If
\begin{equation}\label{eqn:flatness_lemma1}
    \sup_{x\in \omega}|A_1x+A_2\phi(x)+b|\le \delta,
\end{equation}
then $\norm{A_1}\lesssim_n C_1^{-1}$.
\end{lem}
\begin{proof}

    Taking $x=0$ in \eqref{eqn:flatness_lemma1} gives $|b|\le \delta$, and so 
    \begin{equation}
        \sup_{x\in \omega}|A_1x+A_2\phi(x)|\le 2\delta.
    \end{equation}    
    By the definition of $\norm{A_1}$, there exists some $x\in \omega$ such that $|A_1x|=\norm {A_1} |x|$. Also, since $\omega$ is centred at $0$, and $\omega$ has width at least $C_1 \delta$, we may take some $x\in \omega$ with $|x|\sim_n C_1 \delta$. Since $\phi(0)=0$, $\nabla\phi(0)=0$, we have
    \begin{equation*}
        |A_2\phi(x)|\lesssim_n C_0^2|x|^2\lesssim_n C_0^2 C_1^2 \delta^2\le C_0^2 C_1^2 c_0 \delta,
    \end{equation*}
    which is less than $\delta$ if $c_0$ is chosen to be small enough, depending on $n,C_0,C_1$. Thus, we have $\norm {A_1} |x|=|A_1x|\le 3\delta$, which implies $\norm {A_1}\lesssim_n C_1^{-1}$.
\end{proof}

\begin{lem}\label{lem:Oct_21_elementary}
Let $U$ be a $p\times q$ matrix and $V$ be a $q\times r$ matrix, where $p\le \min\{q,r\}$. Let $a$, $b$ be the $p$-volume of the parallelograms formed by the rows of $U$, $UV$, respectively. Then $b\lesssim_{p,q,r} \norm{V}a$.
\end{lem}

\begin{proof}
    Denote by $L:\R^q\to \R^r$ the mapping $Lu:=V^Tu$, where we regard $u$ as a column vector. Then $L$ is Lipschitz with constant $\sim \norm V=\norm {V^T}$. Thus, for every measurable subset $S\sub \R^q$, we have $\mathcal L^r(LS)\lesssim \norm V \mathcal L^q(S)$. In particular, let $S$ be the parallelogram formed by the column vectors of $U^T$, so $\mathcal L^q(S)=a$. Also, $\mathcal L^r(LS)=b$. Thus, the result follows. 
\end{proof}

Now we return to the discussion of flatness in Definition \ref{defn:graphically_flat}. 

Let $C_0\ge 1$ be fixed, and let $C_1$ be a large enough constant that depends on $n$ but not $C_0$. With this, take $c_0$ as in Lemma \ref{lem:flat_lemma1} depending on $n,C_0,C_1$.

\begin{prop}\label{prop:equivalence_flatness_general}
Let $\omega\sub [-1,1]^k$ be a parallelogram with all its $k$ dimensions bounded below by $C_1 \delta$, and $\delta\in (0,c_0]$. Let $\phi:\omega\to \R^{n-k}$ be a $C^2$ function such that $\norm{\phi}_{C^2}\le C_0$. 

Assume $\omega$ is $(\phi,\delta)$-flat in dimension $m$ as in Definition \ref{defn:graphically_flat}, namely, there exist a matrix $A=[A_1|A_2]$ with orthonormal rows, where $A_2\in \R^{(n-m)\times (n-k)}$, and a vector $b\in \R^{n-m}$, such that
\begin{equation}\label{eqn:Feb_26}
    \sup_{x\in \omega}|A_1 x+A_2\phi(x)+b|\le \delta.
\end{equation}
Then we have the following conclusions.
\begin{enumerate}
    \item \label{item:Feb_26_01} $k\le m$.
    \item \label{item:Feb_26_02} The rows of $A_2$ form a parallelogram with $(n-m)$-volume $\sim 1$. 
    \item \label{item:Feb_26_03} Let $\Xi:\R^k\to \R^k$ be an affine bijection given by $x\mapsto Ux+u$ where $\norm{U}\le C_0$. Let $\Lambda:\R^k\to \R^{n-k}$ be an affine transformation given by $x\mapsto Vx+v$ where $\norm{V}\le C_0$. Let
    \begin{equation*}
        \psi(y):=\phi(\Xi y)+\Lambda y,\quad  y\in \Xi^{-1}(\omega).
    \end{equation*}
    Then $\Xi^{-1}(\omega)$ is $(\psi,O(\delta))$-flat in dimension $m$. 
    \item \label{item:Feb_26_04} If we additionally have $k=m$, then
\begin{equation}\label{eqn:Oct_21_01}
    \sup_{x\in \omega}|\phi(x)-Bx-b'|\lesssim \delta,
\end{equation}
for some $B\in \R^{(n-m)\times m}$ and $b'\in \R^{n-m}$, where $B$ has $O(1)$ entries. 
\end{enumerate}
All implicit constants in this proposition and its proof below in turn depend only on $n,C_0$.

\end{prop}

Note that the implication from \eqref{eqn:Oct_21_01} to Definition \ref{defn:graphically_flat} is essentially trivial (if on the right hand side of \eqref{eqn:Oct_22}, the $\le \delta$ is replaced by $\lesssim_n \delta$). Proposition \ref{prop:equivalence_flatness_general} provides the nontrivial converse when $\delta$ is small enough. (It is trivial when $\delta\sim_{n,C_0} 1$.) 

\begin{proof}
    Given a parallelogram $\omega$ with centre $c$, we denote $\omega_0:=\omega-c$, and denote
\begin{equation*}
    \tilde \phi(x):=\phi(x+c)-\phi(c)-\nabla \phi(c) x,
\end{equation*}
so that $\|\tilde \phi\|_{C^2}\lesssim_n C_0$, $\tilde \phi(0)=0$ and $\nabla \tilde \phi(0)=0$. (Here $x,c$ are assumed to be column vectors in $\R^k$, so $\nabla\phi(c)\in \R^{(n-k)\times k}$.) Thus, \eqref{eqn:Feb_26} becomes
\begin{equation*}
    \sup_{x\in \omega_0}|(A_1+A_2\nabla\phi(c))x+A_2\tilde \phi(x)+A_2\phi(c)+b|\le\delta.
\end{equation*}
Since the rows of $[A_1\,|\,A_2]$ are orthonormal, in particular, the matrices $A_1+A_2\nabla \phi(c)$ and $A_2$ have entries $\lesssim C_0$. We may thus apply Lemma \ref{lem:flat_lemma1} with $\tilde \phi,\omega_0,O(C_0)$ to obtain
\begin{equation*}
    \norm{A_1+A_2\nabla\phi(c)}\lesssim C_1^{-1}.
\end{equation*}
Denote
\begin{equation*}
    E:=[-A_2\nabla\phi(c)\,|\,A_2]=A_2[-\nabla\phi(c)\,|\,I_{n-k}],
\end{equation*}
and so $[A_1\,|\,A_2]=E+O (C_1^{-1})$. But since the rows of $[A_1\,|\,A_2]$ are orthonormal, if $C_1$ is chosen to be large enough, then the rows of $E$ form a parallelogram of $(n-m)$-volume $\sim 1$.

In particular, $\rank (E)=n-m$. Since $\rank([-\nabla \phi(c)\,|\,I_{n-k}])=n-k$, by the elementary formula $\rank(UV)\le \min\{\rank(U),\rank(V)\}$ for arbitrary matrices $U,V$ with suitable dimensions, we have $n-m\le n-k$, i.e., $k\le m$. This gives Part \eqref{item:Feb_26_01}.

To prove Part \eqref{item:Feb_26_02}, we note that since $[-\nabla\phi(c)\,|\,I_{n-k}]$ has entries bounded by $O(C_0)$, by Lemma \ref{lem:Oct_21_elementary}, the rows of $A_2$ form a parallelogram of $(n-m)$-volume $\sim 1$.

To prove Part \eqref{item:Feb_26_03}, we apply \eqref{eqn:Feb_26} to $x=\Xi y$ to get
\begin{equation*}
    \sup_{y\in \Xi^{-1}(\omega)}|(A_1 U-A_2 V)y+A_2\psi(y)+b+A_1 u-A_2 v|\le \delta.
\end{equation*}
Since $\norm{A_1 U-A_2 V}\lesssim C_0$ and $\norm{A_2}\le 1$, it suffices to show that the rows of the matrix 
\begin{equation*}
    F:=[A_1'\,|\,A_2]:=[A_1U-A_2 V\,|\,A_2]
\end{equation*}
form a parallelogram of $(n-m)$-volume $V_F\sim 1$. To see this, we compute that
\begin{equation*}
    V_F^2=\det (FF^T)=\det \big(A_1'(A_1')^T +A_2A_2^T\big)\ge \det \big(A_2A_2^T\big),
\end{equation*}
where the last inequality follows from an elementary result in linear algebra, namely, for symmetric positive semidefinite matrices $P,Q$ of the same size, we always have $\det(P+Q)\ge \det(P)$. But by Part \eqref{item:Feb_26_02}, we have $\det (A_2A_2^T)\sim 1$, and so the result follows.

For Part \eqref{item:Feb_26_04}, if in addition $k=m$, then $A_2$ is a square matrix whose determinant has absolute value $\gtrsim 1$. Multiplying \eqref{eqn:Feb_26} by $A_2^{-1}$, we arrive at \eqref{eqn:Oct_21_01}.

\end{proof}



\subsection{Various formulations of decoupling}
We state and prove a technical lemma in the following.

\begin{lem}\label{lem:decoupling_entire_strip}
    Let $\phi:[-1,1]^k\to \R^{n-k}$ be $C^2$, and $\Omega\sub [-1,1]^k$ be a parallelogram. For $\delta>0$, $\Omega$ can be $\phi$-decoupled into parallelograms $\omega$ at scale $\delta$ if and only if 
    $N^\phi_\delta(\Omega)$ can be decoupled into the parallelograms $\omega\times \R^{n-k}$ at the same cost.   
    
\end{lem}
\begin{proof}
By our formulation in Definitions \ref{defn:decoupling} and \ref{defn:graphical_decoupling_convex_hull}, it suffices to show that
    \begin{equation*}
        N^\phi_\delta(\Omega)\cap \mathrm{Co}({N^\phi_\delta(\omega)})=N^\phi_\delta(\Omega)\cap (\omega\times \R^{n-k}).
    \end{equation*}
    It suffices to show $\supseteq$. But $N^\phi_\delta(\Omega)\cap (\omega\times \R^{n-k})=N^\phi_\delta(\omega\cap \Omega)$, which is a subset of both $N^\phi_\delta(\Omega)$ and $\mathrm{Co}({N^\phi_\delta(\omega)})$. Thus, the result follows.

\end{proof}

\subsection{Combining decoupling inequalities}
We provide a proof of how we can combine different $\ell^q(L^p)$ decoupling inequalities at different levels.
\begin{prop}\label{prop:combine_decoupling}
    Let $\mathcal T$ be a $B$-overlapping cover of a parallelogram $P_0\sub \R^n$ by parallelograms $T$. Let $C_2\ge C_1\ge 1$. For each $T$, let $\mathcal S(T)$ be a $B'$-overlapping cover of $C_1 T$ by parallelograms $S\sub C_2 T$. Denote $\mathcal S=\sqcup_{T\in \mathcal T}\mathcal S(T)$, where we assume that $\mathcal S(T)$ is disjoint with $\mathcal S(T')$ for $T\ne T'$ (as \underline{collections} of parallelograms). Then we have
    \begin{enumerate}
        \item \label{item_01_Nov_4}$\mathcal S$ is a $B''$-overlapping cover of $P_0$, where
        \begin{equation*}
            B''(\mu):=B(C_2\mu)B'(\mu).
        \end{equation*}
        \item \label{item_02_Nov_4} For all $p,q,\alpha$ we have (recall Definition \ref{defn:decoupling}).
    \begin{equation*}
        \mathrm{Dec}(P_0,\mathcal S,p,q,\alpha)\le 2^{\alpha+\frac 1 2-\frac 1 q} \log_2 (\# \mathcal S)\mathrm{Dec} (P_0,\mathcal T,p,q,\alpha)\sup_{T\in \mathcal T}\mathrm{Dec}(C_1T,\mathcal S(T),p,q,\alpha).
    \end{equation*}
    \end{enumerate}
\end{prop}
\begin{proof}
We prove Part \eqref{item_01_Nov_4} first. Let $\mu\ge 1$ and $x\in \R^n$. If $x\in \mu S$, then we must have $x\in C_2\mu T$, by $S\sub CT$ and Corollary \ref{cor:transitivity_dilation}. But then the parallelograms $\{C_2\mu T:T\in \mathcal T\}$ have $B(C_2\mu)$-overlap, so this means that $x$ lies in at most $B(C_2\mu)$ many $C_2\mu T$'s. Thus, we have
\begin{align*}
    \sum_{S\in \mathcal S}1_{\mu S}(x)
    &=\sum_{T\in \mathcal T}\sum_{S\in \mathcal S(T)}1_{\mu S}(x)\\
    &=\sum_{T\in \mathcal T}\sum_{S\in \mathcal S(T)}1_{\mu S\cap C_2\mu T}(x)\\
    &\le B(C_2\mu)\sum_{S\in \mathcal S(T)}1_{\mu S\cap C_2\mu T}(x)\\
    &\le B(C_2\mu)\sum_{S\in \mathcal S(T)}1_{\mu S}(x)\\
    &\le B(C_2\mu)B'(\mu).
\end{align*}
We then prove Part \eqref{item_02_Nov_4}. Denote $\alpha_0=\alpha+\frac 1 2-\frac 1 q$. For each $k\ge 1$, denote
    \begin{equation*}
        \mathcal T_k:=\{T\in \mathcal T:\# \mathcal S(T)\in [2^{k-1},2^k)\}.
    \end{equation*}
    Then there are at most $\log_2 (\#\mathcal S)$ many such $k$'s. Given test functions $f_S$ each Fourier supported on $S$, and denote $f_T:=\sum_{S\in\mathcal S(T)}f_S$, $f_k:=\sum_{T\in \mathcal T_k}f_T$, and $f:=\sum_{k }f_k$. Then we have
    \begin{align*}
        &\norm f_p\\
        &\le \sum_{k}\norm{f_k}_p\\
        &\le \sum_k \mathrm{Dec} (P_0,\mathcal T,p,q,\alpha)(\#\mathcal T_k)^{\alpha_0} \norm{\norm{f_T}_p}_{\ell^q(T\in \mathcal T_k)}\\
        &\le \sum_k \mathrm{Dec} (P_0,\mathcal T,p,q,\alpha)(\#\mathcal T_k)^{\alpha_0} 2^{k\alpha_0}\sup_{T\in \mathcal T} \mathrm{Dec} (C_1T,\mathcal S(T),p,q,\alpha)\norm{\norm{f_S}_p}_{\ell^q(S\in \mathcal S)}\\
        &\le 2^{\alpha_0}\sum_k \mathrm{Dec} (P_0,\mathcal T,p,q,\alpha) (\#\mathcal S)^{\alpha_0} \sup_{T\in \mathcal T} \mathrm{Dec} (C_1T,\mathcal S(T),p,q,\alpha)\norm{\norm{f_S}_p}_{\ell^q(S\in \mathcal S)}\\
        &\le 2^{\alpha_0}\log_2 (\#\mathcal S)\mathrm{Dec} (P_0,\mathcal T,p,q,\alpha) (\#\mathcal S)^{\alpha_0} \sup_{T\in \mathcal T} \mathrm{Dec} (C_1T,\mathcal S(T),p,q,\alpha)\norm{\norm{f_S}_p}_{\ell^q(S\in \mathcal S)}.
    \end{align*}
\end{proof}

\subsection{Nested meshes}

Since proofs of both principles will be based on iterative constructions of covering by parallelograms, we need detailed analysis of the overlaps between the covering parallelograms. 

\begin{defn}[Nested mesh]\label{defn:dyadic_mesh}
Let $R_0\sub \R^n$ be a parallelogram. For $1\le i\le N$, let $B_i$ be overlap functions as in Definition \ref{defn:enlarged_overlap}, and $C_{2}\ge C_{1}\ge 1$ be real numbers.

A sequence $\mathcal R_i$, $1\le i\le N$ of collections of parallelograms of $\R^n$ is called a $(B_i,C_{1},C_{2})$-nested mesh covering $R_0$, if the following is true. 
\begin{enumerate}
    \item \label{item_Nov_2_01} The collection $\mathcal R_1=\mathcal R_1(R_0)$ covers $R_0$, and is $B_1$-overlapping as in Definition \ref{defn:enlarged_overlap}.
    \item \label{item_Nov_2_02} For each $1\le i \le N-1$, each $R_i\in \mathcal R_i$, there exists a subcollection $\mathcal R_{i+1}(R_i)\sub \mathcal R_{i+1}$ covering $C_{1} R_i$, such that $\cup \mathcal R_{i+1}(R_i)\sub C_{2} R_i$. Then we require that
    \begin{equation*}
    \mathcal R_{i+1}=\bigcup_{R_i\in \mathcal R_i}\mathcal R_{i+1}(R_i),
\end{equation*}
where we may assume the union is disjoint (as collection $\mathcal R$ of subsets $R$, not to be confused with disjointness of the subsets $R$). Also, $\mathcal R_{i+1}$ is $B_{i+1}$-overlapping.
\end{enumerate}
\end{defn}

In practice, it is often easier to control the overlap function of the subcollections $\mathcal R_{i+1}(R_i)$ first. Then we have the control of the overlap function of $\mathcal R_{i+1}$ via the following proposition.
\begin{prop}\label{prop:iterative_overlap}
Fix $1\le i\le N-1$. For each $R_{i}\in \mathcal R_i$ and each subcollection $\mathcal R_{i+1}(R_i)$ covering $C_{1}R_i$ and contained in $C_{2}R_i$, denote by $B_{i,R_i}$ the overlap function of $\mathcal R_{i+1}(R_i)$. Define $\overline B_{i+1}:=\sup_{R_{i}\in \mathcal R_{i} }B_{i,R_i}$ (and $\overline B_1:=B_1$). Then the overlap function $B_{i+1}$ for $\mathcal R_{i+1}$ can be chosen to be
\begin{equation*}
    B_{i+1}(\mu):=\prod_{i'=1}^{i} \overline B_{i'}\left(\prod_{i''=i'}^{i}C_{2}\mu\right).
\end{equation*}

\end{prop}
\begin{proof}
    This follows from iterative applications of Part \eqref{item_01_Nov_4} Proposition \ref{prop:combine_decoupling}.
\end{proof}
If we allow the overlap number to depend on $N$, we may simply take $B:=B_{N}$, such that every $\mathcal R_i$, $1\le i\le N$ is $B$-overlapping. 

In the canonical case where each $\mathcal R_{i+1}(R_{i})$, $0\le i\le N-1,$ is the tiling of $R_i$ by $2^n$ congruent parallelograms, we may choose $C_1=C_2=1$ and $B(\mu)=B_1(\mu)=\mu^n$. In general, to ensure that the collections $\mathcal R_i$ stay within $2R_0$, we have the following proposition.
\begin{prop}\label{prop:tiling_containedin_2R0}
    With the assumptions of Definition \ref{defn:dyadic_mesh}, if $\mathcal R_1$ is a tiling of $R_0$ by an $N_0\times \cdots N_0$-grid of congruent parallelograms, where $N_0\ge C_2^{N-1}$, then for every $R_1\in \mathcal R_1$ and $2\le i\le N$, we have $\cup \mathcal R_i(R_1)\sub C_2^{N-1}R_1$. In particular, for every $1\le i\le N$, we have $\cup\mathcal R_i\sub 2R_0$. 
\end{prop}
\begin{proof}
    The result is trivial for $i=1$. Fix one $R_1\in \mathcal R_1$. For $i=2$, by assumption, each $R_2\in \mathcal R_2(R_1)$ is contained in $C_2 R_1$. Similarly, each $R_3\in \mathcal R_3(R_2)$ is contained in $C_2 R_2$, which is in turn contained in $C_2^2 R_1$ by Corollary \ref{cor:transitivity_dilation}. In this fashion, we have that each $R_N\in \mathcal R_N$ is contained in $C_2^{N-1}R_1$, which is in turn contained in $2R_0$ if $N_0\ge C_2^{N-1}$.
\end{proof}

\subsection{On Corollary \ref{cor:radial}}\label{sec:corallary_1.4}
To rigorously prove Corollary \ref{cor:radial}, we need to show that the special case of Theorem \ref{thm:radial_principle} gives us exactly the same family $\mathcal S_\delta=\mathcal I_\delta$, under the condition that $r\in \mathcal P_{1,d}$. Note that $\#\mathcal I_\delta\lesssim \delta^{-1/2}$, since in one-dimension, every interval of length $\lesssim\delta^{1/2}$ must be $(r,O(\delta))$-flat. (All implicit constants in the previous sentence depend only on $\norm {r''}_{\infty}$.) 

Thus, we will need to check the condition (*) holds in Theorem \ref{thm:radial_principle}. The part about $T_\sigma$ is trivial, so it remains to prove the part about $S_\sigma$.

We first prove the following lemma, which establishes \eqref{eqn:Nov_21} with $C_r$ depending on $d$ only.
\begin{lem}
    Let $P\in \mathcal P_{1,d}$. Let $0<2\delta<\sigma\le 1$. Let $I_\delta\sub [1,2]$ be a $(P,\delta)$-flat interval, and let $I_\sigma\sub [1,2]$ be an interval such that $2I_\sigma\cap [1,2]$ is not $(P,\sigma)$-flat. If $I_\delta\cap I_\sigma\ne \varnothing$, then $I_\delta\sub C_d I_\sigma$, where $C_d$ depends on $d$ only.
\end{lem}
\begin{proof}[Proof of lemma]
    Since $I_\delta\cap I_\sigma\ne \varnothing$, if $l(I_\delta)\lesssim_d l(I_\sigma)$ then we are done. If not, then there exists a tiny $c=c_d$ such that $(1+c)I_\delta\supseteq 2I_\sigma$.

    Since $2I_\sigma$ is not $(P,\sigma)$-flat, neither is $(1+c)I_\delta$. But $I_\delta$ is $(P,\delta)$-flat, so $(1+c)I_\delta$ is $(P,O(1+c)^d)$-flat by Proposition \ref{prop:polycoeff}. But for $c$ small enough, we have $(1+c)^d\ll 2$, and this contradicts $\sigma>2\delta$. Thus, $I_\delta\sub C_d I_\sigma$.
\end{proof}
It remains to check that for the sum surface $\phi(s,t):=r(s)+\sqrt{1-|t|^2/(10l)}$, its coordinate space can be $\phi$-decoupled into $I_\delta\times T$ at scale $\delta$, where $I_\delta\in \mathcal I_\delta$ and $T$ comes from a tiling of $[-1,1]^l$ by cubes of side length $\delta^{1/2}$, with the decoupling constant independent of the coefficients of $r$. 

Strictly speaking, for the $s$ coordinate, the statement of Theorem \ref{thm:degeneracy_locating_principle} only provides the existence of some family $\mathcal I'_\delta$ of intervals that $\phi$-decouples $[-1,1]$, but it does not imply that $\mathcal I'_\delta$ agrees with our given $\mathcal I_\delta$. To prove the latter stronger result, one may use an adaptation of the proof of \cite[Theorem 4.6]{Yang2} with $\phi(s)=r(s)$ replaced by $r(s)+\sqrt{1-|t|^2/(10l)}$, where the terminology ``sub-admissible partition for $\phi$ at scale $\delta$" is essentially the same as saying that each $2I_\delta$ is not $(\phi,\delta)$-flat here. We omit the details.

\subsection{Bourgain-Demeter with weaker curvature lower bound}
In this section, we give a proof of a variant of Bourgain-Demeter decoupling \cite[Theorem 1.1]{BD2015} when we only have a weak lower bound for the Gaussian curvature.
\begin{prop}\label{prop:Bourgain_Demeter_K-1_lowerbound}
Let $\phi:[-1,1]^{n-1}\to \R$ be $C^{2,\zeta}$ where $\zeta\in (0,1]$, and suppose that $\inf |\det D^2 \phi|\ge K^{-1}$ for some $K\ge 1$. For every $0<\delta<1$, denote by $\mathcal R_\delta$ a tiling of $[-1,1]^{n-1}$ by rectangles $T$ of side length $\delta^{1/2}$. Then for $2\le p\le \frac{2(n+1)}{n-1}$, $[-1,1]^{n-1}$ can be $\phi$-$\ell^p(L^p)$ decoupled into the $(\phi,O(\delta))$-flat cubes $\mathcal R_\delta$ at the cost of $C_\eps \delta^{-\eps}K^{O(1)}$ for every $\eps>0$. Moreover, if $D^2\phi$ is positive-semidefinite, then the above can be upgraded to $\ell^2(L^p)$ decoupling. Here, the constant $C_\eps$ depends only on the $C^{2,\zeta}$ norm of $\phi$ (as well as $\eps,n,p$).
\end{prop}
The key output of this proposition is the $K^{O(1)}$ loss required in Assumption \ref{item:nondegnerate_decoupling} of Theorem \ref{thm:degeneracy_locating_principle}, which is needed in the proof of Theorem \ref{thm:refine_IJ}. It remains open whether we can relax the condition to $\phi$ just being $C^2$.

\begin{proof}
For convenience we write $k=n-1$. Let 
\begin{equation}\label{eqn:eta_appendix}
    \eta:=c_0 K^{-\frac 1 {\zeta}},
\end{equation}
for some small enough $c_0$ to be determined. By losing $\eta^{-O(1)}=K^{O(1)}$, we may trivially decouple $[-1,1]^{k}$ into cubes $R_0$ of side length $\eta$. It then suffices to decouple each $R_0$. Without loss of generality, we just take $R_0=[0,\eta]^{k}$. By affine invariance, we may assume $\phi(0)=0$, $\nabla \phi(0)=0$.

    Write $\phi(x)=Q(x)+E(x)$, where $Q(x)$ is the quadratic form of $\phi$ at $0$, and $|E(x)|\lesssim |x|^{2+\zeta}$. By a rotation, we may assume $Q(x)=\sum_{i=1}^{k}\frac{a_i}2 x_1^2$, where
    \begin{equation}
        \sup_i |a_i|\lesssim 1,\quad p:=|a_1\cdots a_k|\ge K^{-1}.
    \end{equation}
    In particular, each $|a_i|\gtrsim K^{-1}$. We then do another trivial decoupling, partitioning for each $i$ the $i$-th coordinate of $R_0$ into intervals of length $l_i:=\eta\sqrt{p|a_i|^{-1}}$. The loss of this trivial decoupling is $K^{O(1)}$.

    Denote by $R_1$ one smaller rectangle obtained in this way, and denote by $c$ its centre. Denote $x':=(l_1x_1,\dots,l_kx_k)$, and define $\overline \phi=\overline \phi_{R_1}$ by
    \begin{equation*}
        \overline \phi(x)=\phi(x'+c)-\phi(c)-\nabla \phi(c)\cdot x',
    \end{equation*}
    so $\overline \phi(0)=0$, $\nabla\overline \phi(0)=0$, and we can compute
    \begin{align*}
        D^2 \overline \phi(x)
        &=\begin{bmatrix}
            \partial_{11}\phi l_1^2 & \cdots & \partial_{1k}\phi l_1l_k\\
            \vdots & \ddots & \vdots\\
             \partial_{1k}\phi l_1l_k & \cdots & \partial_{kk}\phi l_k^2
        \end{bmatrix}(x'+c)\\
        &=p\eta^2\begin{bmatrix}
            1 + \partial_{11}E  |a_1|^{-1} & \cdots & \partial_{1k}E  |a_1a_k|^{-1/2}\\
            \vdots & \ddots & \vdots\\
             \partial_{1k}E  |a_1a_k|^{-1/2} & \cdots & 1+\partial_{kk}E |a_k|^{-1}
        \end{bmatrix}(x'+c).        
    \end{align*}
    We then let $\tilde \phi(x)=p^{-1}\eta^{-2}\overline \phi(x)$, so that $\tilde \phi(0)=0$, $\nabla \tilde \phi(0)=0$, and 
    \begin{equation*}
        D^2 \tilde \phi(x)=\begin{bmatrix}
            1 + \partial_{11}E  |a_1|^{-1} & \cdots & \partial_{1k}E  |a_1a_k|^{-1/2}\\
            \vdots & \ddots & \vdots\\
             \partial_{1k}E  |a_1a_k|^{-1/2} & \cdots & 1+\partial_{kk}E |a_k|^{-1}
        \end{bmatrix}(x'+c).
    \end{equation*}
    Now we use the assumption that for each pair $i,j$,
    \begin{equation*}
        |\partial_{ij}E(x'+c)|\lesssim |x'+c|^{\zeta}\lesssim \eta^{\zeta},
    \end{equation*}
    so that 
    \begin{equation*}
        |a_ia_j|^{-1/2}|\partial_{ij}E(x'+c)|\lesssim \eta^{\zeta}K^{-1}\lesssim c_0\ll 1,
    \end{equation*}
    using \eqref{eqn:eta_appendix}. Thus, $D^2 \tilde \phi(x)$ is approximately an identity matrix, with $|\det D^2 \tilde \phi(x)|\sim 1$. Also, using $\tilde \phi(0)=0$, $\nabla \tilde \phi(0)=0$ we have $\|\tilde \phi\|_{C^2}\lesssim 1$. 

    We may now apply Bourgain-Demeter \cite{BD2015,BD2017} to $\phi$-decouple $[-1,1]^k$ into $(\tilde \phi, p^{-1}\eta^{-2}\delta)$-flat cubes of side length $p^{-1/2}\eta^{-1}\delta^{1/2}$. Reversing the rescaling from $x$ to $x'$, these cubes become rectangles oriented in some direction, and with all their dimensions bounded below by $\delta^{1/2}$. Lastly, we apply trivial decoupling for a third time to decouple those rectangles into cubes of side length $\delta^{1/2}$, with a loss of the form $K^{O(1)}$. We may assume the cubes are axis-parallel. This finishes the proof.
\end{proof}

\subsection{Refined uniform decoupling inequalities for bivariate polynomials}\label{sec:refined_IJ}

The following theorem is a refinement of \cite{LiYang2023}, which removes the dependence of the overlap function on $\delta$, and provides a lower bound for the dimensions of the decoupling parallelograms. (See Theorem 1.4 of \cite{LiYang2023} and Theorem 2.3 of \cite{Li22}.)

\begin{thm}[Refined uniform decoupling for bivariate polynomials]\label{thm:refine_IJ}
    Let $d\in \N$, $2\le p\le 4$, $0<\delta  \ll_{n,d}  1$, $0 < \varepsilon \ll 1$, and $\phi\in 
\mathcal P_{2,d}$. Then there exists a family of $(\phi,\delta)$-flat parallelograms $\mathcal R_\delta = \mathcal R_\delta(\phi,\varepsilon)$ such that the following holds: 
\begin{enumerate}
    \item $[-1,1]^2$ can be $\phi$-$\ell^p(L^p)$decoupled into $\mathcal{P}_\delta$ at a cost of $ O_\varepsilon( \delta^{-\eps})$; moreover, if $D^2 \phi$ is positive-semidefinite, then this $\ell^p(L^p)$ decoupling can be upgraded to $\ell^2(L^p)$ decoupling. 
    \item $\cup \mathcal R_\delta$ covers $[-1,1]^2$ and is contained in $[-2,2]^2$;
    \item The overlap function of $\mathcal R_\delta$ is $O(1)$; in particular, the cardinality of $\mathcal R_\delta$ is bounded by $\delta^{-O(1)}$.
    \item Each $P_\delta \in \mathcal R_\delta$ has width at least $\delta$.
\end{enumerate}
All implicit constants here may only depend on $d,n,p,\eps$.
\end{thm}

\begin{proof}
Let $\mathcal M$ be the collection of graphs of all polynomials $\phi\in \mathcal P_{2,d}$ over $[-1,1]^2$. Let $\mathfrak R=\{\mathcal R_0\}$ where $\mathcal R_0$ is the collection of all parallelograms contained in $[-1,1]^2$. Let $\mathcal A$ be the collection of all affine bijections on $\R^2$. Let $HM_\phi=\det D^2 \phi$.

By Theorem \ref{thm:degeneracy_locating_principle}, it suffices to check the following ingredients:
    \begin{itemize}
    \item $(\mathcal M,\mathfrak R)$ is $\mathcal A$-rescaling invariant: this follows from Proposition \ref{prop:polycoeff}.
    \item $(\mathcal M,\mathfrak R)$ is $\mathcal A$-compatible: this is trivial, as $\mathcal R_0$ is the collection of all parallelograms.
    \item $H$ is a $\mathfrak R$-regular degeneracy determinant: this follows from Corollary \ref{cor:hessian}.
    \item Trivial covering property: this is trivial, since every cube $R_0$ of side length $K^{-1}$ belongs to $\mathcal R$.
    \item Sublevel set decoupling: this follows from Theorem \ref{thm:2D_general_uniform_IJ} below. Note that it gives the required width lower bound.
    \item Totally degenerate decoupling: this follows from \cite[Theorem 1.4]{Yang2}. Note that it gives the required width lower bound, since it even gives a $\delta^{1/2}$ lower bound when decoupling at scale $\delta$. 
    \item Nondegenerate decoupling: this follows from Proposition \ref{prop:Bourgain_Demeter_K-1_lowerbound}. Note that it gives the required width lower bound, since it even gives a $\delta^{1/2}$ lower bound when decoupling at scale $\delta$. 
    \item Width lower bounds: the respective width lower bounds have been discussed right above. To check Condition (*), the lower dimensional decoupling follows from \cite[Theorem 1.4]{Yang2}. Also, by Proposition \ref{prop:constant_partition} below, we can slightly improve the width lower bound $\gtrsim \delta$ to $\ge \delta$ (or even $\gg_{d,n} \delta$), when $\delta\ll_{d,n}1$.
    \end{itemize}

\end{proof}
{\it Remark. }With the overlap bound independent of $\delta$, one drawback is that we cannot remove the dependence of the cover $\mathcal R_\delta$ on $\eps$, which we managed to do in \cite[Section 2]{LiYang2023}. However, we anticipate that the overlap bound being independent of $\delta$ is more important. See, for example, \cite[Section 5]{GMO24}.

The following theorem is a slightly refined version of \cite[Theorem 3.1]{LiYang2023} by adding a width lower bound. It can be deduced directly from \cite[Proposition 5.1]{Li22}.
\begin{thm}\label{thm:2D_general_uniform_IJ}
For each $d\geq 0$ and $\varepsilon>0$, there is a constant $C_{\varepsilon,d}$ such that the following holds. For any polynomial $\phi\in \mathcal P_{2,d}$, any $0<\delta<1$, there exists a cover $\mathcal R_\delta$ of the set
$$
\{(x,y)\in [-1,1]^2:|\phi(x,y)|\le\delta\}
$$
by rectangles $R$ such that the following holds:
\begin{enumerate}
    \item $|\phi|\lesssim \delta$ on each $R$.
    \item Each $R$ has width at least $\delta$, and the overlap function of  $\mathcal R_\delta$ depends only on $d$.
    \item For any $2\le p\le 6$, $ \{(x,y)\in [-1,1]^2:|\phi(x,y)|\le\delta\}$ can be $\ell^2(L^p)$ decoupled into $\mathcal R_\delta$ at cost $C_{\eps,d}\delta^{-\eps}$.
\end{enumerate}

\end{thm}

The following simple proposition shows that decoupling for polynomials is robust with respect to constants ($C_i$ below) depending only on $n,d$.
\begin{prop}\label{prop:constant_partition}
   For every $C_1$, there exists $C_2=C_2(C_1)$ such that the following holds. Let $\phi\in \mathcal P_{n,d}$, and assume that $R\sub [-1,1]^n$ is a parallelogram of width $w\le C_2^{-1}$. If $R$ is $(\phi,C_3 w)$-flat (in any sense of Definition \ref{defn:new_flatness}), then we can partition $R$ into $O_{C_1,C_2,C_3}(1)$ smaller parallelograms $R'$ of width $w$, such that each $R'$ is $(\phi,C_1^{-1}w)$-flat (in any sense of Definition \ref{defn:new_flatness}).
\end{prop}
This means that to decouple at scale $w\ll 1$ with width lower bound $w$, it suffices to decouple at scale $O(w)$ with width lower bound $w$ (or equivalently, decouple at scale $w$ with width lower bound $\gtrsim w$).

We prove a stronger lemma below. Applying Lemma \ref{lem:C2_flatness} with $k=m=n-1$ and Proposition \ref{prop:F1F2F3_equivalent} gives Proposition \ref{prop:constant_partition}.
\begin{lem}\label{lem:C2_flatness}
For every $C_1\geq 1$, there exists $C_2=C_2(C_1)$ such that the following holds. Suppose $\phi:[-1,1]^k\to \R^{n-k}$ satisfies $\norm{\phi}_{C^2}\le 1$. Assume that $R\sub [-1,1]^k$ is a parallelogram of width $w\le C_2^{-1}$. If $R$ is $(\phi,C_3 w)$-flat, then we can partition $R$ into $O_{C_1,C_2,C_3}(1)$ smaller parallelograms $R'$ of width $w$, such that each $R'$ is $(\phi,C_1^{-1}w)$-flat.

Here and in the proof below, by flatness we mean flatness in the \underline{third} sense as in \eqref{eqn:F3}, in some dimension $m\in [k,n-1]$.
\end{lem}

\begin{proof}
    Without loss of generality, write $R=\bar R\times [-w/2,w/2]$. Define
    \begin{equation}\label{eqn:Jul_3}
        \psi(x,y)=\phi\circ \lambda_{R}(x,y)=\phi(\lambda_{\bar R}x,w y).
    \end{equation}
Since $R$ is $(\phi,C_3w)$-flat in the third sense, $[-1,1]^k$ is $(\psi,C_3 w)$-flat in the third sense, by Proposition \ref{prop:new_flatness_affine_invariance}. But by Lemma \ref{lem:C2flat_entire}, there exist $\psi:[-1,1]^k\to [-1,1]^{n-k}$ with $\norm{\psi}_{C^2}\le 1$, an affine transformation $L:\R^k\to \R^{n-k}$ with norm bounded by $O_n(1)$, and an orthonormal matrix $U : \R^{n-k} \to \R^{n-k}$, such that    \begin{equation*}
        U\phi = C_3\vec w \psi+L,
    \end{equation*}
where $\vec w = (w_1,\dots,w_{n-k})=O(1)$ and $w_i \sim_n w\sim_n \min_i w_i$ when $1\leq i\leq n-m$. Without loss of generality, we assume $U=I_{n-k}$ and $L=0$, so $\phi=\vec w \psi$.

Now we partition $[-1,1]^k$ into parallelograms of the form $Q\times [-1,1]$, where $Q$ is a cube of side length $c\sim_{C_2,C_3,n} 1$ to be determined. Define
\begin{equation*}
    \eta(x,y)=\psi(\lambda_Q x,y).
\end{equation*}
By Proposition \ref{prop:new_flatness_affine_invariance} and Lemma \ref{lem:C2flat_entire} again,
it suffices to prove $\norm{\eta_i}_{C^2}\ll_{n} C_1^{-1}C_3^{-1}$, $1\le i\le n-m$.

Indeed, by \eqref{eqn:Jul_3}, we have $\norm{\partial_{yy} \eta_i}_\infty=O_{n}(w)$, which is $\ll_{n} C_1^{-1}C_3^{-1}$ if $C_2$ is small enough, since $w\le C_2^{-1}$.

Also, $\norm{\partial_{y}\nabla_x\eta_i}_\infty=O_{C_1,n}(c)\ll_{n} C_1^{-1}C_3^{-1}$, and  $\norm{D^2_x\eta_i}_\infty=O_{C_1,n}(c^2)\ll_{n} C_1^{-1}C_3^{-1}$. Then the result follows if $c$ is small enough. 
\end{proof}
\bibliographystyle{alpha}
\bibliography{reference}

\end{document}